\documentclass[reqno,tbtags,a4paper,12pt]{amsart}

\usepackage[in]{fullpage}

\usepackage{amsmath,amssymb,amsthm,amsfonts,amscd}

\newtheorem{theorem}{Theorem}[section]
\newtheorem{question}{Question}[section]
\newtheorem{corr}{Corollary}[section]
\newtheorem{propos}{Proposition}[section]
\newtheorem*{hypothesis}{Hypothesis}
\newtheorem*{cond}{Necessary condition}
\theoremstyle{definition}
\newtheorem{definition}{Definition}[section]
\newtheorem*{ax1}{Axiom I}
\newtheorem*{ax2}{Axiom II}
\newtheorem*{ax3}{Axiom III}
\newtheorem*{ax4}{Axiom IV}
\newtheorem*{ax5}{Axiom V}

\theoremstyle{remark}
\newtheorem{remark}{Remark}[section]
\newcommand{\V}{\mathrm{Vert}}

\newcommand{\cub}{\mathop{\rm cub}\nolimits}
\newcommand{\norm}{\mathrm{norm}}
\newcommand{\diff}{d}
\newcommand{\Lk}{\mathop{\rm link}\nolimits}
\newcommand{\codim}{\mathop{\rm codim}\nolimits}
\newcommand{\cone}{\mathop{\rm cone}\nolimits}

\newcommand{\bd}{\beta}

\newcommand{\Sign}{\mathop{\rm sign}\nolimits}
\newcommand{\Hom}{\mathop{\rm Hom}\nolimits}
\newcommand{\sing}{{\rm sing}}

\newcommand{\SPL}{{\rm SPL}}
\newcommand{\BO}{{\rm BO}}

\newcommand{\SO}{{\rm SO}}
\newcommand{\CC}{{\rm C}}
\newcommand{\PP}{\mathcal P}
\newcommand{\tCC}{\widetilde{\rm C}}
\newcommand{\CS}{{\rm CS}}
\newcommand{\LL}{\Cal L}
\newcommand{\LLL}{\Cal{L}_{\Cal T}}
\newcommand{\Ch}{{\rm Ch}}
\newcommand{\ch}{{\rm ch}}

\newcommand{\HS}{{\rm HS}}

\newcommand{\PM}{{\rm PM}}
\newcommand{\hY}{\widehat{Y}}
\newcommand{\bS}{{\mathbf{K}}}
\newcommand{\bt}{{\mathbf{t}}}
\newcommand{\bQ}{{\mathbf{Q}}}
\newcommand{\tbQ}{{\mathbf{\widetilde{Q}}}}

\newcommand{\St}{\mathop{\rm star}\nolimits}
\newcommand{\Cal}[1]{\mathcal{#1}}
\renewcommand{\Bbb}[1]{\mathbb{#1}}

\author{A.\,A.\,Gaifullin}

\thanks{The work was partially supported by the Russian Foundation
for Fundamental Research, grant no.~05-01-01032-a, and the
Russian Leading Scientific School Support, grant no.~4182.2006.1.}

\title{Construction of combinatorial manifolds with the prescribed
sets of links of vertices}
\date{}

\begin{document}
\sloppy

\begin{abstract}
To each oriented closed combinatorial manifold we assign the set
(with repetitions) of isomorphism classes of links of its
vertices. The obtained transformation~$\LL$ is the main object of
study of the present paper. We pose a problem on the inversion of
the transformation~$\LL$. We shall show that this problem is
closely related to N.\,Steenrod's problem on realization of cycles
and to the Rokhlin-Schwartz-Thom construction of combinatorial
Pontryagin classes. It is easy to obtain a condition of {\it
balancing} that is a necessary condition for a set of isomorphism
classes of combinatorial spheres to belong to the image of the
transformation~$\LL$. In the present paper we give an explicit
construction providing that each balanced set of isomorphism
classes of combinatorial spheres gets into the image of~$\LL$
after passing to a multiple set and adding several pairs of the
form~$(Z,-Z)$, where $-Z$ is the sphere~$Z$ with the orientation
reversed. This construction enables us, for a given singular
simplicial cycle of a space~$R$, to construct explicitly a
combinatorial manifold~$M$ and a mapping~$\varphi:M\to R$ such
that $\varphi_*[M]=r[\xi]$ for some positive integer~$r$. The
construction is based on resolving singularities of the
cycle~$\xi$. We give applications of our main construction to
cobordisms of manifolds with singularities and cobordisms of
simple cells. In particular, we prove that every rational additive
invariant of cobordisms of manifolds with singularities admits a
local formula. Another application is the construction of explicit
(though inefficient) local combinatorial formulae for polynomials
in the rational Pontryagin classes of combinatorial manifolds.
\end{abstract}

\maketitle

\tableofcontents

\section{Introduction}

A {\it combinatorial sphere} is a simplicial complex
piecewise-linearly homeomorphic to the boundary of a simplex. A
{\it combinatorial manifold} is a simplicial complex such that the
link of every its vertex is a combinatorial sphere. All manifolds
considered are supposed to be closed. An {\it isomorphism} of
oriented combinatorial manifolds is an orientation-preserving
simplicial mapping that has a simplicial inverse.

To each triangulation of a manifold one may assign various
combinatorial data characterizing it. The simplest example of such
data is the $f$-vector $(f_0,f_1,\ldots,f_n)$, where by~$f_i$ we
denote the number of~$i$-dimensional simplices of the
triangulation. More complicated combinatorial data in a sense
describe the mutual disposition of simplices. Certain functions in
combinatorial data yield invariants of a manifold independent of
the triangulation. For example the Euler characteristic of a
manifold is expressed via its $f$-vector. In this paper we assign
to each oriented combinatorial manifold the set (with repetitions)
of isomorphism classes of links of its vertices. The interest to
such combinatorial data caused among others by the fact that the
Pontryagin numbers of a manifold can be computed from the set of
isomorphism classes of links of its vertices. The existence of
functions computing  the Pontryagin numbers from the set of
isomorphism classes of links of vertices follows from a result of
N.~Levitt and C.~Rourke~\cite{LeRo78} (see
also~\cite{Gai04},~\cite{Gai05}). The problem of finding explicit
formulae will be discussed below.

Thus the main object of study of the present paper is the
transformation~$\LL$ assigning the set of isomorphism classes of
links of vertices to each oriented combinatorial manifold~$K$. The
links of vertices of~$K$ are endowed with the induced
orientations. Thus the set of isomorphism classes of links of
vertices is a set with repetitions of isomorphism classes of {\it
oriented} combinatorial spheres. In the sequel we often do not
distinguish between isomorphic combinatorial spheres. Also we
often do not distinguish between a combinatorial sphere and its
isomorphism class. We denote the isomorphism class of a
combinatorial sphere by the same letter as a combinatorial sphere
itself. For simplicity we shall often say that the
transformation~$\LL$ assigns to each combinatorial manifold the
set of links of its vertices.

Studying the transformation~$\LL$ it is natural to pose a problem
on its inversion.

\begin{question}
Given a set $Y_1,Y_2,\ldots,Y_k$ of oriented $(n-1)$-dimensional
combinatorial spheres, is there an oriented $n$-dimensional
combinatorial manifold whose set of links of vertices coincides up
to an isomorphism with the set $Y_1,Y_2,\ldots,Y_k$?
\end{question}

Consider the free Abelian group generated by isomorphism classes
of oriented $(n-1)$-dimensional spheres. We take the quotient of
this group by the relations~$Y+(-Y)=0$, where by~$-Y$ we denote
the combinatorial sphere~$Y$ with the orientation reversed. The
obtained Abelian group will be denoted by~$\Cal T_n$. We put,
$$
\diff Y=\sum\Lk y,
$$
where the sum is taken over all vertices~$y$ of the combinatorial
sphere~$Y$. Thus we obtain a differential~$d:\Cal T_n\to\Cal
T_{n-1}$ that turns the graded group~$\Cal T_*$ into a chain
complex. The complex~$\Cal T_*$ was originally defined by the
author in~\cite{Gai04}.

Let us now consider the transformation~$\LLL$ assigning to each
oriented $n$-dimensional combinatorial manifold the sum of links
of its vertices in the group~$\Cal T_n$. Since $\Cal T_n$ is a
group, the transformation~$\LLL$ is simpler to deal with than the
transformation~$\LL$. Another advantage of the
transformation~$\LLL$ consists in the fact that it commutes modulo
elements of order~$2$ with the passing to the barycentric
subdivision. The problem of the inversion on the
transformation~$\LLL$ can be formulated in the following way.

\begin{question}
Given a set $Y_1,Y_2,\ldots,Y_k$ of oriented $(n-1)$-dimensional
combinatorial spheres, is there an oriented $n$-dimensional
combinatorial manifold whose set of links of vertices coincides up
to an isomorphism with the set
$$
Y_1,Y_2,\ldots,Y_k, Z_1,Z_2\ldots,Z_l,-Z_1,-Z_2\ldots,-Z_l
$$
for some set~$Z_1,Z_2,\ldots,Z_l$ of oriented $(n-1)$-dimensional
combinatorial spheres?
\end{question}
\begin{cond}
The answer to Question~1.2 and all the more to Question~1.1 could
be ``yes'' only if the vertices of the disjoint union~$Y_1\sqcup
Y_2\sqcup\ldots\sqcup Y_k$ can be paired off so that the links of
vertices of each pair are isomorphic to each other with the
isomorphism reversing the orientation.
\end{cond}
A set of oriented combinatorial spheres satisfying this necessary
condition will be called {\it balanced}. Obviously, balanced sets
of combinatorial spheres are exactly cycles of the complex~$\Cal
T_*$.

The main result of the present paper is an explicit construction
that gives the following partial answer to Question~1.2.
\begin{theorem}
Suppose $Y_1,Y_2,\ldots,Y_k$ is a balanced set of oriented
$(n-1)$-dimensional combinatorial spheres. Then there is an
oriented $n$-dimensional combinatorial manifold~$K$ whose set of
links of vertices coincides up to an isomorphism with the set
$$\underbrace{Y_1,\ldots,Y_1}_{r},\underbrace{Y_2,\ldots,Y_2}_{r},
\ldots,\underbrace{Y_k,\ldots,Y_k}_{r},Z_1,Z_2,\ldots,Z_l,-Z_1,-Z_2,\ldots,-Z_l$$
for some positive integer~$r$ and some oriented
$(n-1)$-dimensional combinatorial spheres~$Z_1,Z_2,\ldots,Z_l$.
\end{theorem}

Questions~1.1 and~1.2 make sense not only for triangulations but
also for cubic decompositions of manifolds and even for a wider
class of {\it cubic cell} decompositions of manifolds. A {\it
cubic cell complex} (or a {\it cubic cell decomposition}) is a
decomposition into cubes such that two cubes may intersect by a
subcomplex of their boundaries rather than by a unique face. A
{\it link} of a vertex of a cubic cell complex is a simplicial
cell complex cut by cubes of the original complex in a small
sphere around this vertex. A {\it cubic cell combinatorial
manifold} is a cubic cell complex such that the link of each its
vertex is piecewise-linearly homeomorphic to the boundary of a
simplex. (Rigorous definitions will be given in section~2.1.)

For cubic cell combinatorial manifolds we shall obtain the
following partial answer to Question~1.1.
\begin{theorem}
Suppose $Y_1,Y_2,\ldots,Y_k$ is a balanced set of oriented
$(n-1)$-dimensional combinatorial spheres. Then there is an
oriented cubic cell combinatorial manifold~$X$ whose set of links
of vertices coincides up to an isomorphism with the set
$$
\underbrace{Y_1',\ldots,Y_1'}_{r},\underbrace{Y_2',\ldots,Y_2'}_{r},
\ldots,\underbrace{Y_k',\ldots,Y_k'}_{r}
$$
for some positive integer~$r$. Here $Y_i'$ is the first
barycentric subdivision of~$Y_i$.
\end{theorem}

An explicit construction of such cubic cell combinatorial manifold
will be given in sections~2.1--2.6. In sections~2.7--2.8 we use
the main construction to give an explicit construction of the
combinatorial manifold~$K$ in Theorem~1.1.

In the constructions of complexes~$X$ and~$K$ we never require
that~$Y_i$ are combinatorial spheres. Indeed, instead of a
balanced set of combinatorial spheres we may consider a balanced
set of arbitrary normal pseudo-manifolds. Certainly, in this case
the complexes~$X$ and~$K$ will be normal pseudo-manifolds rather
than combinatorial manifolds. In~\S 2 we describe our main
construction in a general context, that is, for normal
pseudo-manifolds. The formulations of analogues of Theorems~1.1
and~1.2 in this context are given in section~2.2.

Suppose~$Y_1,Y_2,\ldots,Y_k$ is a balanced set of oriented
combinatorial spheres. Assume that we are given a fixed pairing
off of vertices of combinatorial spheres~$Y_1,Y_2,\ldots,Y_k$ and
a fixed orientation-reversing isomorphisms of the links of the
vertices of each pair. Then Theorem~1.2 can be strengthened in the
following way. The cubic cell manifold~$X$ can be chosen so that
to agree with the fixed pairing off and isomorphisms
(Theorem~2.3).

In section~2.9 we point out a connection of our main construction
with the construction of~{\it small covers} due to M.~Davis and
T.~Januszkiewicz~\cite{DaJa91} (see also~\cite{BuPa04}).

\begin{remark}
Questions~1.1 and~1.2 are examples of a wide-spread in topology
problem of characterization of sets of local data that can be
realized as local invariants of a certain global object. Another
example of such problem is the problem of finding relations
between the cobordism classes of cycles realizing the Pontryagin
classes of a stably complex manifold. This problem was solved by
V.~M.~Buchstaber and A.~P.~Veselov~\cite{BuVe96}. One more example
of a problem of the considered type is the problem of
characterization of possible sets of local weights of a $\Bbb
Z_p$-action with isolated fixed points on a closed stably complex
manifold (see, for example,~\cite{BuNo71}).
\end{remark}

\S\S 3--5 are devoted to applications of the main construction.

In \S 3 we study cobordisms of manifolds with singularities. Under
a manifold with singularities of class~$\CC$ we mean a
pseudo-manifold such that the links of all its vertices belong to
the class~$\CC$. For~$\CC$ one can take an arbitrary class of
pseudo-manifolds satisfying certain natural axioms. Notice that
the cobordism groups determined by our construction do not
coincide with the cobordism groups determined by the classical
Sullivan-Baas construction (see~\cite{Sul67,Sul99,Baa73}). In more
details the connection between these two constructions will be
discussed in section~3.5. Considering the class~$\CC$ instead of
the class of combinatorial spheres we may define an analogue of
the cochain complex~$\Cal T_*$, which shall be denoted by~$\Cal
T_*^{\CC}$. We apply our main construction of~\S 2 to study the
connection between the cobordism groups of oriented manifolds with
singularities of class~$\CC$ and the homology groups of the
complex~$\Cal T_*^{\CC}$. The main theorem of~\S 3 claims that
these two groups are isomorphic modulo the class of torsion
groups. In section~3.4 we study the dual cochain complex~$\Cal
T^*_{\CC}(A)=\Hom(\Cal T_*^{\CC},A)$, where $A$ is an Abelian
group. Cocycles of this complex give local formulae for $A$-valued
additive invariants of cobordisms of oriented manifolds with
singularities of class~$\CC$. We prove that any $\Bbb Q$-valued
invariant admits a local formula unique up to a coboundary of the
complex~$\Cal T^*_{\CC}(\Bbb Q)$. In particular, we obtain a
simpler and more direct proof of the author's theorem~\cite{Gai04}
claiming that the graded cohomology group of the complex~$\Cal
T^*(\Bbb Q)=\Hom(\Cal T_*,\Bbb Q)$ is additively isomorphic to the
polynomial ring in Pontryagin classes with rational coefficients.

In~\S 4 we apply the main construction to the problem on resolving
singularities of pseudo-manifolds and to N.~Steenrod's problem on
realization of homology classes by images of the fundamental
classes of manifolds. An approach to N.~Steenrod's problem based
on resolving singularities of cycles is due to
D.~Sullivan~\cite{Sul71}. Suppose $Z$ is a pseudo-manifold,
$\Sigma\subset Z$ is a subset such that $Z\setminus\Sigma$ is an
oriented manifold. Resolving singularities of~$Z$ in sense of
D.~Sullivan is a mapping~$g:M\to Z$ such that $M$ is a manifold
and the restriction
$$
g|_{g^{-1}(Z\setminus\Sigma)}:g^{-1}(Z\setminus\Sigma)\to
Z\setminus\Sigma
$$
is a diffeomorphism (respectively, a piecewise-linear
homeomorphism, depending on a considered category of manifolds).
The first example of a pseudo-manifold whose singularities cannot
be resolved was obtained by R.~Thom~\cite{Tho58a}. It is a
$7$-dimensional cycle representing a $7$-dimensional integral
homology class not realizable in sense of N.~Steenrod.

By a {\it blow-up} of the pair~$(Z,\Sigma)$ D.~Sullivan calls a
mapping $g:(\widetilde{Z},\widetilde{\Sigma})\to (Z,\Sigma)$ such
that $\widetilde{Z}\setminus\widetilde{\Sigma}$ is a manifold and
$g|_{g^{-1}(Z\setminus\Sigma)}$ is a diffeomorphism (respectively,
a piecewise-linear homeomorphism). In~\cite{Sul71} he constructed
a complete obstruction $\mathfrak{v}_s\in H_s(Z;\Omega_{n-s-1})$
to the existence of a blow-up $(\widetilde{Z},\widetilde{\Sigma})$
of the pair~$(Z,\Sigma)$ such that
$\dim\widetilde{\Sigma}<\dim\Sigma$ (see also~\cite{Kat73}). Here
$s=\dim\Sigma$, $n=\dim Z$, and $\Omega_q$ is the $q$-dimensional
smooth (respectively, piecewise-linear) oriented cobordism group.
If all Sullivan obstructions vanish in succession, then one can
consistently decrease the dimension of the set~$\Sigma$ so as to
resolve singularities of~$Z$. D.~Sullivan noticed that his
obstructions provide geometric interpretation for the
differentials of the Atiyah-Hirzebruch spectral sequence in smooth
(respectively, piecewise-linear) cobordism. Indeed, the
fundamental class~$[Z]\in H_n(Z;\Bbb Z)=H_n(Z;\Omega_0)$ is a
cycle with respect to the differentials~$d_2,\ldots,d_{s-1}$
and~$d_s([Z])=\mathfrak{v}_s$. Therefore the Sullivan obstructions
are elements of finite order. The best possible estimates for the
orders of differentials of the Atiyah-Hirzebruch spectral sequence
in smooth oriented cobordism were obtained by
V.~M.~Buchstaber~\cite{Buc69}. In~\cite{BuDe80} S.~Buoncristiano
and M.~Ded\`o obtained immediate geometric estimates for the
Sullivan obstructions. These estimates are much weaker then the
estimates obtained by V.~M.~Buchstaber. Notice that if all
obstructions to resolving singularities vanish the results
of~\cite{Sul71},~\cite{Kat73},~\cite{BuDe80} do not give an
explicit construction of a manifold~$M$ and a mapping~$g:M\to Z$
resolving singularities.

We shall work with a more general concept of resolving
singularities, which can be called resolving singularities with
multiplicities. By resolving singularities of a
pseudo-manifold~$Z$ with multiplicity~$r$ we mean a
piecewise-smooth mapping~$g:M\to Z$ such that $M$ is a
piecewise-linear manifold and the restriction
$$
g|_{g^{-1}(Z\setminus\Sigma)}:g^{-1}(Z\setminus\Sigma)\to
Z\setminus\Sigma
$$
is an $r$-fold covering. It can be deduced from the finiteness of
orders of Sullivan obstructions that resolving singularities with
some multiplicity is always possible. The main result of~\S 4 is
an explicit construction assigning to each oriented simple cell
pseudo-manifold a combinatorial manifold~$M$ and a mapping~$g:M\to
Z$ providing resolving singularities of~$Z$ with multiplicity~$r$.
For~$\Sigma$ we take the codimension~$2$ skeleton of~$Z$. (For a
definition of a simple cell decomposition see section~4.1. In
particular, every triangulation is a simple cell decomposition.)

Different procedures for resolving singularities of an oriented
$n$-dimensional pseudo-manifold can yield combinatorial manifolds
representing different oriented piecewise-linear cobordism
classes. The same phenomenon takes  place for Hironaka procedure
for resolving singularities of algebraic varieties. Nevertheless,
our procedure of~\S 4 turns out to give a certain well-defined
class
$$
\frac{[M]}{r}\in \Omega_n^{\SPL}\otimes\Bbb
Q=\Omega_n^{\SO}\otimes\Bbb Q,
$$
depending  on the simple cell decomposition of~$Z$ only (see
section~4.5). In particular if $Z$ is either simplicial or cubic
decomposition, the manifold~$M$ represents zero class in the
group~$\Omega_n^{\SPL}\otimes\Bbb Q$.

Now suppose that $R$ is a topological space and $x\in H_n(R;\Bbb
Z)$. The homology class~$x$ can always be realized by an image of
the fundamental class of a pseudo-manifold~$Z$. Applying to~$Z$
our explicit construction of resolving singularities with
multiplicity~$r$ we shall obtain an oriented piecewise-linear
manifold~$M$ and a mapping
$$
\varphi:M\stackrel{g}{\longrightarrow}Z\longrightarrow R,
$$
realizing the homology class~$rx$.

Thus we shall obtain a purely combinatorial constructive proof of
the fact that a multiple of every integral homology class $x\in
H_n(R;\Bbb Z)$ can be realized by an image of the fundamental
class of a piecewise-linear manifold. This proof uses no
transversality theorems and no algebraic-topological results.
Unfortunately, using this combinatorial approach we are not able
to obtain reasonable information about the number~$r$. In our
construction $r$ essentially depends on the combinatorics of the
cell decomposition~$Z$. Moreover, $r$ is not bounded for a
fixed~$n$.

It is interesting to pose a question on description of the bordism
class
$$
\frac{[\varphi]}r\in\Omega_n^{\SPL}(R)\otimes\Bbb
Q=\Omega_n^{\SO}(R)\otimes\Bbb Q.
$$
A sapid answer can be obtained for two cases.

1) $Z$ is a simplicial pseudo-manifold, that is, $Z\to R$ is a
singular simplicial cycle (see section~4.6);

2) $R$ is a combinatorial manifold and the mapping~$Z\to R$ yields
a cellular cycle in the dual decomposition~$R^*$ (see
section~4.7).

In either of this cases the class~$\frac{[\varphi]}r$ depends on
the homology class~$x$ only.

If the mapping~$Z\to R$ is a singular simplicial cycle, then the
rational Pontryagin classes of~$M$ vanish. Therefore the
class~$\frac{[\varphi]}r$ is the image of~$x$ under the mapping
$$
\begin{CD}
H_*(R;\Bbb Z)@>{\eta}>>H_*(R;\Omega^{\SPL}_*\otimes\Bbb
Q)@>{\left(\ch^{\SPL}\right)^{-1}}>>\Omega^{\SPL}_*(R)\otimes\Bbb
Q,
\end{CD}
$$
where $\eta$ is the homomorphism induced by the embedding~$\Bbb
Z=\Omega^{\SPL}_0\subset\Omega^{\SPL}_0\otimes\Bbb Q$ and
$$
\ch^{\SPL}:\Omega^{\SPL}_*(R)\to H_*(R;\Omega^{\SPL}_*\otimes\Bbb
Q)
$$
is the Chern-Dold character in oriented piecewise-linear bordism.

If the mapping~$Z\to R$ yields a cellular cycle in the
decomposition~$R^*$ dual to a combinatorial manifold~$R$, then the
rational Pontryagin classes of~$M$ coincide with the pullbacks of
the rational Pontryagin classes of~$R$ along the
mapping~$\varphi$. Hence the class~$\frac{[\varphi]}r$ is the
image of~$x$ under the mapping
$$
H_*(R;\Bbb Z)\xrightarrow{D} H^*(R;\Bbb Q) \xrightarrow{\eta}
H^*(R;\Omega_{\SPL}^*\otimes\Bbb Q) \xrightarrow{\ch_{\SPL}^{-1}}
\Omega_{\SPL}^*(R)\otimes\Bbb Q
\xrightarrow{D_{\SPL}^{-1}\otimes\,\Bbb Q}
\Omega^{\SPL}_*(R)\otimes\Bbb Q,
$$
where $\eta$ is the homomorphism induced by the embedding~$\Bbb
Z=\Omega_{\SPL}^0\subset\Omega_{\SPL}^0\otimes\Bbb Q$, $D$ and
$D_{\SPL}$ are the Poincar\'e duality operators in cohomology and
oriented piecewise-linear cobordism respectively, and
$$
\ch_{\SPL}:\Omega_{\SPL}^*(R)\to H^*(R;\Omega_{\SPL}^*\otimes\Bbb
Q)
$$
is the Chern-Dold character in oriented piecewise-linear
cobordism.

In section~4.2 we investigate the cobordism ring~$\PP_*$ of
oriented simple cells. A simple cell is a cell dual to a
combinatorial sphere (for a rigorous definition, see section~4.1).
In particular, all simple convex polytopes are simple cells. The
formal sum of facets of an oriented simple cell is called the
boundary of the simple cell. (The facets are endowed with the
induced orientations.) Thus cobordisms of $n$-dimensional
piecewise-linear simple cells are $(n+1)$-dimensional simple
cells. There is a canonical homomorphism
$\Omega^{\SPL}_*\to\PP_*$, where $\Omega^{\SPL}_*$ is the oriented
piecewise-linear cobordism ring. Properties of this homomorphism
are closely related with Question~1.2. In particular, Theorem~1.1
implies that this homomorphism is an isomorphism modulo the class
of torsion groups.

The results of~\S 4 were announced by the author in~\cite{Gai07}.

One more application of the main construction of~\S 2 is a
construction of explicit local combinatorial formulae for rational
Pontryagin classes of combinatorial manifolds (see~\S 5). We
consider the following problem. For a given combinatorial manifold
construct explicitly a simplicial cycle whose homology class is
the Poincar\'e dual of a given rational Pontryagin class (or a
polynomial in Pontryagin classes) of the given manifold.
``Explicit'' means that we must give an algorithm that computes
the required cycle starting from the combinatorial structure of
the manifold only. The first combinatorial formula for the first
rational Pontryagin class was obtained in~1975 by A.~M.~Gabrielov,
I.~M.~Gelfand, and M.~V.~Losik~\cite{GGL75}. Later their result
was improved by R.~MacPherson~\cite{MP77}. The easiest and most
effective formula for the first Pontryagin class was obtained by
the author~\cite{Gai04} in 2004. For the higher Pontryagin classes
there are two approaches to constructing combinatorial formulae.
The first approach is due to I.~M.~Gelfand and R.~MacPherson. The
second one is due to J.~Cheeger~\cite{Che83}. Unfortunately,
neither of this approaches gives a formula that is purely
combinatorial. (Under a purely combinatorial formula we mean a
formula that can be applied to an arbitrary combinatorial manifold
with no additional structures and gives an explicit algorithm
computing the required cycle.) There is one more approach to the
problem of combinatorial computation of the Pontryagin classes.
This approach develops M.~Gromov's ideas and is due to
A.~S.~Mischenko. In~\cite{Mis99} he constructed a local
combinatorial Hirzebruch formula that allowed him to give a local
definition of rational Pontryagin classes of a piecewise-linear
manifold. Unfortunately, until now there is no obtained in this
way explicit combinatorial formula that computes a characteristic
cycle from a triangulation of a manifold. Author's
paper~\cite{Gai05} is devoted to the comparison of different
formulae for the Pontryagin classes of triangulated manifolds.

In the present paper we construct combinatorial formulae for all
polynomials in rational Pontryagin classes. Unfortunately, these
formulae are very inefficient. This is caused by the complexity of
the main construction of~\S 2. Nevertheless, the obtained formulae
are the first formulae for the higher Pontryagin classes that can
be applied to an arbitrary combinatorial manifold with no
additional structures and give an algorithm for computations. The
approach is distinct from the approach used in~\cite{Gai04}.
Notice that for the first Pontryagin class the formula obtained
in~\cite{Gai04} is much easier than the formula obtained in the
present paper.

Following~\cite{LeRo78},~\cite{Gai04},~\cite{Gai05} we assume that
the required characteristic cycle of a combinatorial manifold~$K$
is given by a {\it universal local formula}
$$
f_{\sharp}(K)=\mathop{\sum\limits_{\sigma\,\,\text{\textnormal{a
simplex of}}\,K,}} \limits_{\codim\sigma=n} f(\Lk\sigma)\sigma,
$$
where $f$ is a chosen function on the set of isomorphism classes
of oriented $(n-1)$-dimensional combinatorial spheres such that
$f$ changes its sign whenever we reverse the orientation of a
combinatorial sphere. ``Universality'' means that the function~$f$
is independent of the combinatorial manifold~$K$ and the
chain~$f_{\sharp}(K)$ is a cycle for any combinatorial
manifold~$K$. The basic result is the fact that for every
polynomial in rational Pontryagin classes there is a formula of
the above form. This result was obtained by the author
in~\cite{Gai04}. It improves a result of N.~Levitt and
C.~Rourke~\cite{LeRo78}. Besides, in~\cite{Gai04} it is proved
that local formulae for polynomials in rational Pontryagin classes
are exactly cocycles of the complex~$\Cal T^*(\Bbb Q)$ and a local
formula for every polynomial in Pontryagin classes is unique up to
a coboundary of~$\Cal T^*(\Bbb Q)$. Thus, for each polynomial in
rational Pontryagin classes, our goal is to compute explicitly the
value of the corresponding function~$f$ on any given combinatorial
sphere~$Y$.

Rational Pontryagin classes of piecewise-linear manifolds were
defined in the late 1950s by V.~A.~Rokhlin and
A.~S.~Schwartz~\cite{RoSh57} and independently by
R.~Thom~\cite{Tho58b} (see also~\cite{MiSt79}). According to their
construction the rational Pontryagin classes of an oriented
piecewise-linear manifold~$K$ are completely determined by the
system of all equations of the form
$$
\left\langle\pi^*L_l(p_1(K),p_2(K),\ldots,p_l(K)),[M]\right\rangle=\Sign
M,
$$
where $M\subset K\times S^q$ is an oriented $4l$-dimensional
submanifold with a trivial normal bundle, $\pi:K\times S^q\to K$
is the projection onto the first multiple, $L_l$ is the $l$th
Hirzebruch polynomial, and $\Sign M$ is the signature of~$M$. This
system completely determines the Hirzebruch classes
$$
L_l(K)=L_l(p_1(K),p_2(K),\ldots,p_l(K))
$$
and, hence, the rational Pontryagin classes~$p_l(K)$ because of
the classical R.~Thom's theorem~\cite{Tho58a} claiming that a
multiple of every integral homology class~$x\in H_{4l}(K;\Bbb Z)$
can be represented as~$\pi_*[M]$, where $M\subset K\times S^q$ is
an oriented submanifold with a trivial normal bundle.
``Implicity'' of the Rokhlin-Schwartz-Thom construction consists
in the absence of an explicit combinatorial construction for such
submanifold~$M$. We notice that instead of the embedding~$M\subset
K\times S^q$ with a trivial normal bundle we suffice to construct
a mapping~$\varphi:M\to K$ such that the pullbacks of the rational
Pontryagin classes of~$K$ are the rational Pontryagin classes
of~$M$. Such manifold~$M$ and mapping~$\varphi$ can be obtained by
an explicit construction of section~4.7.

The main result on local formulae for the Hirzebruch $L$-classes
is as follows. A function~$f:\Cal T_{4l}\to\Bbb Q$ such that
$f(-Y)=-f(Y)$ is a local formula for~$L_l$ if and only if $f$
satisfies the system of linear equations
$$
f(Y_1)+f(Y_2)+\ldots+f(Y_k)=\frac{\Sign X}{r},\eqno(*)
$$
where $Y_1,Y_2,\ldots,Y_k$ is a balanced set of oriented
$(4l-1)$-dimensional combinatorial spheres and $X$ and $r$ are a
cubic cell combinatorial manifold and a positive integer obtained
by the main construction of~\S 2 applied to the
set~$Y_1,Y_2,\ldots,Y_k$. The signature of~$X$ can be computed
explicitly either by definition or by the Ranicki-Sullivan
formula~\cite{RaSu76} (see also section~5.1). Thus the system of
equations~$(*)$ provides an explicit combinatorial description for
all local formulae for the $l$th Hirzebruch polynomial. The choice
of a canonical local formula, that is, a canonical solution~$f_0$
of the system~$(*)$ is described in section~5.2.

Sections~5.3--5.8 are devoted to a combinatorial construction of a
multiplication of the cocycles of the complex~$\Cal T^*(\Bbb Q)$.
This multiplication enables us to obtain explicit local formulae
for all polynomials in rational Pontryagin classes, since the
Hirzebruch $L$-classes generate the ring~$\Bbb Q[p_1,p_2,\ldots]$.
Unfortunately, the constructed multiplication is neither bilinear,
nor associative, nor commutative. Besides, it does not satisfy the
Leibniz formula and most probably has no natural extension to the
whole complex~$\Cal T^*(\Bbb Q)$. The question of the existence of
a bilinear associative multiplication in~$\Cal T^*(\Bbb Q)$
satisfying the Leibniz formula is still open.

I wish to express my deep gratitude to V.~M.~Buchstaber for posing
the problems and permanent attention to my work.

\section{Main construction}

\subsection{Basic definitions} It is convenient to work with the
following definition of a simplicial complex.
\begin{definition}
A {\it finite simplicial complex} is the quotient of a disjoint
union of finitely many closed simplices
$\Delta_1,\Delta_2,\ldots,\Delta_q$ by an equivalence relation
$\sim$ such that

(1) $\sim$ does not identify any two distinct points of
each~$\Delta_i$;

(2) the restriction of $\sim$ to each disjoint union
$\Delta_i\sqcup\Delta_j$, $i\ne j$, either is empty or coincides
with the identification along a linear homeomorphism of a face
$F_1\subset\Delta_i$ onto a face $F_2\subset\Delta_j$.

The images of faces of simplices $\Delta_i$ under the quotient
mapping are called {\it cells} or {\it simplices} of the
simplicial complex.
\end{definition}

The intersection of two simplices of a simplicial complex either
is empty or is a face of each of them. If we allow two simplices
to have several common faces we arrive to a more general concept
of a {\it simplicial cell complex}. To obtain a rigorous
definition of a simplicial cell complex we replace condition (2)
of Definition~2.1 by the following condition.

($2'$) if $\sim$ identifies a point $x_1\in \Delta_i$ with a point
$x_2\in\Delta_j$, then $\sim$ identifies some face
$F_1\subset\Delta_i$ containing $x_1$ with some face
$F_2\subset\Delta_j$ containing $x_2$ along some linear
homeomorphism.

Suppose $X_1$ and $X_2$ are simplicial cell complexes. A mapping
$X_1\to X_2$ is said to be {\it simplicial} if it maps each
simplex of $X_1$ linearly onto some simplex of $X_2$. An {\it
isomorphism} of simplicial cell complexes is a simplicial mapping
with the simplicial inverse.

If we replace simplices with cubes in all above definitions we
shall obtain the definitions of a {\it cubic complex}, a {\it
cubic cell complex}, a {\it cubic mapping}, and an {\it
isomorphism} of cubic cell complexes. All considered complexes
are supposed to be finite.

\begin{remark}
The term ``simplicial cell complex'' was used for example in the
papers of V.\,M.\,Buchstaber and T.\,E.\,Panov
(see~\cite{BuPa04}). Other authors used different terms for the
same concept, namely, ``pseudo-complex'',
``pseudo-triangulation'', and ``pseudo-dissection''. These terms
seem to be less convenient. The term ``cubic cell complex'' is a
natural analogue of the term ``simplicial cell complex''. As far
as we know the term ``cubic cell complex'' has never been used
before.
\end{remark}

If $X$ is a simplicial cell complex, then the cone $\cone(X)$ and
the unreduced suspension $\Sigma X$ are simplicial cell complexes
too. The barycentric subdivision of a simplicial cell or cubic
cell complex $X$ is denoted by $X'$. Notice that $X'$ is always a
simplicial complex. By $\V(X)$ we denote the set of vertices of a
complex $X$. By $C_*(X;A)$ and $C^*(X;A)$ we denote the simplicial
(respectively, cubic) chain and cochain complexes of a simplicial
cell (respectively, cubic cell) complex $X$ with coefficients in
an Abelian group $A$. An isomorphism of simplicial cell or cubic
cell complexes is denoted by~$\cong$.

The {\it link} of a simplex~$\sigma$ of a simplicial complex~$X$
is usually defined to be the subcomplex of~$X$ consisting of all
simplices~$\tau$ such that $\sigma\cap\tau=\emptyset$ and there
exists a simplex of~$X$ containing both~$\sigma$ and~$\tau$.
Nevertheless, in this paper it is more convenient to use another
definition that works well for simplicial cell complexes and cubic
cell complexes too.

\begin{definition}
Suppose $X$ is a simplicial cell complex or a cubic cell complex,
$\sigma$ is its $k$-dimensional cell. For each $l$-dimensional
cell $\tau\supset\sigma$, $\tau\ne\sigma$, there are exactly
$l-k$ cells $\rho$ such that $\dim\rho=k+1$ and
$\sigma\subset\rho\subset\tau$. The convex hull of the
barycenters of all such cells $\rho$ is an $(l-k-1)$-dimensional
simplex, which will be denoted by $\Delta_{\sigma,\tau}$. The
union of the simplices $\Delta_{\sigma,\tau}$ over all cells
$\tau\supset\sigma$, $\tau\ne\sigma$, is a simplicial cell
complex, which will be called the {\it link} of the cell $\sigma$
in the complex $X$ and will be denoted by $\Lk\sigma$ or
$\Lk_X\sigma$.
\end{definition}

Notice that for a simplicial complex $X$ the two definitions of
the link yield canonically isomorphic simplicial complexes. The
partially ordered set of simplices of the complex $\Lk\sigma$
(including the empty simplex $\emptyset$) is canonically
isomorphic to the partially ordered set of cells of $X$
containing the cell $\sigma$.

Suppose $X$ is a simplicial cell complex, $\sigma$ is a simplex
of $X$. The {\it star} of $\sigma$ is the subcomplex
$\St\sigma\subset X$ consisting of all closed simplices
containing the simplex $\sigma$. If $X$ is a simplicial complex,
then $\St\sigma\cong\sigma*\Lk\sigma$.

A simplicial complex is called a {\it combinatorial sphere} if it
is piecewise linearly homeomorphic to the boundary of a simplex. A
simplicial complex is called an $n$-dimensional {\it combinatorial
manifold} if the link of each its vertex is an $(n-1)$-dimensional
combinatorial sphere. Similarly, a cubic (respectively, simplicial
cell or cubic cell) complex is called an $n$-dimensional {\it
cubic} (respectively, {\it simplicial cell} or {\it cubic cell})
{\it combinatorial manifold} if the link of each its vertex is
piecewise linearly homeomorphic to the boundary of an
$n$-dimensional simplex.

A simplicial (respectively, simplicial cell, cubic, or cubic cell)
complex is called an $n$-dimensional {\it simplicial}
(respectively, {\it simplicial cell}, {\it cubic}, or {\it cubic
cell}) {\it pseudo-manifold} if each cell of $X$ is contained in
some $n$-dimensional cell and each $(n-1)$-dimensional cell of
$X$ is contained in exactly two $n$-dimensional cells.

It is convenient to work with so-called {\it normal}
pseudo-manifolds (see~\cite{GoMP80}). A simplicial cell or cubic
cell $n$-dimensional pseudo-manifold is called normal if
$H_n(X,X\setminus x)\cong\Bbb Z$ for every point $x\in X$.
Equivalently, $X$ is normal if the link of each its cell of
dimension not greater than $n-2$ is connected. The connected
components of a normal pseudo-manifold $X$ are {\it strongly
connected}, that is, for each two $n$-dimensional cells $\tau_1$
and $\tau_2$ there is a sequence of $n$-dimensional cells
$\tau_1=\rho_1,\rho_2,\ldots,\rho_r=\tau_2$ such that for every
$i$ the cells $\rho_i$ and $\rho_{i+1}$ have a common
$(n-1)$-dimensional face. It is easy to check that for
$\dim\sigma\leqslant n-2$ the link of the cell $\sigma$ of a
normal pseudo-manifold $X$ is a connected normal pseudo-manifold.
In particular, the link of $\sigma$ is strongly connected.

In the sequel we shall always consider only normal
pseudo-manifolds and under a pseudo-manifold we shall always mean
a normal pseudo-manifold. The class of normal pseudo-manifolds
include all most interesting examples of pseudo-manifolds, namely,
combinatorial manifolds, homology manifolds, manifolds with conic
singularities. For an arbitrary $n$-dimensional pseudo-manifold
$X$ one can construct its {\it normalization}
$X^{\norm}$~\cite{GoMP80}. Let $\sigma_1,\sigma_2,\ldots,\sigma_q$
be all $n$-dimensional simplices (respectively, cubes) of a
pseudo-manifold $X$. We consider the disjoint union of
$\sigma_1,\sigma_2,\ldots,\sigma_q$ and make the following
identifications. If two cells $\sigma_i$ and $\sigma_j$ have a
common facet in $X$ we identify their corresponding facets along
the corresponding isomorphism. It is easy to check that the
obtained pseudo-manifold is normal. In addition there is a
simplicial (respectively, cubic) mapping~$X^{\norm}\to X$ whose
restriction to the complement of the $(n-2)$-skeleton of
$X^{\norm}$ is a homeomorphism with the complement of the
$(n-2)$-skeleton of $X$.

In the sequel we shall work mostly with oriented pseudo-manifolds.
Under an isomorphism of oriented pseudo-manifolds we always mean
an orientation-preserving isomorphism. An orientation-reversing
isomorphism is called an {\it anti-isomorphism}. An orientation of
a simplex is conveniently given by an ordering of its vertices.
Suppose $\sigma$ and $\tau$ are oriented simplices. Then the join
$\sigma*\tau$ is endowed with the orientation given by the
concatenation of a sequence of vertices yielding the orientation
of $\sigma$ with a sequence of vertices yielding the orientation
of $\tau$. In the same way the link of two oriented simplicial
cell pseudo-manifolds is endowed with the orientation. Suppose
$\sigma$ is an oriented cell of an oriented pseudo-manifold $X$.
Then the orientation of $\Lk\sigma$ will be chosen so that the
induced orientation of the complex $\sigma*\Lk\sigma\subset X$
coincides with the restriction of the given orientation of $X$.
The cone over an oriented simplicial cell pseudo-manifold $X$ is
always endowed with the orientation such that the canonical
isomorphism of the link of the cone vertex with~$X$ preserves the
orientation.

\subsection{Results} Let $Y_1,Y_2,\ldots,Y_k$ be a set (with repetitions)  of
connected oriented simplicial pseudo-manifolds of the same
dimension. The set $Y_1,Y_2,\ldots,Y_k$ is called {\it balanced}
if the vertices of the pseudo-manifold $Y=Y_1\sqcup
Y_2\sqcup\ldots\sqcup Y_k$ can be paired off so that the links of
the vertices of each pair are anti-isomorphic. (Sometimes, we
shall say that the pseudo-manifold $Y$ is balanced.) Let us
formulate the versions of Theorems~1.1 and~1.2 for
pseudo-manifolds. (Recall that under a pseudo-manifold we always
mean a normal pseudo-manifold.)
\begin{theorem}
Suppose $Y_1,Y_2,\ldots,Y_k$ is a balanced set of connected
oriented $(n-1)$-dimensional simplicial pseudo-manifolds. Then
there is an oriented $n$-dimensional simplicial pseudo-manifold
$K$ whose set of links of vertices coincides up to an isomorphism
with the set
$$\underbrace{Y_1,\ldots,Y_1}_{r},\underbrace{Y_2,\ldots,Y_2}_{r},
\ldots,\underbrace{Y_k,\ldots,Y_k}_{r},Z_1,Z_2,\ldots,Z_l,-Z_1,-Z_2,\ldots,-Z_l$$
for some positive integer $r$ and some connected oriented
$(n-1)$-dimensional simplicial pseudo-manifolds
$Z_1,Z_2,\ldots,Z_l$.
\end{theorem}
\begin{theorem}
Suppose $Y_1,Y_2,\ldots,Y_k$ is a balanced set of connected
oriented $(n-1)$-dimensional simplicial pseudo-manifolds. Then
there is an oriented $n$-dimensional cubic cell pseudo-manifold
$X$ whose set of links of vertices coincides up to an isomorphism
with the set
$$
\underbrace{Y_1',\ldots,Y_1'}_{r},\underbrace{Y_2',\ldots,Y_2'}_{r},
\ldots,\underbrace{Y_k',\ldots,Y_k'}_{r}
$$
for some positive integer $r$.
\end{theorem}

The main goal of this section is to give explicit constructions of
such pseudo-manifolds $K$ and $X$. Theorem~1.2 is a
straightforward consequence of Theorem~2.2. Indeed, if
$Y_1,Y_2,\ldots,Y_k$ are combinatorial spheres, then  the
pseudo-manifold $X$ is necessarily a cubic cell combinatorial
manifold. An explicit construction of the pseudo-manifold $X$ will
be given in \S\S~2.3--2.5.

Theorem~1.1 is not a straightforward consequence of Theorem~2.1.
It is possible that $Y_1,Y_2,\ldots,Y_k$ are combinatorial spheres
and $K$ is not a combinatorial manifold if some of the
pseudo-manifolds $Z_1,Z_2,\ldots,Z_l$ is not a combinatorial
sphere. However, there is an explicit construction that works both
under the conditions of Theorem~1.1 and under the conditions of
Theorem~2.1. Notice that the most complicated part of this
construction is the construction of the complex $X$. Given the
complex $X$ it is not very hard to construct the complex $K$. To
simplify the description of the construction we consider the case
of combinatorial spheres $Y_1,Y_2,\ldots,Y_k$ (see~\S\S~2.7,~2.8).
The case of pseudo-manifolds is quite similar. The assertions
similar to Theorems~1.1 and~2.1 hold for certain intermediate
classes between the class of combinatorial manifolds the class of
all simplicial pseudo-manifolds, for example, for the class of all
simplicial homology manifolds (see~\S~4.3).

Suppose $X$ is a simplicial cell pseudo-manifold or a cubic cell
pseudo-manifold. We put,
$$
\Cal L(X)=\bigsqcup\limits_{x\in \V(X)}\Lk x.
$$
The vertices of the pseudo-manifolds $\Cal L(X)$ are in
one-to-one correspondence with the directed edges of the
pseudo-manifold $X$. Reversing the direction of an edge gives us
the involution $\lambda_X:z\mapsto \tilde{z}$ on the set of
vertices of the pseudo-manifolds $\Cal L(X)$ and the set of
anti-isomorphisms
$$
\chi_{X,z}:\St z\to\St \tilde{z},\qquad z\in \V(\Cal L(X))
$$
such that $\chi_{X,\tilde{z}}=\chi_{X,z}^{-1}$.

Now let us consider a balanced set $Y_1,Y_2,\ldots,Y_k$ of
$(n-1)$-dimensional simplicial pseudo-manifolds. The vertices of
the pseudo-manifold $Y=Y_1\sqcup Y_2\sqcup\ldots\sqcup Y_k$ can be
paired off so that the links of the vertices of each pair are
anti-isomorphic. Equivalently, the stars of the vertices of each
pair are anti-isomorphic. Hence there is an involution
$\lambda:y\mapsto \tilde{y}$ on the set $\V(Y)$ and the set of
anti-isomorphisms
$$
\chi_y:\St y\to\St \tilde{y},\qquad y\in \V(Y)
$$
such that $\chi_{\tilde{y}}=\chi_y^{-1}$. Obviously, it is
possible that the involution $\lambda$ and the anti-isomorphisms
$\chi_y$ can be chosen in several different ways. Let us fix an
arbitrary choice of them. The set of vertices of~$Y'$ can be
decomposed in the following way.
$$
\V(Y')=\bigsqcup\limits_{j=1}^{n}W_j(Y),
$$
where $W_j(Y)$ is the set of barycenters of $(j-1)$-dimensional
simplices of~$Y$. (In particular, $W_1(Y)=\V(Y)$.)

The construction described in \S\S~2.3--2.5 will give us a result
that is stronger than Theorem~2.2. Actually we shall construct a
pseudo-manifold $X$ so that the involution $\lambda_X$ and the
anti-isomorphisms $\chi_{X,z}$ will agree with the involution
$\lambda$ and the anti-isomorphisms $\chi_y$. The precise
formulation of this result is as follows.

\begin{theorem}
Suppose $Y$ is an oriented $(n-1)$-dimensional simplicial
pseudo-manifold, $\lambda:y\mapsto \tilde{y}$ is an involution on
the set $\V(Y)$, $\chi_y:\St y\to\St \tilde{y}$ are
anti-isomorphisms such that $\chi_{\tilde{y}}=\chi_y^{-1}$. Then
there are an oriented $n$-dimensional cubic cell pseudo-manifold
$X$ and an isomorphism $\Cal L(X)\cong Y'\times S$, where $S$ is
a finite set, such that

1\textnormal{)} the subsets $W_j(Y)\times S\subset \V(Y')\times
S=\V(\Cal L(X))$ are invariant under the involution $\lambda_X$;

2\textnormal{)} if $y\in \V(Y)$, $s\in S$, then
$\lambda_X(y,s)=(\lambda(y),s_1)$ for some element $s_1\in S$
depending on $y$;

3\textnormal{)} if $y\in \V(Y)$, $s\in S$, then the
anti-isomorphism
$$
\chi_{X,(y,s)}:\St_{Y'\times
S}(y,s)\to\St_{Y'\times S}(\lambda(y),s_1)
$$
is induced by the anti-isomorphism
$\chi_y:\St_{Y}y\to\St_{Y}\lambda(y)$.
\end{theorem}

We notice that this theorem is non-trivial even if all
combinatorial spheres $Y_i$ are isomorphic to the boundary of an
$n$-dimensional simplex. This case will be used in~\S 4.6.

\subsection{Pseudo-manifolds constructed from graphs}
In this paper a graph is always a finite graph with indirect
edges. A graph may contain multiple edges but should not contain
loops. A {\it homogeneous} graph is a graph with all vertices of
the same degree. Let $\Gamma$ be a homogeneous graph of degree
$n$. Let $A$ be an $n$-element set. Suppose that to each edge of
$\Gamma$ is assigned an element of $A$ called a {\it colour} of
the edge. We shall say that the colouration is {\it regular} if
for any two adjacent edges their colours are distinct.
Equivalently, for each vertex $v$ the $n$ edges containing $v$ are
coloured in $n$ pairwise distinct colours.

Let $V$ be the vertex set of a homogeneous graph $\Gamma$ with
regularly coloured edges. To a colour $a\in A$ we assign the
mapping $\Phi_a:V\to V$ that takes each vertex $v\in V$ to the
vertex connected with $v$ by an edge of colour $a$. Then $\Phi_a$
is an involution without fixed points. On the other hand, from an
arbitrary finite set $V$ and a set of involutions $\Phi_a:V\to V,\
a\in A$, we can construct a homogeneous graph $\Gamma$ on the
vertex set $V$ with edges regularly coloured by elements of the
set $A$. To obtain the graph $\Gamma$ we connect the vertices $v$
and $\Phi_a(v)$ by an edge of colour $a$ for every vertex $v\in V$
and every element $a\in A$.

Suppose $Y$ is an $(n-1)$-dimensional simplicial cell
pseudo-manifold. A colouration of vertices of $Y$ by elements of
an $n$-element set $A$ is said to be {\it regular} if the vertices
of each $(n-1)$-dimensional simplex are coloured in $n$ pairwise
distinct colours. It is worth noting that vertices of the
pseudo-manifold $Y$ can be naturally regularly coloured in colours
from an $n$-element set if $Y$ is the barycentric subdivision of a
simplicial cell pseudo-manifold $Z$. To obtain the colouration we
colour the barycenter of every $j$-dimensional simplex of $Z$ in
colour $j+1$.

It turns out that $(n-1)$-dimensional simplicial cell
pseudo-manifolds with regularly coloured vertices are in
one-to-one correspondence with homogeneous graphs of degree $n$
with regularly coloured edges. The construction yielding this
correspondence is due to M.\,Pezzana~\cite{Pez75} in dimension~$3$
and to M.\,Ferri~\cite{Fer76} in general case (see
also~\cite{FGG86}). (M.\,Pezzana and M.\,Ferri did not consider
the condition of normality. Hence their correspondence was not
one-to-one.) We shall describe this construction in a convenient
to us form. Without loss of generality we may assume that $A=\{
1,2,\ldots,n\}$.

{\it The construction of a graph from a pseudo-manifold.} Suppose
$Y$ is an $(n-1)$-dimensional simplicial cell pseudo-manifold with
regularly coloured vertices. To each $(n-1)$-dimensional simplex
$\sigma$ of the complex $Y$ we assign a vertex $v_{\sigma}$. To
each $(n-2)$-dimensional simplex $\tau$ of $Y$ we assign an edge
$e_{\tau}$ connecting the vertices $v_{\sigma_1}$ and
$v_{\sigma_2}$, where $\sigma_1$ and $\sigma_2$ are the two
$(n-1)$-dimensional simplices containing $\tau$. Paint the edge
$e_{\tau}$ in colour $j$ if the simplex $\tau$ does not contain a
vertex of colour~$j$.

{\it The construction of a pseudo-manifold from a graph.} Suppose
$V$ is the vertex set of a graph~$\Gamma$,
$\Phi_1,\Phi_2,\ldots,\Phi_n$ are the involutions defined above.
Consider the standard $(n-1)$-dimensional simplex
$$
\Delta^{n-1}=\left\{\left.\left(t_1,t_2,\ldots,t_n\right)\in\Bbb
R^n\right| t_1+t_2+\cdots+t_n=1,\,t_j\geqslant
0,\,j=1,2,\ldots,n\right\}.
$$
Put,
$$
\bS(\Gamma)=\left.V\times\Delta^{n-1}\right/\sim,
$$
where the equivalence relation $\sim$ is generated by the
identifications
$$
(v,t_1,t_2,\ldots,t_n)\sim (\Phi_j(v),t_1,t_2,\ldots,t_n)\text{ if
}t_j=0.
$$
It can be immediately checked that $\bS(\Gamma)$ is a normal
pseudo-manifold. The colouration of vertices of the complex
$\bS(\Gamma)$ is induced by the standard colouration of vertices
of the simplex~$\Delta^{n-1}$. It is easy to check that the
described constructions are inverse to each other.

{\it The partially ordered set of simplices of the complex
$\bS(\Gamma)$.} For each subset $B\subset A$ by $\Gamma_B$ we
denote the graph obtained from $\Gamma$ after deleting all edges
whose colours do not belong to~$B$. Then $\Gamma_B$ is a
homogeneous graph of degree equal to the cardinality of the set
$B$. By $\Cal K(\Gamma)$ we denote the set of all pairs
$(B,\Upsilon)$ such that $B$ is a subset of $A$, $B\ne A$,
$\Upsilon$ is a connected component of~$\Gamma_B$. We introduce a
partial order on the set $\Cal K(\Gamma)$ by putting
$(B_1,\Upsilon_1)\leqslant(B_2,\Upsilon_2)$ if and only if
$B_2\subset B_1$ and $\Upsilon_2\subset \Upsilon_1$.
\begin{propos}
The set of simplices of the complex $\bS(\Gamma)$
\textnormal{(}not including the empty simplex\textnormal{)}
partially ordered by inclusion is isomorphic to the partially
ordered set $\Cal K(\Gamma)$. This isomorphism takes each simplex
$\sigma$ of $\bS(\Gamma)$ to a pair $(B,\Upsilon)$ such that
$B\subset A$ is the subset complement to the subset of colours of
vertices of the simplex~$\sigma$.
\end{propos}
\begin{proof}
By $\Delta_v$ we denote the $(n-1)$-dimensional simplex of
$\bS(\Gamma)$ corresponding to a vertex $v\in V$. For each subset
$B\subsetneq A$ by $\Delta_{B,v}$ we denote the face of $\Delta_v$
spanned by all vertices whose colours do not belong to $B$. Then
$\Delta_{B,v}$ is a $(n-k-1)$-dimensional simplex, where $k$ is
the cardinality of~$B$. Obviously, each simplex of the complex
$\bS(\Gamma)$ is $\Delta_{B,v}$ for certain $B$ and $v$. Looking
at the equivalence relation $\sim$ in the definition of the
complex $\bS(\Gamma)$ we can easily deduce that
$\Delta_{B,v_1}=\Delta_{B,v_2}$ whenever the vertices $v_1$ and
$v_2$ are connected by an edge whose colour belong to~$B$.
Therefore $\Delta_{B,v_1}=\Delta_{B,v_2}$ if the vertices $v_1$
and $v_2$ lie in the same connected component of the
graph~$\Gamma_B$. Furthermore, it can be easily checked that the
simplices $\Delta_{B,v_1}$ and $\Delta_{B,v_2}$ are distinct if
the vertices $v_1$ and $v_2$ lie in different connected components
of~$\Gamma_B$.
\end{proof}

The simplex of~$\bS(\Gamma)$ corresponding to a
pair~$(B,\Upsilon)$ will be denoted by~$\Delta_{B,\Upsilon}$. For
$(n-1)$-dimensional simplices we shall usually use notation
$\Delta_v$ instead of~$\Delta_{\emptyset,v}$.

\begin{propos}
The link of a simplex~$\Delta_{B,\Upsilon}$ in the simplicial cell
complex~$\bS(\Gamma)$ is isomorphic to~$\bS(\Upsilon)$.
\end{propos}
\begin{proof}
It is easy to check that the partially ordered set~$\Cal
K(\Upsilon)$ is isomorphic to the interval
$\left((B,\Upsilon),+\infty\right)$ of the partially ordered
set~$\Cal K(\Gamma)$. Hence it is isomorphic to the partially
ordered set of simplices of~$\bS(\Gamma)$ that
contain~$\Delta_{B,\Upsilon}$ and do not coincide
with~$\Delta_{B,\Upsilon}$. Consequently the partially ordered set
of simplices of the complex~$\bS(\Upsilon)$ is isomorphic to the
partially ordered set of simplices of the
complex~$\Lk\Delta_{B,\Upsilon}$.
\end{proof}

{\it Orientation.} It is easy to check that the
pseudo-manifold~$\bS(\Gamma)$ is orientable if and only if the
graph~$\Gamma$ does not contain a cycle of odd length, that is, is
bipartite. Suppose that we have a decomposition of the set~$V$
into two disjoint subsets $V_+$ and $V_-$ such that the
graph~$\Gamma$ is bipartite with respect to this decomposition. We
endow the simplex $\Delta_v$ with the orientation induced by the
canonical orientation of the standard simplex
$\Delta^{n-1}\subset\Bbb R^n$ if $v\in V_+$ and we endow the
simplex $\Delta_v$ with the orientation opposite to the
orientation induced by the canonical orientation
of~$\Delta^{n-1}\subset\Bbb R^n$ if $v\in V_-$. It can be easily
checked that the introduced orientations on the
simplices~$\Delta_v$ agree. Thus the decomposition $V=V_+\sqcup
V_-$ yields an orientation of~$\bS(\Gamma)$. Vice versa, each
orientation of~$\bS(\Gamma)$ provides a decomposition $V=V_+\sqcup
V_-$ with respect to which the graph~$\Gamma$ is bipartite.

We need the following version of the Pezzana-Ferri construction
for cubic cell complexes. Let~$\Gamma$ be a homogeneous graph of
degree~$2n$ with edges regularly coloured by the elements of the
set
$$
A=\left\{ 1^0,2^0,\ldots,n^0,1^1,2^1,\ldots,n^1\right\}.
$$
Let $V$ be the vertex set of $\Gamma$. By~$\Phi_j^{e}$ we denote
the involution corresponding to the colour~$j^{e}$. To the
graph~$\Gamma$ we assign a normal $n$-dimension cubic cell
pseudo-manifold
$$
\bQ(\Gamma)=\left.V\times I^n\right/\sim,
$$
where $I=[0,1]$ and the equivalence relation~$\sim$ is generated
by the identifications
\begin{gather*}
(v,t_1,t_2,\ldots,t_n)\sim (\Phi^0_j(v),t_1,t_2,\ldots,t_n)\text{
if }t_j=0;\\
(v,t_1,t_2,\ldots,t_n)\sim (\Phi^1_j(v),t_1,t_2,\ldots,t_n)\text{
if }t_j=1.
\end{gather*}

Let $x$ be a vertex of a cube~$v\times I^n$. The vector of
coordinates of the corresponding vertex of the standard cube~$I^n$
is called the {\it type} of the vertex $x$. We notice that the
equivalence relation~$\sim$ identifies the vertices of the cubes
only if they are of the same type. Therefore the type of a vertex
of the complex~$\bQ(\Gamma)$ is well defined. The type is an
$n$-dimensional vector with each coordinate equal either to $0$ or
to $1$. Let $\mathbf{e}=(e_1,e_2,\ldots,e_n)$ be a vector with
each coordinate equal either to $0$ or to $1$. By definition, we
put
$$
A_{\mathbf{e}}=\left\{1^{e_1},2^{e_2},\ldots,
n^{e_n}\right\};\qquad
\Gamma_{\mathbf{e}}=\Gamma_{A_{\mathbf{e}}}.
$$
We shall say that a subset $B\subset A$ is {\it good} if for each
$j$ the subset $B$ contains not more than one of the
elements~$j^0$ and~$j^1$. By~$\Cal Q(\Gamma)$ we denote the set of
all pairs~$(B,\Upsilon)$ such that $B$ is a good subset of~$A$ and
$\Upsilon$ is a connected component of the graph~$\Gamma_B$. Let
us introduce a partial order on the set~$\Cal Q(\Gamma)$ in the
same manner as on the set~$\Cal K(\Gamma)$. The following two
propositions are quite similar to Propositions~2.1 and~2.2.

\begin{propos}
The set of cubes of the complex $\bQ(\Gamma)$ \textnormal{(}not
including the empty cube\textnormal{)} partially ordered by
inclusion is isomorphic to the partially ordered set~$\Cal
Q(\Gamma)$. This isomorphism takes each cube~$\sigma$
of~$\bQ(\Gamma)$ to a pair $(B,\Upsilon)$, where $B\subset A$ is
the intersection of all subsets $A_{\mathbf{e}}$ such that the
cube~$\sigma$ contains a vertex of type~$\mathbf{e}$. In
particular, a pair $(A_{\mathbf{e}},\Upsilon)$, where $\Upsilon$
is a connected component of~$\Gamma_{\mathbf{e}}$, corresponds to
a vertex of type $\mathbf{e}$.
\end{propos}

The cube of~$\bQ(\Gamma)$ corresponding to a pair~$(B,\Upsilon)$
will be denote by~$\square_{B,\Upsilon}$. For $n$-dimensional
cubes we shall usually use notation~$\square_v$ instead
of~$\square_{\emptyset,v}$.

\begin{propos}
The link of a cube $\square_{B,\Upsilon}$ in the cubic cell
complex~$\bQ(\Gamma)$ is isomorphic to~$\bS(\Upsilon)$.
\end{propos}

The pseudo-manifold~$\bQ(\Gamma)$ is orientable if and only if
the graph~$\Gamma$ does not contain a cycle of odd length. A
decomposition of the set~$V$ into two disjoint subsets~$V_+$
and~$V_-$ with respect to which the graph~$\Gamma$ is bipartite
induces an orientation of the pseudo-manifold~$\bQ(\Gamma)$.

Let $x=\square_{A_{\mathbf{e}},\Upsilon}$ be an arbitrary vertex
of type~$\mathbf{e}$. It follows from Proposition~2.4 that the
link of~$x$ is isomorphic to~$\bS(\Upsilon)$. Now we assume that
the graph~$\Gamma$ is bipartite. Then the orientation of the
pseudo-manifold~$\bQ(\Gamma)$ induces the orientation of~$\Lk x$.
On the other hand, forgetting about the upper indices of colours
we may assume that edges of the graph~$\Gamma_{\mathbf{e}}$ are
coloured regularly in the colours~$1,2,\ldots,n$. The
decomposition of~$\Gamma$ into two parts provides the
decomposition of~$\Gamma_{\mathbf{e}}$ into two parts.
Consequently the graph~$\Gamma_{\mathbf{e}}$ and its connected
component~$\Upsilon$ are bipartite graphs with edges coloured
regularly in colours $1,2,\ldots,n$. Thus the simplicial cell
pseudo-manifold~$\bS(\Upsilon)$ obtains the orientation. Does the
isomorphism between the complexes~$\Lk x$ and~$\bS(\Upsilon)$
preserve the orientation? One can immediately check the following
proposition.
\begin{propos}
The isomorphism between the complexes $\Lk x$ and $\bS(\Upsilon)$
established in Proposition~2.4 preserves the orientation  if and
only if the sum~$e_1+e_2+\ldots+e_n$ is even.
\end{propos}

\subsection{Large cubes}
Let us consider a cube~$[0,1]^n\subset\Bbb R^n$ and divide it
into~$2^n$ cubes by the hyperplanes~$t_j=\frac12$. This
decomposition will be called the {\it canonical subdivision} of
the standard cube. Now suppose that $X$ is an arbitrary cubic cell
complex. We subdivide canonically each cube of~$X$. The obtained
cubic cell complex will be called the {\it canonical subdivision}
of~$X$.

In this section we shall show that for a graph $\Gamma$ satisfying
certain special conditions the cubic cell complex~$\bQ(\Gamma)$ is
the canonical subdivision of a cubic cell complex, which will be
denoted by~$\tbQ(\Gamma)$.

\begin{propos}
Suppose that the graph~$\Gamma$ satisfies the following
conditions.

1\textnormal{)} $\Phi_i^1\circ\Phi_j^1=\Phi^1_j\circ\Phi^1_i$ for
all $i$ and $j$;

2\textnormal{)} $\Phi_i^0\circ\Phi_j^1=\Phi^1_j\circ\Phi^0_i$ for
$i\neq j$;

3\textnormal{)}  There is a mapping $h:V\to\Bbb Z_2^n$ such that

\ \ \ $h(\Phi_i^0(v))=h(v)$, $i=1,2,\ldots,n$;

\ \ \ $h(\Phi_i^1(v))=h(v)+\varepsilon_i$, $i=1,2,\ldots,n$, where
$(\varepsilon_1,\ldots,\varepsilon_n)$ is the basis of the group~$\Bbb Z_2^n$.\\
Then the pseudo-manifold~$\bQ(\Gamma)$ is the canonical
subdivision of a certain cubic cell
pseudo-manifold~$\tbQ(\Gamma)$.
\end{propos}
\begin{proof}
It follows from condition~1) that the involutions~$\Phi^1_j:V\to
V$ yield an action of the group~$\Bbb Z_2^n$ on the set~$V$.
Besides, it follows from condition~3) that this action is free. We
put $\widetilde{V}=V/\Bbb Z_2^n$. Let $p:V\to \widetilde{V}$ be
the quotient mapping. It can be easily deduced from condition~3)
that the mapping~$p\times h:V\to \widetilde{V} \times\Bbb Z_2^n$
is a bijection. Therefore we may regard the
involutions~$\Phi_j^{e}$ as involutions on the
set~$\widetilde{V}\times\Bbb Z_2^n$. We define the
mappings~$\widetilde{\Phi}^{e}_j:\widetilde{V}\to \widetilde{V}$
by putting
$$
\widetilde{\Phi}^{e}_j(u)=p\left(\Phi^0_j(u,e
\varepsilon_j)\right).
$$
Then $\Phi^0_j(u,e
\varepsilon_j)=\left(\widetilde{\Phi}^{e}_j(u),e
\varepsilon_j\right)$. Hence,
$\widetilde{\Phi}^{e}_j\left(\widetilde{\Phi}^{e}_j(u)\right)=u$.
Besides, if $\widetilde{\Phi}^{e}_j(u)=u$, then $\Phi^0_j(u,e
\varepsilon_j)=(u,e \varepsilon_j)$, which cannot be true. Thus
the mappings $\widetilde{\Phi}^{e}_j$ are involutions without
fixed points. These involutions generate a homogeneous graph of
degree~$2n$ on the vertex set~$\widetilde{V}$ with edges regularly
coloured by elements of the set~$A$. We denote this graph
by~$\widetilde{\Gamma}$. We put
$\tbQ(\Gamma)=\bQ(\widetilde{\Gamma})$. It is easy to check that
the canonical subdivision of~$\tbQ(\Gamma)$ is isomorphic
to~$\bQ(\Gamma)$.
\end{proof}

We notice that the cubic cell pseudo-manifold~$\tbQ(\Gamma)$ can
be given by
$$
\tbQ(\Gamma)=\left.V\times I^n\right/\sim,
$$
where $\sim$ is the equivalence relation generated by the
identifications
\begin{gather*}
(v,t_1,t_2,\ldots,t_n)\sim
(\Phi^0_j(v),t_1,t_2,\ldots,t_n),\text{
if }t_j=0;\\
(v,t_1,\ldots,t_{j-1},t_j,t_{j+1}\ldots,t_n)\sim
(\Phi^1_j(v),t_1,\ldots,t_{j-1},1-t_j,t_{j+1}\ldots,t_n).
\end{gather*}
Vertices of the complex~$\bQ(\Gamma)$ are exactly the barycenters
of cubes of the complex~$\tbQ(\Gamma)$. Moreover, vertices
of~$\bQ(\Gamma)$ of type $\mathbf{0}=(0,\ldots,0)$ are exactly
vertices of~$\tbQ(\Gamma)$. Obviously, for each vertex $x$
of~$\tbQ(\Gamma)$ the link of~$x$ in the complex~$\tbQ(\Gamma)$ is
isomorphic to the link of~$x$ in the complex~$\bQ(\Gamma)$.

\subsection{Construction of the pseudo-manifold~$X$} In this section we shall
give an explicit construction of the cubic cell
pseudo-manifold~$X$ in Theorem~2.3. We colour the barycenter of
every $j$-dimensional simplex of~$Y$ in colour~$j+1$. Then $Y'$ is
a simplicial pseudo-manifold with vertices regularly coloured
in~$n$ colours. By $\Upsilon$ we denote the corresponding
homogeneous graph of degree~$n$. Our goal is to construct a
homogeneous graph~$\Gamma$ of degree~$2n$ with edges regularly
coloured in colours $1^0,2^0,\ldots,n^0,1^1,2^1,\ldots,n^1$ such
that $\tbQ(\Gamma)$ is the required pseudo-manifold~$X$.

Let $U$ be the set of vertices of the graph~$\Upsilon$. The
orientation of the pseudo-manifold~$Y$ yields the decomposition
$U=U_+\sqcup U_-$ with respect to which the graph~$\Upsilon$ is
bipartite. By $\Phi_1,\ldots,\Phi_n$ we denote the involutions
determining the graph~$\Upsilon$. Let $S$ be a finite set of $r$
elements. We put,
$$
V=U\times S,\qquad V_+=U_+\times S,\qquad V_-=U_-\times S.
$$
We define the involutions $\Phi^0_j:V\to V$ by
$\Phi^0_j(u,s)=\left(\Phi_j(u),s\right)$.

\begin{propos}
Suppose $\Phi_j^1:V\to V$, $j=1,2,\ldots,n$, are involutions
satisfying the conditions~$\Phi_j^1(V_+)=V_-$, $\Phi_j^1(V_-)=V_+$
and conditions~1\textnormal{)}--3\textnormal{)} of
Proposition~2.6. Let $\Gamma$ be the graph given by the
involutions~$\Phi_j^0$,~$\Phi_j^1$. Then $X=\tbQ(\Gamma)$ is an
oriented cubic cell pseudo-manifold with $rk$ vertices that can be
divided into $k$ groups each consisting of $r$ vertices so that
the link of each vertex of the $l$th groups is isomorphic
to~$Y'_l$.
\end{propos}
\begin{proof}
It follows from Proposition~2.6 that the
pseudo-manifold~$\tbQ(\Gamma)$ is well defined. The
graph~$\Gamma$ is bipartite with respect to the
decomposition~$V=V_+\sqcup V_-$. Hence the
pseudo-manifold~$\tbQ(\Gamma)$ is oriented.

By $\Upsilon_l$ we denote the graph corresponding to the
pseudo-manifold~$Y_l'$. The graphs~$\Upsilon_l$ are connected and
$$
\Upsilon=\Upsilon_1\sqcup\Upsilon_2\sqcup\ldots\sqcup\Upsilon_k.
$$

Vertices of~$\tbQ(\Gamma)$ are vertices of~$\bQ(\Gamma)$ of type
$\mathbf{0}=(0,\ldots,0)$. Therefore vertices of~$\tbQ(\Gamma)$
are in one-to-one correspondence with connected components
of~$\Gamma_{\mathbf{0}}$. The graph $\Gamma_{\mathbf{0}}$ is
obtained from the graph~$\Gamma$ by deleting all edges of colours
$1^1,2^1,\ldots,n^1$. Hence,
$$
\Gamma_{\mathbf{0}}\cong\Upsilon^{\sqcup
r}=\left(\Upsilon_1\sqcup\Upsilon_2\sqcup\ldots
\sqcup\Upsilon_k\right)^{\sqcup r}.
$$
From Propositions~2.4 and~2.5 it follows that the
pseudo-manifold~$\tbQ(\Gamma)$ has $rk$ vertices that can be
divided into $k$ groups each consisting of $r$ vertices so that
the link of each vertex of the $l$th groups is isomorphic
to~$\bS(\Upsilon_l)\cong Y'_l$.
\end{proof}

Thus our goal is to find a finite set $S$ and
involutions~$\Phi^1_j$ satisfying the conditions of
Proposition~2.7. First we shall describe the construction in the
following special case.

\begin{hypothesis} Assume that to each vertex of~$Y$ is assigned a
label from a finite set~$\Cal C$ such that

1\textnormal{)} for each vertex $y\in Y$ the vertices of the
subcomplex~$\St y\subset Y$ have pairwise distinct labels;

2\textnormal{)} for each vertex $y\in Y$ the anti-isomorphism
$\chi_y:\St y\to\St \tilde{y}$ preserve the labels of vertices.
\end{hypothesis}

By $W$ we denote the set of all nonempty simplices of~$Y$.
Vertices of the graph~$\Upsilon $ are in one-to-one correspondence
with sequences
$$\sigma^0\subset\sigma^1\subset\ldots\subset\sigma^{n-1},\ \sigma^j\in W,\ \dim\sigma^j=j.$$
We shall use the notation
$u(\sigma^0,\sigma^1,\ldots,\sigma^{n-1})$ for the vertex
corresponding to a sequence
$\sigma^0\subset\sigma^1\subset\ldots\subset\sigma^{n-1}$. In the
sequel under an isomorphism (respectively, an anti-isomorphism) of
simplicial complexes with labeled vertices we shall always mean an
isomorphism (respectively, an anti-isomorphism) preserving the
labels of vertices. For each simplex $\sigma\in W$ the vertices
of~$\St\sigma$ has pairwise distinct labels. Hence the star of a
simplex~$\sigma$ cannot possess an anti-automorphism. Besides, for
arbitrary simplices $\sigma_1,\sigma_2\in W$ there exists not
more than one anti-isomorphism~$\St\sigma_1\to\St\sigma_2$.

For each $j=1,2,\ldots,n$ we construct a finite graph~$G_j$. The
vertex set of the graph~$G_j$ coincides with the set
of~$(j-1)$-dimensional simplices of~$Y$. Suppose $\sigma$ is a
$(j-1)$-dimensional simplex of~$Y$, $y$ is a vertex of~$\sigma$;
then we connect the vertices $\sigma$ and $\chi_y(\sigma)$ by an
edge of the graph $G_j$. Thus for each vertex $\sigma$ of $G_j$
there are exactly $j$ edges of $G_j$ entering~$\sigma$. (The graph
$G_j$ may contain multiple edges.)

\begin{propos}
The graph $G_j$ does not contain a cycle of odd length.
\end{propos}
\begin{proof}
If two $(j-1)$-dimensional simplices of~$Y$ are connected by an
edge, then their stars are anti-isomorphic. On the other hand, the
star of any simplex of~$Y$ does not possess an anti-automorphism.
Hence the graph $G_j$ cannot contain a cycle of odd length.
\end{proof}
\begin{corr}
Any connected component of $G_j$ is a bipartite graph with the
same number of vertices in both parts.
\end{corr}

By $G$ we denote the disjoint union of the
graphs~$G_1,G_2,\ldots,G_n$. The vertex set of the graph~$G$
coincides with the set $W$. By $P$ we denote the set of all
involutions $\Lambda:W\to W$ such that for each $\sigma\in W$ the
vertices~$\sigma$ and~$\Lambda(\sigma)$ lie in different parts of
the same connected component of the graph~$G$. Corollary~2.1
implies that the set~$P$ is nonempty. It follows immediately from
the definition of the graph~$G$ that the stars of the simplices
$\sigma$ and~$\Lambda(\sigma)$ are anti-isomorphic for
any~$\Lambda\in P$, $\sigma\in W$. Also we notice that
$\Lambda(y)=\tilde{y}$ for any involution~$\Lambda\in P$ and any
vertex~$y$ of the complex~$Y$. We put,
$$
S=P\times\Bbb Z_2^n,\qquad V=U\times S,\qquad
V_{\pm}=U_{\pm}\times S.
$$
To define the involutions~$\Phi_j^1$ we need the following
auxiliary constructions.

For each simplex $\sigma\in W$ we denote by~$c(\sigma)$ the set of
labels of all vertices of $\sigma$. We shall say that the
set~$c(\sigma)$ is the {\it label} of the simplex~$\sigma$. The
set $c(\sigma)$ consists of exactly $\dim\sigma+1$ elements
because the vertices of $\sigma$ have pairwise distinct labels.
The mapping $c$ can be interpreted as a simplicial
mapping~$Y\to\Delta^{|\Cal C|-1}$. Suppose $c\subset\Cal C$ is an
arbitrary subset. Let $W_c\subset W$ be the subset consisting of
all simplices~$\sigma$ such that the complex~$\St\sigma$ contains
(a unique) simplex $\rho$ with $c(\rho)=c$. Suppose that
$\Lambda\in P$. We define an involution $\Lambda_c:W_c\to W_c$ in
the following way. Let $\sigma\in W_c$ be an arbitrary simplex. We
consider a unique simplex $\rho$ such that $\rho$ is contained
in~$\St\sigma$ and $c(\rho)=c$. For $\Lambda_c(\sigma)$ we take
the image of the simplex $\sigma$ under the unique
anti-isomorphism~$\St\rho\to\St\Lambda(\rho)$.

For any subset $c\subset\Cal C$ we define a mapping $\Theta_c:P\to
P$ by
$$
\Theta_c(\Lambda)(\sigma)=\left\{
\begin{aligned}
(\Lambda_c\circ \Lambda&\circ\Lambda_c)(\sigma),&&\text{if
} c(\sigma)\supset c;\\
\Lambda&(\sigma),&&\text{if } c(\sigma)\not\supset c.
\end{aligned}
\right.
$$
It is easy to check that~$\Theta_c(\Lambda)\in P$.

Now we can define the mappings
$$
\Phi^1_j:U\times P\times\Bbb Z^n_2\to U\times P\times\Bbb Z^n_2.
$$
By definition, we put,
$$
\Phi^1_j\left(u\left(\sigma^0,\sigma^1,\ldots,\sigma^{n-1}\right),\Lambda,g\right)=
\left( u\left( \Lambda_c(\sigma^0), \Lambda_c(\sigma^1),\ldots,
\Lambda_c(\sigma^{n-1})\right),
\Theta_c(\Lambda),g+\varepsilon_j  \right),
$$
where $c=c(\sigma^{j-1})$.

\begin{propos}
The mappings $\Phi^1_j$ are well-defined involutions interchanging
the sets~$V_+$ and~$V_-$ and satisfying
conditions~~1\textnormal{)},~2\textnormal{)}, and~3{)} of
Proposition~2.6.
\end{propos}
\begin{proof}
Obviously, for any~$\Lambda\in P$, $c\subset \Cal C$ the
mapping~$\Lambda_c:W_c\to W_c$ preserves the dimensions and the
labels of simplices and the inclusion relation. (Notice that the
initial mapping~$\Lambda:W\to W$ does not preserve the inclusion
relation.) Therefore the mappings~$\Phi^1_j$ are well defined.
Condition~2) and the equalities $\Phi^1_j(V_+)=V_-$ and
$\Phi^1_j(V_-)=V_+$ follow immediately from the fact that the
simplex~$\Lambda_{c}(\sigma^l)$ is the image of~$\sigma^l$ under
the anti-isomorphism $\St\sigma^{j-1}\to\St\Lambda(\sigma^{j-1})$.
For a mapping  $h$ satisfying condition~3) we can take the
projection onto the last multiple $U\times P\times\Bbb
Z_2^n\to\Bbb Z_2^n$.

If $c_1\subset c_2\subset\Cal C$, then $W_{c_2}\subset W_{c_1}$.
Hence for any involution~$\Lambda\in P$ the mapping
$\Lambda_{c_1}\circ\Lambda_{c_2}\circ\Lambda_{c_1}:W_{c_2}\to
W_{c_2}$ is well defined. To prove that the mappings~$\Phi^1_j$
are involutions and commute we need the following auxiliary
propositions.

\begin{propos}
Suppose $c_1\subset c_2\subset\Cal C$, $\Lambda\in P$. Then
\begin{gather*}
\left(\Theta_{c_2}(\Lambda)\right)_{c_1}=\Lambda_{c_1};\\
\left(\Theta_{c_1}(\Lambda)\right)_{c_2}=\Lambda_{c_1}\circ
\Lambda_{c_2}\circ\Lambda_{c_1};\\
\left(\Theta_{c_1}(\Lambda)\right)_{c_2}\circ\Lambda_{c_1}=
\left(\Theta_{c_2}(\Lambda)\right)_{c_1}\circ\Lambda_{c_2}.
\end{gather*}
\end{propos}
\begin{proof}
The mapping $\Lambda_c$ preserves the labels of simplices and the
inclusion relation. Hence we suffice to prove the first equality
for simplices $\sigma$ such that $c(\sigma)=c_1$ and we suffice to
prove the second equality for simplices $\sigma$ such that
$c(\sigma)=c_2$.

If $c(\sigma)=c_1$, then
$$
\left(\Theta_{c_2}(\Lambda)\right)_{c_1}(\sigma)=
\Theta_{c_2}(\Lambda)(\sigma)=\Lambda(\sigma)=\Lambda_{c_1}(\sigma).
$$

If $c(\sigma)=c_2$, then
$$
\left(\Theta_{c_1}(\Lambda)\right)_{c_2}(\sigma)=
\Theta_{c_1}(\Lambda)(\sigma)=
\left(\Lambda_{c_1}\circ\Lambda\circ\Lambda_{c_1}\right)(\sigma)
=\left(\Lambda_{c_1}\circ\Lambda_{c_2}\circ\Lambda_{c_1}\right)(\sigma).
$$

The third equality is a straightforward consequence of the first
two.
\end{proof}

\begin{propos}
If $c_1\subset c_2$, then
$\Theta_{c_1}\circ\Theta_{c_2}=\Theta_{c_2}\circ\Theta_{c_1}$.
\end{propos}
\begin{proof}
Suppose that $\Lambda\in P$, $\sigma\in W$. If
$c(\sigma)\not\supset c_1$, then
$$
\Theta_{c_2}\left(\Theta_{c_1}(\Lambda)\right)(\sigma)=\Lambda(\sigma)
=\Theta_{c_1}\left(\Theta_{c_2}(\Lambda)\right)(\sigma).
$$
If $c(\sigma)\supset c_1$ and $c(\sigma)\not\supset c_2$, then
\begin{gather*}
\Theta_{c_1}\left(\Theta_{c_2}(\Lambda)\right)(\sigma)=
\left(\left(\Theta_{c_2}(\Lambda)\right)_{c_1}\circ
\Theta_{c_2}(\Lambda)\circ
\left(\Theta_{c_2}(\Lambda)\right)_{c_1}\right)(\sigma)=
\left(\Lambda_{c_1}\circ\Lambda\circ\Lambda_{c_1}\right)
(\sigma);\\
\Theta_{c_2}\left(\Theta_{c_1}(\Lambda)\right)(\sigma)=
\Theta_{c_1}(\Lambda)(\sigma)=
\left(\Lambda_{c_1}\circ\Lambda\circ\Lambda_{c_1}\right)(\sigma).
\end{gather*}
If $c(\sigma)\supset c_2$, then
\begin{multline*}
\Theta_{c_1}\left(\Theta_{c_2}(\Lambda)\right)(\sigma)=
\left(\left(\Theta_{c_2}(\Lambda)\right)_{c_1}\circ
\Theta_{c_2}(\Lambda)\circ
\left(\Theta_{c_2}(\Lambda)\right)_{c_1}\right)(\sigma)=\\
\shoveright{=\left(
\Lambda_{c_1}\circ\Lambda_{c_2}\circ\Lambda\circ\Lambda_{c_2}\circ
\Lambda_{c_1}\right)(\sigma);}\\
\shoveleft{
\Theta_{c_2}\left(\Theta_{c_1}(\Lambda)\right)(\sigma)=
\left(\left(\Theta_{c_1}(\Lambda)\right)_{c_2}\circ
\Theta_{c_1}(\Lambda)\circ\left(\Theta_{c_1}(\Lambda)\right)_{c_2}
\right)(\sigma)=}\\
=\left(\left(\Lambda_{c_1}\circ\Lambda_{c_2}\circ\Lambda_{c_1}\right)
\circ\left(\Lambda_{c_1}\circ\Lambda\circ\Lambda_{c_1}\right)\circ
\left(\Lambda_{c_1}\circ\Lambda_{c_2}\circ\Lambda_{c_1}\right)\right)
(\sigma)=\\
=\left(\Lambda_{c_1}\circ\Lambda_{c_2}\circ\Lambda
\circ\Lambda_{c_2}\circ\Lambda_{c_1}\right)(\sigma).
\end{multline*}
\end{proof}

The equality $\left(\Theta_c(\Lambda)\right)_c=\Lambda_c$ easily
implies that the mappings~$\Phi_j^1$ are involutions. The third
equality of Proposition~2.10 and Proposition~2.11 imply that this
involutions commute.
\end{proof}

Thus the set~$S$ and the involutions~$\Phi_j^1$ satisfy the
conditions of Proposition~2.7. Therefore $\tbQ(\Gamma)$ is the
required cubic cell pseudo-manifold. The agreement
conditions~1)--3) of Theorem~2.3 follow immediately from the
construction.

Now we consider the general case and reduce it to the special
case considered above. Suppose that the pseudo-manifold~$Y$ has
$q$ vertices. Put,~$\Cal C=\{ 1,2,\ldots,q\}$. By~$\Cal B$ we
denote the set of bijections $\V(Y)\to\Cal C$. For each
vertex~$y$ of the pseudo-manifold~$Y$ the
anti-isomorphism~$\chi_y$ induces the bijection~$\V(\St y)\to
\V(\St \tilde{y})$. Extend arbitrarily these bijections to
bijections $\varkappa_y:\V(Y)\to \V(Y)$ such that
$\varkappa_{\tilde{y}}=\varkappa_y^{-1}$. Put,
$\overline{Y}=Y\times\Cal B$. To each vertex $(y,\nu)\in
\V(Y)\times\Cal B$ we assign the label $\nu(y)\in\Cal C$. We
define the involution
$$
\overline{\lambda}:\V(\overline{Y})\to \V(\overline{Y})
$$
by
$$
\overline{\lambda}(y,\nu)=(\tilde{y},\nu\circ\varkappa_y^{-1}).
$$
Let $\overline{\chi}_{(y,\nu)}:\St(y,\nu)\to
\St\overline{\lambda}(y,\nu)$ be the anti-isomorphism induced by
the anti-isomorphism~$\chi_y$. Then the
pseudo-manifold~$\overline{Y}$ satisfy conditions~1) and~2) of
Hypothesis. To obtain the required pseudo-manifold~$X$ we should
apply the construction described above to the
pseudo-manifold~$\overline{Y}$.

\subsection{Example}
Let $Y_1=Y_2$ be the boundary of a triangle. We put~$\Cal
C=\{1,2,3\}$. We label the vertices of~$Y_1$ by $1,2,3$ in the
clockwise order and we label the vertices of~$Y_2$ by $1,2,3$ in
the counterclockwise order. The pseudo-manifold~$Y=Y_1\sqcup Y_2$
satisfy Hypothesis in~\S~2.5. The construction described in \S~2.5
yields the disjoint union of two surfaces of genus two each with
cubic cell decomposition shown in Fig.~1. In this figure segments
marked by identical numbers should be glued so as to obtain an
orientable surface.

\begin{center}
{\unitlength=1mm
\begin{picture}(48,60)
\put(0,9){
\begin{picture}(48,48)

\put(0,24){\line(1,0){48}}

\put(12,4){\line(3,5){24}}

\put(36,4){\line(-3,5){24}}

\put(0,24){\line(1,3){4}}

\put(4,36){\line(1,1){8}}

\put(12,44){\line(3,1){12}}

\put(24,48){\line(3,-1){12}}

\put(36,44){\line(1,-1){8}}

\put(44,36){\line(1,-3){4}}

\put(0,24){\line(1,-3){4}}

\put(4,12){\line(1,-1){8}}

\put(12,4){\line(3,-1){12}}

\put(24,0){\line(3,1){12}}

\put(36,4){\line(1,1){8}}

\put(44,12){\line(1,3){4}}

\put(-0.5,30){1}

\put(6,40.5){6}

\put(17,47){2}

\put(47,30){2}

\put(40.5,40.5){3}

\put(30,47){1}

\put(-0.5,16){5}

\put(6,5){6}

\put(17,-1.5){4}

\put(47,16){4}

\put(40.5,5){3}

\put(30,-1.5){5}

\end{picture}
}

\put(16,0){Figure 1.}
\end{picture}
}

\end{center}

\subsection{Operator of barycentric subdivision}
In this section we shall obtain several results on the chain
complex~$\Cal T_*$. In \S~2.8 these results will be used to
construct the combinatorial manifold~$K$ in Theorem~1.1. We recall
that the definition of the complex~$\Cal T_*$ was given in~\S 1.

We define the linear {\it operator of barycentric subdivision}
$\bd:\Cal T_*\to\Cal T_*$ by putting $\bd Y=Y'$ for each
combinatorial sphere $Y$.

\begin{propos}
The mapping $\bd$ is a chain mapping modulo elements of order~$2$,
that is, $2(\diff\bd-\bd\diff)\xi=0$ for any $\xi\in\Cal T_*$.
\end{propos}
\begin{proof}
Let $Y$ be an oriented combinatorial sphere. If $x$ is the
barycenter of a positive-dimensional simplex of~$Y$, then the link
of~$x$ in the complex~$Y'$ possesses an anti-automorphism and
hence is an element of order $2$ in~$\Cal T_*$. If $x$ is a vertex
of $Y$ then the link of~$x$ in~$Y'$ is isomorphic to the
barycentric subdivision of the link of~$x$ in~$Y$. Consequently,
$$
\diff (Y')=\sum_{x\in \V(Y)}\left(\Lk_Yx\right)'+\text{elements
of order 2}.
$$
\end{proof}

There is a standard construction assigning to each simplicial
complex~$Y$ a simplicial decomposition~$Z$ of the
cylinder~$Y\times [0,1]$ such that the restriction of~$Z$ to the
lower base of the cylinder is $Y$ and the restriction of~$Z$ to
the upper base of the cylinder~$Y\times 1$ is $Y'$. For each
simplex $\sigma$ of~$Y$ we denote its barycenter by~$b(\sigma)$.
Then $Z$ is the simplicial complex consisting of the simplices
spanned by all sets
$$
(x_1,0),\ldots,(x_k,0),(b(\sigma_1),1),\ldots,(b(\sigma_l),1),
$$
where $x_1,\ldots,x_k$ are pairwise distinct vertices of~$Y$
spanning a~$(k-1)$-dimensional simplex $\tau$ and
$\sigma_1,\ldots,\sigma_l$ are pairwise distinct simplices of~$Y$
such that $\tau\subset\sigma_1\subset\ldots\subset\sigma_l$.

Now let $Y$ be an $(n-1)$-dimensional combinatorial sphere. Then
$Z$ is a piecewise linear triangulation of the cylinder
$S^{n-1}\times [0,1]$ with the boundary consisting of two
connected components one of which is isomorphic to~$Y$ and the
other is isomorphic to~$Y'$. We attach to these components the
cones~$\cone(Y)$ and~$\cone(Y')$ respectively. As a result we
obtain an $n$-dimensional combinatorial sphere, which will be
denoted by~$\hY$.

The combinatorial sphere $\hY$ has vertices of four types

1) the cone vertex $u_0$ of $\cone(Y)$;

2) vertices $(x,0)$, where $x$ is a vertex of $Y$;

3) vertices $(b(\sigma),1)$, where $\sigma$ is a simplex of~$Y$;

4) the cone vertex $u_1$ of $\cone(Y')$.

If the combinatorial sphere~$Y$ is oriented, we endow the
combinatorial sphere~$\hY$ with the orientation such that $\Lk
u_1\cong Y'$ and $\Lk u_0\cong -Y$.

\begin{propos}
The mapping $\bd$ is chain homotopic to the identity mapping
modulo elements of order~$2$, that is, there is a linear mapping
$D:\Cal T_*\to \Cal T_*$ of degree~$1$ such that
$$
\bd\xi-\xi=\diff D\xi+D\diff\xi+\text{\textnormal{elements of
order 2}}
$$
for any $\xi\in\Cal T_*$.
\end{propos}
\begin{proof}
For each combinatorial sphere $Y$ we put~$DY=\widehat{Y}$. The
link of each vertex~$(x,0)$ of a combinatorial
sphere~$\widehat{Y}$ is anti-isomorphic to $\widehat{\Lk_Yx}$. The
link of each vertex~$(b(\sigma),1)$ of~$\widehat{Y}$ possesses an
anti-automorphism and hence is an element of order $2$ in~$\Cal
T_*$. Therefore,
\begin{multline*}
\diff DY=Y'-Y-\sum_{x\in \V(Y)}\widehat{\Lk_Yx}+\text{elements of
order 2}=\\
=\bd Y-Y-D\diff Y+\text{elements of order 2}.
\end{multline*}
\end{proof}
\begin{corr}
The mapping $(2\bd)_*:H_*(\Cal T_*)\to H_*(\Cal T_*)$ coincides
with the multiplication by~$2$.
\end{corr}
\subsection{Construction of the combinatorial manifold $K$}
In this section we use the construction of the cubic cell
complex~$X$ given in \S~2.5 to obtain an explicit construction of
the combinatorial manifold $K$ in Theorem~1.1.

We consider a balanced set $Y_1,Y_2,\ldots,Y_k$ of oriented
combinatorial spheres and apply to it the construction described
in~\S 2.5. The obtained pseudo-manifold~$X$ is a cubic cell
combinatorial manifold because the links of all its vertices are
combinatorial spheres. The set of links of vertices of~$X'$
consists of the $r$-fold multiple of the set
$Y_1'',Y_2'',\ldots,Y_k''$ and several combinatorial spheres each
of which possesses an anti-automorphism. Therefore,
$$
\diff(X')=r\sum_{i=1}^kY_i''+\text{elements of order 2}.
$$
The chain $Y_1+Y_2+\ldots+Y_k$ is a cycle of the complex~$\Cal
T_*$. Hence Propositions~2.12 and~2.13 imply that
\begin{gather*}
\diff\sum_{i=1}^k\hY_i=\sum_{i=1}^kY_i'-\sum_{i=1}^kY_i+\text{elements of order 2};\\
\diff\sum_{i=1}^k\widehat{Y_i'}=\sum_{i=1}^kY_i''-\sum_{i=1}^kY_i'+\text{elements
of order 2}.
\end{gather*}
We put,
$$
K=X'\sqcup X'\sqcup
\left(\left(-\widehat{Y}_1\right)\sqcup\ldots\sqcup
\left(-\widehat{Y}_k\right)\right)^{\sqcup\, 2r}\sqcup
\left(\left(-\widehat{Y_1'}\right)\sqcup
\ldots\sqcup\left(-\widehat{Y_k'}\right)\right)^{\sqcup\, 2r}.
$$
Then
$$
\diff K=2r\sum_{i=1}^kY_i.
$$
Therefore $K$ is a required combinatorial manifold.

\subsection{Small covers}
Small covers are the $\Bbb Z_2$-analogue of quasi-toric manifolds
(see.~\cite{DaJa91},~\cite{BuPa04}). A {\it small cover} is a
smooth manifold $M^n$ with a locally standard action of the group
$\Bbb Z_2^n$ such that $M^n/\Bbb Z_2^n=P^n$ is a simple convex
polytope. (An action is {\it locally standard} if locally it can
be modeled by the standard action of~$\Bbb Z_2^n$ by reflections
on~$\Bbb R^n$.) In 1991 M.\,Davis and
T.\,Januszkiewicz~\cite{DaJa91} proved that any small cover can be
obtain by a certain standard construction. In this section we
describe the Davis-Januszkiewicz construction and point out the
connection of this construction with our construction described
in~\S 2.5.

Suppose $P$ is an $n$-dimensional simple convex polytope with $m$
facets $F_1,F_2,\ldots,F_m$, $Y$ is the boundary of the dual
simplicial polytope. Let $(a_1,a_2,\ldots,a_m)$ be the basis of
the group~$\Bbb Z_2^m$. A {\it characteristic function} is an
arbitrary homomorphism $\lambda:\Bbb Z_2^m\to\Bbb Z_2^n$ such that
the elements
$\lambda(a_{i_1}),\lambda(a_{i_2}),\ldots,\lambda(a_{i_k})$ are
linearly independent whenever the intersection $F_{i_1}\cap
F_{i_2}\cap\ldots\cap F_{i_k}$ is nonempty. Let $F=F_{i_1}\cap
F_{i_2}\cap\ldots\cap F_{i_k}$ be a face of~$P$.
By~$G(F)\subset\Bbb Z_2^n$ we denote the subgroup generated by the
elements
$\lambda(a_{i_1}),\lambda(a_{i_2}),\ldots,\lambda(a_{i_k})$. For
any point~$x\in P$ by $F(x)$ we denote a unique face of $P$ whose
relative interior contains~$x$. We put
$$
M_{P,\lambda}=P\times\Bbb Z_2^n/\sim,
$$
where $(x_1,g_1)\sim(x_2,g_2)$ if and only if $x_1=x_2$ and
$g_1^{-1}g_2\in G(F(x))$. Then $M_{P,\lambda}$ is a manifold
decomposed into cells each of which is isomorphic to the
polytope~$P$. Let us consider the cell decomposition
$X_{P,\lambda}$ dual to the obtained decomposition
of~$M_{P,\lambda}$. It is easy to check that $X_{P,\lambda}$ is a
cubic cell combinatorial manifold and the link of each vertex
of~$X_{P,\lambda}$ is isomorphic (or anti-isomorphic) to~$Y$.
Unfortunately, the construction of the manifold~$X_{P,\lambda}$
does not answer Question~1.2. The matter is that the complex
$X_{P,\lambda}$ has~$2^n$ vertices among which there are
$2^{n-1}$ vertices with links isomorphic to $Y$ and $2^{n-1}$
vertices with links anti-isomorphic to $Y$.

Now let us show that in a certain special case the cubic
decomposition~$X_{P,\lambda}$ can be obtained by a certain
analogue of Proposition~2.7. We shall consider the case of
so-called {\it manifolds induced from linear models}
(see~\cite{DaJa91}). A small cover~$M_{P,\lambda}$ is called a
manifold induced from a linear model if the function~$\lambda$
takes each vertex~$a_i$ to an element of a chosen basis of~$\Bbb
Z_2^n$. Such characteristic function yields a regular colouration
of vertices of~$Y$ in $n$ colours. Suppose $\Upsilon$ is the
corresponding homogeneous graph of degree $n$, $U$ is its vertex
set, $\Phi_j$ are the involutions giving the graph~$\Upsilon$. Let
us proceed as in Proposition~2.7. For a certain finite set $S$ we
put
$$
V=U\times S,\qquad \Phi_j^0(u,s)=(\Phi_j(u),s).
$$
We shall look for involutions~$\Phi_j^1:V\to V$ satisfying
conditions~1)--3) of Proposition~2.6. The only difference with
Proposition~2.7 is that we do not want to carry about the
orientations of links of vertices of the complex obtained. Hence
the involutions $\Phi_j^1$ need not satisfy the conditions
$\Phi^1_j(V_+)=V_-$  and $\Phi^1_j(V_-)=V_+$. In this case all
difficulties appearing in \S 2.5 can be easily got around. Put
$S=\Bbb Z_2^n$ and
$$
\Phi_j^1(u,s)=(u,s+\varepsilon_j),
$$
where $(\varepsilon_1,\varepsilon_2,\ldots,\varepsilon_n)$ is a
basis of the group~$\Bbb Z_2^n$. Denote by~$\Gamma$ the obtained
homogeneous graph of degree~$2n$. It is easy to check
that~$\tbQ(\Gamma)\cong X_{P,\lambda}$.

\section{Cobordisms of manifolds with singularities}
\subsection{Classes of pseudo-manifolds}
We recall that under a pseudo-manifold we always mean a normal
pseudo-manifold (see~\S 2.1). In this section all pseudo-manifolds
are supposed to be simplicial pseudo-manifolds if otherwise is not
stated. Suppose $\CC$ is a class of oriented pseudo-manifolds
satisfying the following axioms.

\begin{ax1} A zero-dimensional pseudo-manifold belongs to the
class~$\CC$ if and only if it is a zero-dimensional sphere. All
positive-dimensional pseudo-manifolds belonging to~$\CC$ are
connected.
\end{ax1}

\begin{ax2} Suppose that a pseudo-manifold $Y_1$ belongs to $\CC$
and a pseudo-manifold~$Y_2$ is piecewise linearly homeomorphic
to~$Y_2$ with the homeomorphism either preserving or reversing the
orientation. Then $Y_2$ also belongs to $\CC$.
\end{ax2}

\begin{ax3}
Suppose that an $n$-dimensional pseudo-manifold~$Y$ belongs
to~$\CC$, $\sigma$ is a simplex of~$Y$ of dimension less than~$n$.
Then the pseudo-manifold~$\Lk\sigma$ belongs to~$\CC$.
\end{ax3}

\begin{ax4}
If a pseudo-manifold~$Y$ belongs to~$\CC$, then the unreduced
suspension~$\Sigma Y$ belongs to~$\CC$.
\end{ax4}

By~$\PM$ we denote the class consisting of all zero-dimensional
spheres and all connected oriented pseudo-manifolds of positive
dimension. $\PM$ is the maximal class satisfying Axioms~I--IV. On
the other hand, the axioms immediately imply that every
combinatorial sphere belongs to~$\CC$. Therefore the class $\CS$
of all oriented combinatorial spheres is the minimal class
satisfying Axioms~I--IV.

Let us give several examples of classes satisfying Axioms~I--IV.

1) Suppose $P_1,P_2,\ldots$ is a either finite or infinite
sequence of combinatorial manifolds. We define the class
$\CC(P_1,P_2,\ldots)$ to be the minimal class generated by the
manifolds~$P_1,P_2,\ldots$, that is, the minimal class
containing~$P_1,P_2,\ldots$ and satisfying axioms~I--IV. The class
$\CC(P_1,P_2,\ldots)$ consists of all combinatorial spheres and
all pseudo-manifolds $X$ that are piecewise linearly homeomorphic
to an iterated unreduced suspension over a join
$P_{i_1}*P_{i_2}*\ldots*P_{i_k}$, $i_1<i_2<\ldots<i_k$, with the
homeomorphism either preserving or reversing the orientation.

2) A simplicial complex is called an {\it $n$-dimensional
simplicial homology manifold} if the link of every its
$k$-dimensional simplex has the homology of
an~$(n-k-1)$-dimensional sphere. A simplicial homology manifold
with the homology of a sphere is called a {\it simplicial homology
sphere}. The class~$\HS$ of all oriented simplicial homology
spheres satisfies Axioms~I--IV. Considering the homology with
coefficients in an Abelian group~$A$ instead of the integral
homology, we shall similarly obtain the class~$\HS(A)$ of all
oriented simplicial $A$-homology spheres. The class~$\HS(A)$ also
satisfies Axioms~I--IV.

3) In~\cite{Che80},~\cite{Che83} J.\,Cheeger built the $L_2$ Hodge
theory for the so-called {\it pseudo-manifolds with negligible
boundary}. An $n$-dimensional pseudo-manifold~$X$ can be endowed
with a locally flat metrics whose restriction to each simplex
coincides with the Euclidean metrics on a regular simplex with
edge~$1$. Then $X\setminus X^{n-2}$ is an incomplete Riemannian
manifold, where~$X^{n-2}$ is the $(n-2)$-skeleton of~$X$.  The
ring of $L_2$-cohomology of $X\setminus X^{n-2}$ is denoted by
$H_{(2)}^*(X\setminus X^{n-2})$ (see~\cite{GMPC83}). By
definition, put $H_{(2)}^*(X)=H_{(2)}^*(X\setminus X^{n-2})$. The
defined ring~$H_{(2)}^*(X)$ is invariant under piecewise linear
homeomorphisms~\cite{Che80}. A pseudo-manifold $X$ is said to be a
{\it pseudo-manifold with negligible boundary} if for any simplex
$\sigma$ of~$X$ we have $H^k_{(2)}(\Lk\sigma)=0$ whenever
$\dim\Lk\sigma=2k>0$. This condition is important because it
ensures that the strong closures of the operators~$d$ and~$\delta$
on the space of $L_2$-forms on $X\setminus X^{n-2}$ are conjugate.
(see~\cite{Che80},~\cite{GMPC83}). By $\Ch$ we denote the class
consisting of all zero-dimensional spheres, all connected
odd-dimensional pseudo-manifolds with negligible boundary, and all
connected even dimensional pseudo-manifolds $X$ with negligible
boundary such that $H^m_{(2)}(X)=0$, where $\dim X=2m$. It can be
immediately checked that the class~$\Ch$ satisfies Axioms~I--IV.

In~\cite{Gai04} the author constructed a chain complex~$\Cal T_*$
of oriented combinatorial spheres (for a definition see~\S 1 of
the present paper). Now we shall describe a generalization of
this construction for the case of an arbitrary class~$\CC$.

For the sake of simplicity we shall denote the isomorphism class
of a pseudo-manifold by the same letter as the pseudo-manifold
itself. Usually we shall not make difference between a
pseudo-manifold and its isomorphism class.

Let us fix a positive integer~$n$. We consider the free Abelian
group generated by all isomorphism classes of oriented
$(n-1)$-dimensional pseudo-manifolds of the class~$\CC$. We take
the quotient of this group by the relations $Y+(-Y)=0$, where $-Y$
is a pseudo-manifold~$Y$ with the reversed orientation. The group
obtained is denoted by~$\Cal T_n^{\CC}$. The group~$\Cal
T_n^{\CC}$ can be decomposed into the direct sum of groups each
isomorphic either to~$\Bbb Z$ or to~$\Bbb Z_2$. Summands~$\Bbb Z$
correspond to pseudo-manifolds that do not possess
anti-automorphisms and summands $\Bbb Z_2$ correspond to
pseudo-manifolds that possesses anti-automorphisms. It is
convenient to suppose that $\Cal T_0^{\CC}\cong\Bbb Z$ is the free
cyclic group generated by~$\emptyset$.

We introduce a differential
$$
\diff:\Cal T_n^{\CC}\to\Cal T_{n-1}^{\CC}
$$
by putting
$$
\diff Y=\sum_{y\in \V(Y)} \Lk y,
$$
where each pseudo-manifold $\Lk y$ is endowed with the
orientation induced by the orientation of~$Y$. (The differential
$\diff:\Cal T_1^{\CC}\to\Cal T_{0}^{\CC}$ is trivial.) It is easy
to check that~$\diff^2=0$. Thus $\Cal T^{\CC}_*$ is a chain
complex.

For an arbitrary Abelian group~$A$ we put,
$$
\Cal T^*_{\CC}(A)=\Hom(\Cal T^{\CC}_*,A).
$$
Elements of~$\Cal T^n_{\CC}(A)$ are $A$-valued functions $f$ on
the set of isomorphism classes of $(n-1)$-dimensional
pseudo-manifolds belonging to~$\CC$ such that $f(-Y)=-f(Y)$. The
differential
$$
\delta:\Cal T_{\CC}^n(A)\to\Cal T_{\CC}^{n+1}(A)
$$
is given by
$$
(\delta f)(Y)=(-1)^n\sum_{y\in \V(Y)}f(\Lk y).
$$

Now we assume that a class $\CC$ satisfy the following
multiplicative axiom.
\begin{ax5}
If pseudo-manifolds~$Y_1$ and~$Y_2$ belong to~$\CC$, then the
join~$Y_1*Y_2$ belongs to~$\CC$.
\end{ax5}

Then the join operation yields the bilinear multiplication
$$
*:\Cal T_m^{\CC}\times\Cal T_n^{\CC}\to \Cal T_{m+n}^{\CC},
$$
which endows the group~$\Cal T_*^{\CC}$ with the structure of a
commutative graded ring. We have the Leibniz formula
$$
\diff (Y_1*Y_2)=\diff Y_1*Y_2+(-1)^mY_1*\diff Y_2.
$$
Hence $H_*(\Cal T_*^{\CC})$ is a commutative graded ring.

Notice that the classes $\PM$, $\CS$, $\HS(A)$, $\Ch$ satisfy
Axiom~V, while the classes~$\CC(P_1,P_2,\ldots)$ need not
necessarily satisfy Axiom~V.

\subsection{$\CC$-cobordism groups}
Suppose $\CC$ is a class of pseudo-manifolds satisfying
Axioms~I--IV. A {\it manifold with singularities of class $\CC$}
(or a {\it $\CC$-manifold}) is a pseudo-manifold~$K$ such that the
links of all vertices of~$K$ belong to~$\CC$. In particular,
manifolds with singularities of class~$\PM$ are normal
pseudo-manifolds, manifolds with singularities of class $\CS$ are
combinatorial manifolds, manifolds with singularities of class
$\HS$ are simplicial homology manifolds, manifolds with
singularities of class $\Ch$ are pseudo-manifolds with negligible
boundary. By~$\tCC$ we denote the class of all oriented (not
necessarily connected) manifolds with singularities of
class~$\CC$. Axiom~III implies that~$\CC\subset\tCC$.

Suppose $K=K_1\sqcup\ldots\sqcup K_k$ and $L=L_1\sqcup\ldots\sqcup
L_l$ are $n$-dimensional manifolds with singularities of
class~$\CC$, where the complexes $K_i$ and $L_i$ are connected. We
shall say that the $\CC$-manifolds $K$ and~$L$ are $\CC$-cobordant
if there is a closed oriented simplicial pseudo-manifold~$Z$ with
vertices~$x_1,\ldots,x_k,y_1,\ldots,y_l,z_1,\ldots,z_m$ such that
$\Lk x_i\cong K_i$ for $i=1,\ldots,k$, $\Lk y_i\cong -L_i$ for
$i=1,\ldots,l$, and $\Lk z_i\in\CC$ for $i=1,\ldots,m$.
By~$\Omega_n^{\CC}$ we denote the corresponding cobordism
semigroup. Since the class~$\CC$ satisfies Axiom~IV, we see that
the pseudo-manifold $\Sigma K$ yields a $\CC$-cobordism between
the pseudo-manifold $K\sqcup (-K)$ and the empty pseudo-manifold.
Hence the semigroup~$\Omega_n^{\CC}$ is a group. The
group~$\Omega_0^{\CC}$ is the free cyclic group generated by the
cobordism class of a point.

\begin{remark}
The definition of $\CC$-cobordism given above has a simple
geometric interpretation. If we delete the regular neighborhoods
of the vertices $x_1,\ldots,x_k,y_1,\ldots,y_l$ from the
pseudo-manifold~$Z$ we shall obtain a manifold with singularities
of class~$\CC$ and the boundary isomorphic to~$K\sqcup (-L)$.
\end{remark}

If we want the notion of $\CC$-cobordism to be geometrically sapid
we need to prove the following proposition.

\begin{propos}
Suppose that $K_1$ and $K_2$ are manifolds with singularities of
class~$\CC$ piecewise linearly homeomorphic to each other with the
homeomorphism preserving the orientation. Then $K_1$ and $K_2$ are
$\CC$-cobordant.
\end{propos}
\begin{proof}
Without loss of generality we may assume that the
pseudo-manifolds~$K_1$ and~$K_2$ are connected. We consider a
piecewise linear triangulation~$Z$ of the polyhedron $\Sigma K_1$
such that the links of the suspension vertices $x$ and $y$ are
isomorphic to the pseudo-manifolds~$K_1$ and $-K_2$ respectively.
Let $z$ be an arbitrary vertex of~$Z$ distinct from~$x$ and~$y$.
We consider the simplex~$\sigma$ of $K_1$ such that the image
of~$z$ under the natural projection $\Sigma K_1\setminus\{x,y\}\to
K_1$ belongs to the relative interior of~$\sigma$. (The relative
interior of a vertex is supposed to be the vertex itself.) Then
the link of~$z$ in~$Z$ is piecewise linearly homeomorphic to the
polyhedron~$\Sigma(\partial\sigma*\Lk_{K_1}\sigma)$. Therefore,
$\Lk z\in\CC$. Thus the pseudo-manifold~$Z$ yields a
$\CC$-cobordism between~$K_1$ and~$K_2$.
\end{proof}

If the class~$\CC$ satisfies Axiom~V, the graded group
$\Omega_*^{\CC}$ can be endowed with a multiplicative structure.
The direct product of two simplicial pseudo-manifolds
$K_1,K_2\in\tCC$ is not a simplicial pseudo-manifold. However, we
can consider an arbitrary piecewise linear triangulation~$K$ of
the polyhedron~$K_1\times K_2$. Suppose $y$ is an arbitrary vertex
of~$K$, $y_i$, $i=1,2$, are the images of~$y$ under the
projections $K_1\times K_2\to K_i$. Let $\sigma_i$ be the simplex
of~$K_i$ whose relative interior contains the point~$y_i$. It is
easy to prove that the link of~$y$ in~$K$ is piecewise linearly
homeomorphic to the complex
$$
\partial\sigma_1*\Lk_{K_1}\sigma_1*\partial\sigma_2*\Lk_{K_2}\sigma_2.
$$
Hence $\Lk y$ belongs to~$\CC$. Therefore $K$ is a manifold with
singularities of class~$\CC$. By Proposition~3.1, the
$\CC$-cobordism class~$[K]$ does not depend on the choice of the
triangulation~$K$. We put~$[K_1][K_2]=[K]$. Thus we obtain a
well-defined multiplication in the graded group~$\Omega_*^{\CC}$
which endows~$\Omega_*^{\CC}$ with the structure of a commutative
associative graded ring.

It follows immediately from the definition that any
pseudo-manifold belonging to~$\CC$ represents a zero cobordism
class in~$\Omega_*^{\CC}$. In particular, we have the following
proposition.
\begin{propos}
We have $\Omega_n^{\PM}=0$ if $n>0$.
\end{propos}

The ring $\Omega^{\CS}_*$ coincides with the oriented piecewise
linear cobordism ring~$\Omega_*^{\SPL}$. Recall that the
ring~$\Omega_*^{\SPL}\otimes\Bbb Q\cong\Omega_*^{\SO}\otimes\Bbb
Q$ is the polynomial ring with generators in each dimension
divisible by~$4$ and
$$
\Hom(\Omega_*^{\SPL},\Bbb
Q)\cong\Hom(\Omega_*^{\SO},\Bbb Q)\cong H^*(\BO;\Bbb Q)\cong\Bbb
Q[p_1,p_2,\ldots],\ \ \deg p_i=4i.
$$

\subsection{Homomorphism $d_*:\Omega_*^{\CC}\to H_*(\Cal T_*^{\CC})$}
We have the following analogue of Theorems~1.1 and~2.1.
\begin{theorem}
Suppose $Y_1,Y_2,\ldots,Y_k\in\CC$ is a balanced set of oriented
$(n-1)$-dimensional pseudo-manifolds. Then there is an
$n$-dimensional \textnormal{(}simplicial\textnormal{)}
manifold~$K$ with singularities of class~$\CC$ whose set of links
of vertices coincides up to an isomorphism with the set
$$\underbrace{Y_1,\ldots,Y_1}_{r},\underbrace{Y_2,\ldots,Y_2}_{r},
\ldots,\underbrace{Y_k,\ldots,Y_k}_{r},Z_1,Z_2,\ldots,Z_l,-Z_1,-Z_2,\ldots,-Z_l$$
for some positive integer~$r$ and some oriented
$(n-1)$-dimensional pseudo-manifolds $Z_1,Z_2,\ldots,Z_l\in\CC$.
\end{theorem}

The construction of the pseudo-manifold~$K$ is quite identical
with the construction of the combinatorial manifold~$K$ in~\S 2.8.
Indeed, all results of \S\S 2.7,~2.8 would hold if we everywhere
replace the words ``combinatorial sphere'' with the words
``pseudo-manifold of the class~$\CC$'', the words ``combinatorial
manifold'' with the words ``manifold with singularities of
class~$\CC$'', and the words ``chain complex~$\Cal T_*$'' with the
words ``chain complex~$\Cal T_*^{\CC}$''. In particular there is
an operator of barycentric subdivision $\bd:\Cal T_*^{\CC}\to \Cal
T_*^{\CC}$. The operator~$\bd$ is a chain mapping modulo elements
of order~$2$ chain homotopic to the identity mapping modulo
elements of order~$2$. The axioms easily imply that the
pseudo-manifold obtained is a manifold with singularities of
class~$\CC$.

Suppose $K$ is an oriented $n$-dimensional manifold with
singularities of class~$\CC$. By~$\diff K$ we denote the element
of the group~$\Cal T_*^{\CC}$ given by
$$
\diff K=\sum_{y\in \V(K)}\Lk y.
$$
\begin{propos}
The correspondence $\diff$ induces a well-defined additive
homomorphism
$$
\diff_*:\Omega_*^{\CC}\to H_*(\Cal T^{\CC}_*).
$$
\end{propos}
\begin{proof}
Obviously, $\diff\diff K=0$, that is, $\diff K$ is a cycle of the
chain complex~$\Cal T_*^{\CC}$. If the pseudo-manifold $K$ belongs
to~$\CC$, then $\diff K$ is a boundary of the complex~$\Cal
T_*^{\CC}$. Let~$Z$ be a $\CC$-cobordism
between~$K=K_1\sqcup\ldots\sqcup K_k$ and~$L=L_1\sqcup\ldots\sqcup
L_l$. We shall use the notation for the vertices of~$Z$ introduced
above. Then
$$
\sum_{i=1}^k\diff\Lk x_i+\sum_{i=1}^l\diff\Lk
y_i+\sum_{i=1}^m\diff\Lk z_i=0.
$$
Therefore,
$$
\diff K-\diff L=\sum_{i=1}^k\diff K_i-\sum_{i=1}^l\diff
L_i=-\sum_{i=1}^m\diff\Lk z_i,
$$
Now $\diff K-\diff L$ is a boundary of the complex~$\Cal
T^{\CC}_*$, since $\Lk z_i\in\CC$ for $i=1,\ldots,m$.
\end{proof}

\begin{theorem}
If the class~$\CC$ satisfies Axioms~I--IV, then the kernel and the
cokernel of the homomorphism $\diff_*:\Omega_*^{\CC}\to H_*(\Cal
T_*^{\CC})$ are torsion groups. If the class~$\CC$ also satisfies
Axiom~V, then $\diff_*$ is a multiplicative homomorphism modulo
elements of order~$2$.
\end{theorem}

The author do not know whether the mapping~$\diff_*$ is
multiplicative.

\begin{corr} If the class~$\CC$ satisfies Axioms~I--V, then the
homomorphism
$$
\diff_*\otimes\Bbb Q:\Omega_*^{\CC}\otimes\Bbb Q\to H_*(\Cal
T_*^{\CC})\otimes\Bbb Q
$$
is a multiplicative isomorphism. In  particular, the homomorphisms
\begin{gather*}
\diff_*\otimes\Bbb Q: \Omega_*^{\SPL}\otimes\Bbb Q\to H_*(\Cal
T^{\CS}_*)\otimes\Bbb
Q;\\
\diff_*\otimes\Bbb Q: \Omega_*^{\HS}\otimes\Bbb Q\to H_*(\Cal
T^{\HS}_*)\otimes\Bbb Q
\end{gather*}
are multiplicative isomorphisms.
\end{corr}
\begin{corr}
The groups $H_n(\Cal T_*^{\PM})$ are torsion groups for~$n>0$.
Thus $H_n(\Cal T_*^{\PM}\otimes\Bbb Q)=0$ and $H^n(\Cal
T^*_{\PM}(\Bbb Q))=0$ for $n>0$.
\end{corr}

\begin{proof}[Proof of Theorem~3.2]
Theorem~3.1 immediately implies that the cokernel of~$\diff_*$ is
a torsion group. The claim that the kernel of~$\diff_*$ is a
torsion group is equivalent to the following proposition.

\begin{propos}
Suppose that a set $K_1,K_2,\ldots,K_k$ of manifolds with
singularities of class~$\CC$ is balanced. Then there is a
positive integer $r$ such that the disjoint union of $r$ copies
of the pseudo-manifold $K=K_1\sqcup\ldots\sqcup K_k$ is
$\CC$-cobordant to zero.
\end{propos}
\begin{proof}
We apply Theorem~2.2 to the balanced set of
pseudo-manifolds~$K_1,K_2,\ldots,K_k$. Let $X$ be the cubic cell
pseudo-manifold obtained. The link of every simplex of every
complex~$K_i'$ belongs to~$\CC$. Hence the link of a vertex $x$ of
$X'$ belongs to~$\CC$ whenever $x$ is the barycenter of a
positive-dimension simplex of~$X$. If $x$ is a vertex of~$X$, then
the link of~$x$ in~$X'$ is isomorphic to~$K_i''$ for some~$i$, and
for each $i$ there are exactly $r$ vertices with the links
isomorphic to~$K_i''$. Hence the pseudo-manifold~$X'$ yields a
$\CC$-cobordism between the pseudo-manifold~$(K'')^{\sqcup r}$ and
the empty pseudo-manifold. To conclude the proof we notice that,
by Proposition~3.1, the pseudo-manifold~$K''$ is $\CC$-cobordant
to~$K$.
\end{proof}

Let us now prove that $\diff_*$ is a multiplicative homomorphism
modulo elements of order~$2$. Suppose $K_1$, $K_2$ are manifolds
with singularities of class~$\CC$. We consider the direct
product~$K_1\times K_2$. It is a complex glued from cells each of
which is the product of two simplices. Then~$K=(K_1\times K_2)'$
is a simplicial manifold with singularities of class~$\CC$ and
$[K]=[K_1][K_2]$ in~$\Omega_*^{\CC}$. Any vertex of~$K$
is~$(b(\sigma_1),b(\sigma_2))$, where $\sigma_1$ is a simplex
of~$K_1$ and~$\sigma_2$ is a simplex of~$K_2$. We have
$$
\Lk_K(b(\sigma_1),b(\sigma_2))\cong
(-1)^{\codim\sigma_1\dim\sigma_2}
(\partial(\sigma_1\times\sigma_2))'*
(\Lk_{K_1}\sigma_1*\Lk_{K_2}\sigma_2)'.
$$
If either $\dim\sigma_1\ne 0$ or $\dim\sigma_2\ne 0$, then the
complex~$(\partial(\sigma_1\times\sigma_2))'$ possesses an
anti-automorphism. Therefore the
pseudo-manifold~$\Lk_K(b(\sigma_1),b(\sigma_2))$ possesses an
anti-automorphism and hence is an element of order $2$. Therefore,
\begin{multline*}
\diff K=\mathop{\sum_{y_1\in \V(K_1)}}\limits_{y_2\in \V(K_2)}
(\Lk_{K_1}y_1*\Lk_{K_2}y_2)'+\text{elements of order 2}=\\
=\bd((\diff K_1)*(\diff K_2))+\text{elements of order 2}.
\end{multline*}
Consequently,
$$
2\diff_*([K_1][K_2])=(2\bd)_*((\diff_* [K_1])*(\diff_*
[K_2]))=2(\diff_* [K_1])*(\diff_* [K_2]).
$$
\end{proof}

\subsection{Local formulae for $\CC$-cobordism invariants}

An {\it additive $n$-dimensional $\CC$-cobordism invariant} is an
additive homomorphism~$q:\Omega_n^{\CC}\to A$, where $A$ is an
Abelian group. The value of the invariant~$q$ on the
$\CC$-cobordism class of a $\CC$-manifold~$K$ will be denoted
by~$q(K)$.

\begin{definition} A {\it local formula} for the invariant~$q$ is
an additive homomorphism~$f:\Cal T_n^{\CC}\to A$ such that
$$
q(K)=\sum_{y\in \V(K)}f(\Lk y)
$$
for any oriented $n$-dimensional manifold $K$ with singularities
of class~$\CC$.
\end{definition}

Proposition~3.3 easily implies the following assertion.

\begin{corr}
An additive homomorphism~$f:\Cal T_n^{\CC}\to A$ is a local
formula for an additive $\CC$-cobordism invariant if and only if
$f$ is a cocycle of the complex~$\Cal T^*_{\CC}(A)=\Hom(\Cal
T_*^{\CC},A)$. Cohomological cocycles yield the same
$\CC$-cobordism invariant.
\end{corr}

Thus we obtain an additive homomorphism
$$
\delta^*:H^*(\Cal T^*_{\CC}(A))\to \Hom(\Omega_*^{\CC},A).
$$
The homomorphism $\delta^*$ is conjugate to the
homomorphism~$\diff_*$. Corollary~3.1 implies the following
assertion.
\begin{corr}
Suppose that a class~$\CC$ satisfies Axioms~I--IV; then the
homomorphism $\delta^*:H^*(\Cal T^*_{\CC}(\Bbb Q))\to
\Hom(\Omega_*^{\CC},\Bbb Q )$ is an isomorphism. Thus any
rational additive $\CC$-cobordism invariant possesses a local
formula, which is unique up to a coboundary of the complex~$\Cal
T^*_{\CC}(\Bbb Q)$.
\end{corr}

Now let us consider in more details the case~$\CC=\CS$. Rational
additive invariants of oriented piecewise linear cobordism are
exactly linear combinations of the Pontryagin numbers. By
Corollary~3.4, we have
$$
H^*(\Cal T^*_{\CS}(\Bbb Q))\cong\Hom(\Omega_*^{\SPL},\Bbb
Q)\cong\Bbb Q[p_1,p_2,\ldots],\ \deg p_i=4i.
$$
This is one of the main results of the author's
paper~\cite{Gai04}. The proof given in~\cite{Gai04} is based on
the theory of block bundles and is rather complicated. In this
paper we have obtained a more direct purely combinatorial proof.

For an arbitrary Abelian group $A$ and even for $A=\Bbb Z$ the
problem of characterization of $\CC$-cobordism invariants
admitting local formulae is still open even in the case~$\CC=\CS$.
For example, in~\cite{Gai04} the author proved that no multiple of
the first Pontryagin class admits an integral local formula.

Let us consider two more examples.

1) The ring $\Omega_*^{\HS}$ is the cobordism ring of simplicial
homology manifolds. It is well known that the rational Pontryagin
classes of simplicial homology manifolds are well defined
(see~\cite{MiSt79}). Hence the Pontryagin numbers of oriented
simplicial homology manifolds are well defined.
C.\,R.\,F.\,Maunder~\cite{Mau71} proved that the Pontryagin
numbers of oriented simplicial homology manifolds are cobordism
invariants. Therefore the Pontryagin numbers admit local
formulae. Thus the cohomology group~$H^*(\Cal T_{\HS}^*(\Bbb Q))$
contains a subgroup additively isomorphic to the polynomial ring
in Pontryagin classes with rational coefficients.
N.\,Martin~\cite{Mar73} proved that $\Omega_4^{\HS}\cong\Bbb
Z\oplus\Theta_3^H$, where $\Theta_3^H$ is the $H$-cobordism
(homology cobordism) group of three-dimensional homology spheres.
M.\,Furuta~\cite{Fur90} proved that the group~$\Theta_3^H$
contains an infinitely generated free Abelian group. Thus the
polynomials in Pontryagin classes do not exhaust the
group~$H^*(\Cal T_{\HS}^*(\Bbb Q))$. We have
$$
H^4(\Cal T_{\HS}^*(\Bbb Q))\cong\Hom(\Omega_4^{\HS},\Bbb
Q)\cong\Bbb Q\oplus\Hom(\Theta_3^H,\Bbb Q).
$$
The direct summand~$\Bbb Q$ is generated by the first Pontryagin
number. Local formulae for the first Pontryagin class and local
formulae for the cobordism invariants corresponding to elements
of~$\Hom(\Theta_3^H,\Bbb Q)$ are quite different in nature.
Suppose $q$ is the $\HS$-cobordism invariant corresponding to a
homomorphism~$h:\Theta_3^H\to\Bbb Q$. Let $\iota$ be the mapping
taking each simplicial homology sphere to its $\HS$-cobordism
class. Then the function
$$
\Cal T_4^{\HS}\xrightarrow{\iota}\Theta_3^H\xrightarrow{h}\Bbb Q,
$$
is a local formula for~$q$. We stress that the value of this local
formula on a three-dimensional simplicial homology sphere~$Y$
depends only on the $H$-cobordism class of~$Y$ and does not depend
on the combinatorial structure of the triangulation~$Y$. In
particular, this local formula vanishes on every combinatorial
sphere. On the other hand, if $f$ is a local formula for the first
Pontryagin class, then the value $f(Y)$ necessarily depends on the
combinatorial structure of the triangulation~$Y$. In particular,
there always must be a three-dimensional combinatorial sphere~$Y$
such that~$f(Y)\ne 0$.
\begin{remark}
Local formulae for cobordism invariants corresponding to elements
of the group~$\Hom(\Theta_3^H,A)$ are closely related with the
following construction due to D.\,Sullivan~\cite{Sul71}. To any
oriented simplicial homology manifold~$Z$ D.\,Sullivan assigned
the simplicial cycle
$$
\mathop{\sum\limits_{\sigma\text{\textnormal{ a simplex of
}}Z,}}\limits_{\codim\sigma=4}
\sigma\otimes\left[\Lk\sigma\right]\in C_{\dim Z-4}(Z;\Theta_3^H).
$$
The homology class of this cycle is a complete obstruction to the
existence of a resolution of singularities $g:M\to Z$ such that
$M$ is a combinatorial manifold and the set~$g^{-1}(z)$ is
acyclic for any point~$z\in Z$.
\end{remark}

2) In~\cite{Che83} J.\,Cheeger constructed real local formulae for
the homology Hierzebruch $L$-classes of pseudo-manifolds with
negligible boundary. In particular, for each positive integer~$k$
he obtained a real local formula for the signature of a
$4k$-dimensional pseudo-manifold with negligible boundary. Thus,
the $\Bbb R$-module $H^{4k}(\Cal T^*_{\Ch}(\Bbb R))$ contains a
one-dimensional submodule generated by the cohomology class of
local formulae for the signature. The same result holds under the
field~$\Bbb Q$. Indeed, rational local formulae for the signature
can be obtained from real local formulae for the signature by any
$\Bbb Q$-linear projection~$\Bbb R\to\Bbb Q$. Unfortunately,
nobody knows how to obtain a rational local formula for the
signature of a pseudo-manifold with negligible boundary in a more
natural way.

\subsection{Relation with Sullivan-Baas cobordisms}

Suppose, $\CC=\CC(P_1,P_2,\ldots)$. Then the cobordism
ring~$\Omega^{\CC}_*$ can be regarded as a cobordism ring of
manifolds with ``join-like'' singularities
$P_{i_1}*\ldots*P_{i_k}$, $i_1<i_2<\ldots<i_k$. The idea of such
cobordisms and their first applications are due to D.~Sullivan
(see~\cite{Sul67},~\cite{Sul99},~\cite{Sul71}). The first rigorous
construction of the theory of cobordism with join-like
singularities was given by N.\,Baas~\cite{Baa73}. (He worked in
the smooth category. Below we shall refer to a piecewise-linear
version of his construction.) The cobordism groups defined by
N.\,Baas do not coincide with the groups~$\Omega^{\CC}_*$. To
illustrate this difference we shall consider the case of a single
singularity type~$P$.

A manifold with singularities of type~$P$ in sense of
Sullivan-Baas is a polyhedron
$$
V\cup_{W\times P} (W\times\cone(P)),
$$
where $W$ is a piecewise linear manifold without boundary and $V$
is a piecewise linear manifold with boundary~$W\times P$. The
class of manifolds with singularities of class~$\CC(P)$ in sense
of the present paper turns out to be wider than the class of
manifolds with singularities of type~$P$ in sense of
Sullivan-Baas. Indeed, suppose $B$ is a closed piecewise-linear
manifold, $E\to B$ is a piecewise linear fiber bundle with
fiber~$P$, $Z\to B$ is the associated fiber bundle with
fiber~$\cone(P)$, and $V$ is a piecewise linear manifold with
boundary~$E$. We consider the polyhedron~$X=V\cup_EZ$. It is easy
to check that according to our definition in~\S 3.2 the
polyhedron~$X$ is a manifold with singularities of class~$\CC(P)$.
Nevertheless, if the fiber bundle~$E$ is nontrivial the
polyhedron~$X$ need not be a manifold with singularities of
type~$P$ in sense of Sullivan-Baas.

\section{Resolving singularities and realization of cycles.}

In the following we shall omit index $\CS$ in notation for the
groups $\Cal T_*^{\CS}$ and $\Cal T^*_{\CS}(A)$.

\subsection{Simple cells}
Let $K$ be a simplicial cell combinatorial manifold. For each
vertex $y$ of~$K$ we define the cell $Dy$ dual to $y$ by
$Dy=\St_{K'}y$. Thus $Dy$ is the union of all closed simplices of
$K'$ containing the vertex $v$. Now let $\sigma$ be a simplex of
$K$. A cell $D\sigma$ dual to the simplex $\sigma$ is the
intersection of all cells dual to the vertices of~$\sigma$.
Suppose that the combinatorial manifold~$K$ and the
simplex~$\sigma$ are oriented. Then the cell $D\sigma$ will also
be endowed with the orientation such that the product of the
orientation of~$\sigma$ and the orientation of~$D\sigma$ yields
the orientation of~$K$. The cell decomposition consisting of all
cells $D\sigma$ will be denoted by~$K^*$ and will be called the
{\it cell decomposition dual to the triangulation~$K$}. A
definition of the cell decomposition dual to a cubic cell
combinatorial manifold repeats word for word the definition of the
cell decomposition dual to a simplicial cell combinatorial
manifold replacing everywhere simplices by cubes.

Now let $Y$ be a combinatorial sphere. We consider the simplicial
complex~$\cone(Y')$, which is PL homeomorphic to the standard
disk. The decomposition $Y^*$ of the boundary of $\cone(Y')$
endows $\cone(Y')$ with the structure of a manifold with corners.
We extend the decomposition~$Y^*$ to a decomposition of the whole
complex $\cone(Y')$ by putting $D\emptyset=\cone(Y')$. The
simplicial complex $\cone(Y')$ with such cell decomposition is
called a {\it simple cell} dual to the combinatorial sphere~$Y$.
If the combinatorial sphere~$Y$ is realized as the boundary of a
simplicial convex polytope, then the simple cell dual to~$Y$ can
be realized as the dual simple convex polytope.

The cell $D\sigma$ dual to a simplex~$\sigma$ of a combinatorial
manifold $K$ possesses a canonical structure of a $(\dim
K-\dim\sigma)$-dimensional simple cell dual to the combinatorial
sphere $\Lk\sigma$. (A similar assertion holds for cubic cell
combinatorial manifolds.) In particular, every face of a simple
cell has a canonical structure of a simple cell.

Dealing with simple cells we shall always mind the
triangulation~$\cone(Y')$. This means that a simple cell is always
considered as a polyhedron with two cell decompositions. The
larger decomposition is the decomposition into faces described
above. The smaller decomposition is the triangulation $\cone(Y')$
itself. In particular, suppose that $P_1$ and~$P_2$ are the simple
cells dual to combinatorial spheres $Y_1$ and $Y_2$ respectively.
Then an isomorphism of simple cells~$P_1$ and~$P_2$ is a
simplicial isomorphism $\cone(Y_1')\to\cone(Y_2')$ preserving the
decompositions into faces.

A {\it \textnormal{(}finite\textnormal{)} simple cell complex} is
the quotient of the disjoint union of finitely many simple cells
$P_1,P_2,\ldots,P_q$ by an equivalence relation~$\sim$ such that

(1) $\sim$ identifies no pair of distinct points of the same
cell~$P_i$;

(2) if $x_1\sim x_2$, $x_1\in P_i$, $x_2\in P_j$, then $\sim$
identifies some face $F_1\subset P_i$ such that $F_1\ni x_1$ with
a face $F_2\subset P_j$ such that $F_2\ni x_2$ along an
isomorphism.

{\it Cells} of this complex are images of faces of the simple
cells~$P_i$ under the quotient mapping. If all simple cells~$P_i$
are simplices (respectively, cubes), then we arrive to a
definition of a simplicial cell (respectively, cubic cell)
complex. The cell decomposition~$K^*$ dual to a combinatorial
manifold~$K$ can serve as a basic example of a simple cell
complex.

Let~$P$ be a simple cell dual to a combinatorial sphere~$Y$. Then
the simplicial complex~$\cone(Y')$ is said to be the barycentric
subdivision of~$P$. Now let~$Z$ be a simple cell complex. To
obtain a {\it barycentric subdivision} of the complex~$Z$ one
needs to take the barycentric subdivision of every cell of $Z$.
The barycentric subdivision of $Z$ will be denoted by~$Z'$. For
example,~$(K^*)'=K'$.

A simple cell complex is said to be an $n$-dimensional {\it simple
cell pseudo-manifold} if each its cell is contained in an
$n$-dimensional cell and each $(n-1)$-dimensional cell is
contained in exactly two $n$-dimensional cells. A simple cell
pseudo-manifold is said to be {\it normal} if its $n$-th local
homology group at every point is isomorphic to~$\Bbb Z$. (This
definition is quite similar to the definitions of a normal
simplicial cell and cubic cell pseudo-manifolds.) In the sequel
all pseudo-manifolds are supposed to be normal.

\subsection{Cobordism of simple cell}
Let $P$ be an oriented simple cell. The formal sum of facets of
$P$ taken with the induced orientations will be called the {\it
boundary} of~$P$ and will be denoted by~$\partial P$. Now let
$P_1,P_2,\ldots,P_k,Q_1,Q_2,\ldots,Q_l$ be oriented
$n$-dimensional simple cells. We shall say that the sum
$P_1+P_2+\ldots+P_k$ is {\it cobordant} to the sum
$Q_1+Q_2+\ldots+Q_l$ if there are oriented $n$-dimensional simple
cells $R_1,R_2,\ldots,R_m$ and oriented $(n+1)$-dimensional simple
cells $T_1,T_2,\ldots,T_q$ such that we have the following
isomorphism
\begin{multline*}
\partial (T_1+T_2+\ldots+T_q)\cong
P_1+P_2+\ldots+P_k+R_1+R_2+\ldots+R_m+\\
+(-Q_1)+(-Q_2)+\ldots+(-Q_l) +(-R_1)+(-R_2)+\ldots+(-R_m).
\end{multline*}
The corresponding cobordism group will be called the cobordism
group of $n$-dimensional simple cells and will be denoted
by~$\PP_n$. The graded group~$\PP_n$ is a graded ring with respect
to the direct product of simple cells. The boundary of a simple
cell is dual to the sum of links of vertices of a combinatorial
sphere. The direct product of simple cells is dual to the join of
combinatorial spheres. Therefore the ring~$\PP_*$ is canonically
isomorphic to the ring~$H_*(\Cal T_*)$. Thus the results of
sections~3.2,~3.3 can be reformulated in the following way.
\begin{theorem}
The functor taking each oriented $n$-dimensional combinatorial
manifold~$K$ to a formal sum of $n$-dimensional cells of the dual
decomposition~$K^*$ yields a well-defined additive homomorphism
$\psi:\Omega^{\SPL}_n\to\PP_n$. The homomorphism~$\psi$ is
multiplicative modulo elements of order~$2$. The kernel and the
cokernel of~$\psi$ are torsion groups. Thus we have a
multiplicative isomorphism
$$
\PP_*\otimes\Bbb Q\cong\Omega^{\SPL}_*\otimes\Bbb
Q\cong\Omega^{\SO}_*\otimes\Bbb Q.
$$
\end{theorem}
\begin{remark}
Suppose~$X$ is a topological space. Considering singular simple
cells of~$X$, that is, continuous mappings $P\to X$ one may in a
standard way define an extraordinary homology theory~$\PP_*(X)$.
There is a natural homomorphism
$\psi:\Omega^{\SPL}_*(X)\to\PP_*(X)$ that induces an isomorphism
$$
\PP_*(X)\otimes\Bbb Q\cong\Omega^{\SPL}_*(X)\otimes\Bbb Q\cong
H_*(X;\Omega^{\SPL}_*\otimes\Bbb Q)\cong
H_*(X;\Omega^{\SO}_*\otimes\Bbb Q).
$$
\end{remark}

\subsection{Resolving singularities of a pseudo-manifold}
In this section for each oriented $n$-dimensional simple cell
pseudo-manifold~$Z$ we construct explicitly a cubic cell
combinatorial manifold~$M$ and a piecewise-smooth mapping~$g:M\to
Z$ such that

1) the restriction of~$g$ to~$g^{-1}(Z\setminus\Sigma)$ is a
finite-fold covering
$$
g^{-1}(Z\setminus\Sigma)\to Z\setminus\Sigma,
$$
where $\Sigma$ is the codimension~$2$ skeleton of~$Z$;

2) $\dim g^{-1}(\Sigma)=n-1$.

Let $P_1,P_2,\ldots,P_k$ be all $n$-dimensional cells of a
pseudo-manifold~$Z$. We consider the oriented $(n-1)$-dimensional
combinatorial spheres $Y_1,Y_2,\ldots,Y_k$ such that
$P_i=\cone(Y_i')$ is the simple cell dual to~$Y_i$. By~$f_i$ we
denote the embedding
$$
\cone(Y_i')=P_i\subset Z.
$$
For each simplex $\sigma$ of~$Y_i$ we put
$z(\sigma)=f_i(b(\sigma))$, where $b(\sigma)$ is the barycenter
of~$\sigma$. We put,
$$
Y=Y_1\sqcup Y_2\sqcup\ldots\sqcup Y_k.
$$
As in \S 2, by $U$ we denote the set of $(n-1)$-dimensional
simplices of~$Y'$. Simplices $u\in U$ are in one-to-one
correspondence with sequences
$\sigma^0\subset\sigma^1\subset\ldots\subset\sigma^{n-1}$ of
simplices of~$Y$. Suppose $u$ is the simplex of~$Y_i'$
corresponding to a sequence
$\sigma^0\subset\sigma^1\subset\ldots\subset\sigma^{n-1}$. Let us
introduce the following notation.
$$
b_j(u)=b(\sigma^{j-1}),\qquad
z_j(u)=z(\sigma^{j-1})=f_i(b_j(u)),\qquad j=1,2,\ldots,n.
$$
Then $b_1(u),b_2(u),\ldots,b_n(u)$ are the vertices of the
simplex~$u$. By~$z_0(u)$ we denote the barycenter of the simple
cell~$P_i\subset Z$. It is easy to check that each point~$z_j(u)$
is the barycenter of an~$(n-j)$-dimensional cell $F_j(u)$ of the
decomposition~$Z$. Besides,
$$
F_n(u)\subset F_{n-1}(u)\subset\ldots\subset F_0(u)=P_i.
$$

Let us glue a pseudo-manifold~$Z$ from the simple cells~$P_i$. In
the process of gluing we pair off facets of the cells
$P_1,P_2,\ldots,P_k$ and glue the facets of each pair together
along a certain anti-isomorphism. Passing to the dual objects we
obtain the involution~$\lambda:y \mapsto \tilde{y}$ on the
set~$\V(Y)$ and the set of anti-isomorphisms~$\chi_y:\St y\to \St
\tilde{y}$ such that~$\chi_{\tilde{y}}=\chi_y^{-1}$. We apply a
construction in section~2.5 to a combinatorial manifold~$Y$. This
construction yields a homogeneous graph~$\Gamma$ of degree~$2n$
such that $X=\tbQ(\Gamma)$ is a cubic cell combinatorial manifold
satisfying the requirements of Theorem~2.3. (The
pseudo-manifold~$X$ is a combinatorial manifold, since the links
of all its vertices are combinatorial spheres.) Now we shall work
with the cubic cell combinatorial manifold $M=\bQ(\Gamma)$ without
passing to large cubes. We recall that
$$
M=(U\times\Cal B\times S\times [0,1]^n)/\sim,
$$
where the equivalence relation~$\sim$ is generated by the
identifications
\begin{gather*}
(u,\nu,s,\bt)\sim
(\Phi_j^0(u,\nu,s),\bt)\text{ if }t_j=0;\\
(u,\nu,s,\bt)\sim (\Phi_j^1(u,\nu,s),\bt)\text{ if }t_j=1.
\end{gather*}
Here $\bt=(t_1,t_2,\ldots,t_n)$ is a point of the cube~$[0,1]^n$.
The finite sets~$\Cal B$ and~$S$ and the involutions
$$
\Phi^{e}_j:U\times \Cal B\times S\to U\times\Cal B\times S
$$
were described in section~2.5. The involutions $\Phi_j^0$ are of
the following form
$$
\Phi_j^0(u,\nu,s)=(\Phi_j(u),\nu,s).
$$
The definition of the involutions~$\Phi_j$ immediately implies
that $b_l(\Phi_j(u))=b_l(u)$ for~$l\ne j$. Besides, the
simplices~$u$ and~$\Phi_j(u)$ lie in the same connected component
of the combinatorial manifold~$Y'$. Hence,
$$
z_l(\Phi_j(u))=z_l(u)\ \ \text{for}\ \ 0\leqslant l\leqslant n,\
l\ne j.
$$
We need the following property of the involutions~$\Phi_j^1$. Its
proof is postponed to section~4.4.
\begin{propos}
If $\Phi_j^1(u_1,\nu_1,s_1)=(u_2,\nu_2,s_2)$, then
$z_l(u_1)=z_l(u_2)$ for $l\geqslant j$.
\end{propos}

Suppose that $z_0,z_1,\ldots,z_n$ are the vertices of a
simplex~$\rho$ of~$Z'$ and $\alpha_0,\alpha_1,\ldots,\alpha_n$ are
nonnegative real numbers such that
$\alpha_0+\alpha_1+\cdots+\alpha_n=1$. Then we shall denote by
$$
\alpha_0z_0+\alpha_1z_1+\cdots+\alpha_nz_n
$$
the point in the simplex~$\rho$ with barycentric
coordinates~$\alpha_0,\alpha_1,\ldots,\alpha_n$.

We consider the functions
\begin{align*}
 \alpha_0(\bt)&=(1-t_1)(1-t_2)(1-t_3)\ldots(1-t_n);\\
 \alpha_1(\bt)&=t_1(1-t_2)(1-t_3)\ldots(1-t_n);\\
 \alpha_2(\bt)&=t_2(1-t_3)\ldots(1-t_n);\\
 & \ldots\\
 \alpha_{n-1}(\bt)&=t_{n-1}(1-t_n);\\
 \alpha_n(\bt)&=t_n.
\end{align*}
Obviously,
$$
\alpha_0(\bt)+\alpha_1(\bt)+\cdots+\alpha_n(\bt)=1.
$$
We define a mapping
$$
\underline{g}:U\times\Cal B\times S\times [0,1]^n\to Z
$$
by
$$
\underline{g}(u,\nu,s,\bt)=\alpha_0(\bt)z_0(u)+\alpha_1(\bt)z_1(u)+
\cdots+\alpha_n(\bt)z_n(u).
$$
\begin{propos}
We have
\begin{gather*}
\underline{g}(u,\nu,s,\bt)=\underline{g}(\Phi_j^0(u,\nu,s),\bt)\text{
if }t_j=0;\\
\underline{g}(u,\nu,s,\bt)=\underline{g}(\Phi_j^1(u,\nu,s),\bt)\text{
if }t_j=1.
\end{gather*}
Thus the mapping~$\underline{g}$ induces a well-defined
mapping~$g:M\to Z$.
\end{propos}
\begin{proof}
The first equality holds since $z_l(\Phi_j(u))=z_l(u)$ for $l\ne
j$ and $\alpha_j(\bt)=0$ for $t_j=0$. The second equality follows
from Proposition~4.1, since $\alpha_l(\bt)=0$ if $t_j=1$
and~$l<j$.
\end{proof}

By~$H$ we denote the union of the facets $\{t_j=1\}$,
$j=2,3,\ldots,n$, of the cube $[0,1]^n$. By $\Xi\subset M$ we
denote the image of the set $U\times\Cal B\times S\times H$ under
the quotient mapping
$$
U\times\Cal B\times S\times [0,1]^n\to M.
$$
Obviously, $\Xi=g^{-1}(\Sigma)$ and $\dim\Xi=n-1$. The mapping
$$
g|_{M\setminus \Xi}:M\setminus \Xi\to Z\setminus\Sigma
$$
is a covering since the system of equations
$$
\alpha_j(\bt)=\beta_j,\qquad j=0,1,\ldots,n,
$$
has a unique solution for any nonnegative real numbers
$\beta_0,\beta_1,\ldots,\beta_n$ such that
$\beta_0+\beta_1+\cdots+\beta_n=1$ and either $\beta_0\ne 0$ or
$\beta_1\ne 0$.

\begin{remark}
The complex~$M$ is the canonical subdivision of the cubic cell
combinatorial manifold~$X$ (see section~2.4). Consider the dual
decomposition~$X^*$. Let~$Q$ be the cell of~$X^*$ dual to a
vertex~$x$ of~$X$. Then $Q$ is the union of all closed cubes
of~$M$ containing the vertex~$x$. If $\Lk x\cong Y_i'$, then $Q$
is a simple cell dual to the combinatorial sphere~$Y_i'$. The
mapping~$g$ maps the simple cell~$Q$ onto the simple cell~$P_i$.
Further, $g$ is injective on the interior of the cell~$Q$ and on
the interiors of the facets of~$Q$ dual to those vertices $y\in
Y_i'$ that are vertices of~$Y_i$. Other facets collapse onto faces
of~$P_i$ of smaller dimensions. Suppose that~$F$ is the face
of~$Q$ dual to a simplex~$\tau$ of~$Y_i'$. There are two
possibilities.

1) There is no simplex~$\sigma$ of~$Y_i$ such that
$\tau\subset\sigma$ and $\dim\sigma=\dim\tau$. Then the link
of~$\tau$ possesses an anti-automorphism. Hence the face~$F$
possesses an anti-automorphism. It is not hard to check that in
this case $g$ maps~$F$ onto a face of a smaller dimension.

2) There is a simplex~$\sigma$ of~$Y_i$ such that
$\tau\subset\sigma$ and $\dim\sigma=\dim\tau$. Let~$E$ be the face
of~$P_i$ dual to~$\sigma$. Then $g$ maps the cell~$F$ onto the
cell~$E$ so that the interior of~$F$ is mapped homeomorphically
onto the interior of~$E$. Notice that the combinatorial
sphere~$\Lk_{Y_i'}\tau$ dual to the cell~$F$ is isomorphic to the
barycentric subdivision of the combinatorial
sphere~$\Lk_{Y_i}\sigma$ dual to the cell~$E$.

For example, let $Y_i$ be the boundary of a quadrangle. Then $P_i$
is a quadrangle and $Q$ is an octagon. The mapping~$g|_Q$ is shown
in Fig.~2. The inclined sides of the octagon~$Q$ are mapped to the
vertices of the quadrangle~$P_i$.

Another example is shown in Fig.~3. Here $Y_i$ is the boundary of
a tetrahedron. Then~$P_i$ is a tetrahedron and $Q$ is a polytope
shown in the left part of Fig.~3. (This polytope is called a {\it
permutohedron}.) The shaded hexagon faces of~$Q$ are mapped to the
corresponding vertices of the tetrahedron. The quadrangle faces
of~$Q$ are mapped onto the edges of the tetrahedron and every
segment parallel to the shading collapses to a point. The
interiors of the unshaded hexagon faces of~$Q$ are mapped
homeomorphically onto the interiors of the corresponding faces of
the tetrahedron.
\end{remark}

\begin{center}
{\unitlength=2.5mm
\begin{picture}(42,22)
\put(9,13){\circle*{0.3}}

\put(1,9){\circle*{0.3}}

\put(1,17){\circle*{0.3}}

\put(5,21){\circle*{0.3}}

\put(13,21){\circle*{0.3}}

\put(17,17){\circle*{0.3}}

\put(17,9){\circle*{0.3}}

\put(13,5){\circle*{0.3}}

\put(5,5){\circle*{0.3}}

\put(29,7){\circle*{0.3}}

\put(29,19){\circle*{0.3}}

\put(41,7){\circle*{0.3}}

\put(41,19){\circle*{0.3}}

\put(35,13){\circle*{0.3}}

\thicklines

\put(1,9){\line(0,1){8}}

\put(17,9){\line(0,1){8}}

\put(5,5){\line(1,0){8}}

\put(5,21){\line(1,0){8}}

\put(1,9){\line(1,-1){4}}

\put(1,17){\line(1,1){4}}

\put(13,5){\line(1,1){4}}

\put(13,21){\line(1,-1){4}}

\put(29,7){\line(1,0){12}}

\put(29,7){\line(0,1){12}}

\put(41,7){\line(0,1){12}}

\put(29,19){\line(1,0){12}}

\thinlines

\put(1,13){\line(1,0){16}}

\put(9,5){\line(0,1){16}}

\put(3,7){\line(1,1){12}}

\put(3,19){\line(1,-1){12}}

\put(1,9){\line(2,1){16}}

\put(5,5){\line(1,2){8}}

\put(1,17){\line(2,-1){16}}

\put(5,21){\line(1,-2){8}}

\put(29,7){\line(1,1){12}}

\put(35,7){\line(0,1){12}}

\put(29,13){\line(1,0){12}}

\put(29,19){\line(1,-1){12}}

\renewcommand{\qbeziermax}{1000}

\qbezier[1000](35,13)(34,8)(29,7)

\qbezier[1000](35,13)(30,12)(29,7)

\qbezier[1000](35,13)(34,18)(29,19)

\qbezier[1000](35,13)(30,14)(29,19)

\qbezier[1000](35,13)(36,8)(41,7)

\qbezier[1000](35,13)(40,12)(41,7)

\qbezier[1000](35,13)(36,18)(41,19)

\qbezier[1000](35,13)(40,14)(41,19)

\put(20,13){\vector(1,0){6}}

\put(19,0){Figure 2.}

\end{picture}

}

\end{center}
\begin{center}
{\unitlength=1.5mm
\begin{picture}(70,40)

\put(1.5,15){\circle*{0.5}}

\put(1.5,19){\circle*{0.5}}

\put(3,12){\circle*{0.5}}

\put(9,34){\circle*{0.5}}

\put(10.5,12){\circle*{0.5}}

\put(12,9){\circle*{0.5}}

\put(12,15){\circle*{0.5}}

\put(12,33){\circle*{0.5}}

\put(12,35){\circle*{0.5}}

\put(28.5,15){\circle*{0.5}}

\put(28.5,19){\circle*{0.5}}

\put(27,12){\circle*{0.5}}

\put(21,34){\circle*{0.5}}

\put(19.5,12){\circle*{0.5}}

\put(18,9){\circle*{0.5}}

\put(18,15){\circle*{0.5}}

\put(18,33){\circle*{0.5}}

\put(18,35){\circle*{0.5}}

\put(51,15){\circle*{0.5}}

\put(60,12){\circle*{0.5}}

\put(60,33){\circle*{0.5}}

\put(69,15){\circle*{0.5}}

\thicklines

\put(1.5,15){\line(1,-2){1.5}}

\put(1.5,15){\line(3,-1){9}}

\put(1.5,15){\line(0,1){4}}

\put(1.5,19){\line(1,2){7.5}}

\put(3,12){\line(3,-1){9}}

\put(9,34){\line(3,-1){3}}

\put(9,34){\line(3,1){3}}

\put(10.5,12){\line(1,-2){1.5}}

\put(10.5,12){\line(1,2){1.5}}

\put(12,15){\line(0,1){18}}

\put(28.5,15){\line(-1,-2){1.5}}

\put(28.5,15){\line(-3,-1){9}}

\put(28.5,15){\line(0,1){4}}

\put(28.5,19){\line(-1,2){7.5}}

\put(27,12){\line(-3,-1){9}}

\put(21,34){\line(-3,-1){3}}

\put(21,34){\line(-3,1){3}}

\put(19.5,12){\line(-1,-2){1.5}}

\put(19.5,12){\line(-1,2){1.5}}

\put(18,15){\line(0,1){18}}

\put(12,35){\line(1,0){6}}

\put(12,33){\line(1,0){6}}

\put(12,15){\line(1,0){6}}

\put(12,9){\line(1,0){6}}

\put(51,15){\line(3,-1){9}}

\put(51,15){\line(1,2){9}}

\put(60,12){\line(3,1){9}}

\put(60,12){\line(0,1){21}}

\put(60,33){\line(1,-2){9}}

\thinlines


\put(10,33.67){\line(0,1){0.67}}

\put(11,33.33){\line(0,1){1.33}}

\put(9.5,33.83){\line(0,1){0.33}}

\put(10.5,33.5){\line(0,1){1}}

\put(11.5,33.17){\line(0,1){1.67}}

\put(12,33){\line(0,1){2}}

\put(13,33){\line(0,1){2}}

\put(14,33){\line(0,1){2}}

\put(15,33){\line(0,1){2}}

\put(16,33){\line(0,1){2}}

\put(17,33){\line(0,1){2}}

\put(18,33){\line(0,1){2}}

\put(12.5,33){\line(0,1){2}}

\put(13.5,33){\line(0,1){2}}

\put(14.5,33){\line(0,1){2}}

\put(15.5,33){\line(0,1){2}}

\put(16.5,33){\line(0,1){2}}

\put(17.5,33){\line(0,1){2}}

\put(19,33.33){\line(0,1){1.33}}

\put(20,33.67){\line(0,1){0.67}}

\put(20.5,33.83){\line(0,1){0.33}}

\put(19.5,33.5){\line(0,1){1}}

\put(18.5,33.17){\line(0,1){1.67}}

\put(10.5,33.5){\line(1,0){9}}

\put(10.5,34.5){\line(1,0){9}}

\put(9,34){\line(1,0){12}}


\put(11,11){\line(0,1){2}}

\put(11.5,10){\line(0,1){4}}

\put(12,9){\line(0,1){6}}

\put(19,11){\line(0,1){2}}

\put(18.5,10){\line(0,1){4}}

\put(18,9){\line(0,1){6}}

\put(13,9){\line(0,1){6}}

\put(14,9){\line(0,1){6}}

\put(15,9){\line(0,1){6}}

\put(16,9){\line(0,1){6}}

\put(17,9){\line(0,1){6}}

\put(12.5,9){\line(0,1){6}}

\put(13.5,9){\line(0,1){6}}

\put(14.5,9){\line(0,1){6}}

\put(15.5,9){\line(0,1){6}}

\put(16.5,9){\line(0,1){6}}

\put(17.5,9){\line(0,1){6}}

\put(11.75,9.5){\line(1,0){6.5}}

\put(11.5,10){\line(1,0){7}}

\put(11.25,10.5){\line(1,0){7.5}}

\put(11,11){\line(1,0){8}}

\put(10.75,11.5){\line(1,0){8.5}}

\put(10.5,12){\line(1,0){9}}

\put(10.75,12.5){\line(1,0){8.5}}

\put(11,13){\line(1,0){8}}

\put(11.75,14.5){\line(1,0){6.5}}

\put(11.5,14){\line(1,0){7}}

\put(11.25,13.5){\line(1,0){7.5}}


\put(12,16){\line(1,0){6}}

\put(12,17){\line(1,0){6}}

\put(12,18){\line(1,0){6}}

\put(12,19){\line(1,0){6}}

\put(12,20){\line(1,0){6}}

\put(12,21){\line(1,0){6}}

\put(12,22){\line(1,0){6}}

\put(12,23){\line(1,0){6}}

\put(12,24){\line(1,0){6}}

\put(12,25){\line(1,0){6}}

\put(12,26){\line(1,0){6}}

\put(12,27){\line(1,0){6}}

\put(12,28){\line(1,0){6}}

\put(12,29){\line(1,0){6}}

\put(12,30){\line(1,0){6}}

\put(12,31){\line(1,0){6}}

\put(12,32){\line(1,0){6}}


\put(19.5,9.5){\line(1,2){1.5}}

\put(21,10){\line(1,2){1.5}}

\put(22.5,10.5){\line(1,2){1.5}}

\put(24,11){\line(1,2){1.5}}

\put(25.5,11.5){\line(1,2){1.5}}


\put(10.5,9.5){\line(-1,2){1.5}}

\put(9,10){\line(-1,2){1.5}}

\put(7.5,10.5){\line(-1,2){1.5}}

\put(6,11){\line(-1,2){1.5}}

\put(4.5,11.5){\line(-1,2){1.5}}

\put(51,15){\line(1,0){3}}

\put(56,15){\line(1,0){3}}

\put(61,15){\line(1,0){3}}

\put(66,15){\line(1,0){3}}

\put(33.33,21.67){\vector(1,0){10}}

\put(31.67,0){Figure 3.}

\end{picture}

}
\end{center}

\subsection{Proof of Proposition 4.1}
We shall use the notation of section~2.5. The only difference is
that the set~$S$ and the involutions~$\Phi^{e}_j$ are now
constructed from the combinatorial
manifold~$\overline{Y}=Y\times\Cal B$ rather than from the
combinatorial manifold~$Y$. In particular, the vertex set of the
graph~$G$ is the set $\overline{W}=W\times\Cal B$ rather then~$W$.
We denote by~$\pi$ the projection~$W\times \Cal B\to W$. We put,
$\bar{z}=z\circ\pi$. We have noticed in section~2.5 that for two
simplices $\rho_1$ and $\rho_2$ of the complex~$\overline{Y}$ a
labeling-preserving  anti-isomorphism~$\St\rho_1\to\St\rho_2$ is
unique (if any exists). If such an anti-isomorphism exists we
shall denote it by~$\omega_{\rho_1,\rho_2}$. In particular, if
$\bar{y}=(y,\nu)$ is a vertex of~$\overline{Y}$, then
$\omega_{\bar{y},\Tilde{\bar{y}}}=\overline{\chi}_{\bar{y}}$ is
the anti-isomorphism induced by the anti-isomorphism~$\chi_y$.

The definition of the anti-isomorphisms~$\chi_y$ implies
immediately that $z(\chi_y(\sigma))=z(\sigma)$ for any
simplex~$\sigma\in W$ containing the vertex~$y$. Therefore,
$\bar{z}(\overline{\chi}_{\bar{y}}(\sigma))=\bar{z}(\sigma)$ for
any simplex~$\sigma\in \overline{W}$ containing the
vertex~$\bar{y}$. Suppose $\rho_1,\rho_2\in\overline{W}$ are two
simplices connected by an edge in the graph~$G$. Then there is a
vertex~$\bar{y}\in\rho_1$ such that
$\overline{\chi}_{\bar{y}}(\rho_1)=\rho_2$. The
anti-isomorphism~$\omega_{\rho_1,\rho_2}$ is the restriction of
the anti-isomorphism~$\overline{\chi}_{\bar{y}}$ to the
subcomplex~$\St\rho_1\subset\St\bar{y}$.
Hence~$\bar{z}(\omega_{\rho_1,\rho_2}(\sigma))=\bar{z}(\sigma)$
for every simplex~$\sigma\supset\rho_1$. Consequently the same
equality $\bar{z}(\omega_{\rho_1,\rho_2}(\sigma))=\bar{z}(\sigma)$
holds for any simplices~$\rho_1,\rho_2$ belonging to the distinct
parts of the same connected component of the graph~$G$ and any
simplex $\sigma\supset\rho_1$. Therefore,
$\bar{z}(\Lambda_c(\sigma))=\bar{z}(\sigma)$ for any $\Lambda\in
P$, $c\subset\Cal C$ such that $c(\sigma)\supset c$.

Let
$\sigma^0_i\subset\sigma^1_i\subset\ldots\subset\sigma^{n-1}_i$,
$i=1,2$, be the sequences of simplices of~$\overline{Y}$
corresponding to the simplices~$(u_i,\nu_i)$ of~$\overline{Y}'$.
We put, $c=c(\sigma^{j-1}_1)$. Then
$\sigma^{l-1}_2=\Lambda_c(\sigma^{l-1}_1)$, $l=1,2,\ldots,n$.
Hence for $l\geqslant j$ we obtain
$\bar{z}(\sigma_2^{l-1})=\bar{z}(\sigma_1^{l-1})$, which implies
Proposition~4.1.

\subsection{The cobordism class of the manifold $M$}
Resolving singularities can yield manifolds representing different
cobordism classes. We deal with resolving singularities ``with
multiplicities'', that is, with mappings~$M\to Z$ that are
finite-fold coverings on an open everywhere dense subset. Hence it
makes sense to consider the class~$\frac{[M]}{r}\in
\Omega_n^{\SPL}\otimes\Bbb Q=\Omega_n^{\SO}\otimes\Bbb Q$, where
$r$ is the number of sheets of the covering. It turns out that our
construction in section~4.3 yields the manifold~$M$ with the
class~$\frac{[M]}{r}$ completely determined by the set of simple
cells~$P_1,P_2,\ldots,P_k$.

\begin{theorem}
Suppose $Z$ is an oriented $n$-dimensional simple cell
pseudo-manifold, $P_1,P_2,\ldots,P_k$ are all its $n$-dimensional
cells, $M$ is the manifold constructed in section~4.3, $r$ is the
number of sheets of the covering~$M\setminus\Xi\to
Z\setminus\Sigma$. Then
$$
\frac{[M]}r=\psi^{-1}([P_1+P_2+\cdots+P_k]),
$$
where $\psi:\Omega_n^{\SPL}\otimes\Bbb Q\to\PP_n\otimes\Bbb Q$ is
the isomorphism in Theorem~4.1.
\end{theorem}
\begin{proof}
Consider the combinatorial manifold~$M'$. The set of links of
vertices of~$M'$ consists of the $r$-fold multiple of the set
$Y_1'',Y_2'',\ldots,Y_k''$ and several combinatorial spheres each
of which possesses an anti-automorphism. Hence,
$$
2\diff(M')=2r\sum_{i=1}^kY_i''
$$
in the group $\Cal T_n$. It follows from Proposition~2.13 that the
cycles
$$
2\sum_{i=1}^kY_i''\qquad\text{and}\qquad 2\sum_{i=1}^kY_i
$$
represent equal homology classes in the group~$H_n(\Cal T_*)$. The
homomorphism~$\psi$ corresponds to the homomorphism~$\diff_*$
under the canonical isomorphism~$\PP_n\cong H_n(\Cal T_*)$.
Therefore,
$$
2\psi([M])=2r[P_1+P_2+\ldots+P_k]
$$
in the group~$\PP_n$.
\end{proof}

\subsection{Realization of simplicial cycles}
Suppose $R$ is a topological space, $C_*^{\sing}(R)$ is its
singular simplicial chain complex. Suppose $\xi\in C_n^{\sing}(R)$
is a cycle. The cycle $\xi$ can easily be realized as the image of
the fundamental class of an oriented simplicial
pseudo-manifold~$Z$ under a continuous mapping~$h:Z\to R$. A
simplex can be considered as a simple cell dual to the boundary of
a simplex. Thus a simplicial pseudo-manifold is always a simple
cell pseudo-manifold. We resolve the singularities of~$Z$ using
the construction described in section~4.3. The image of the
fundamental class of the manifold~$M$ under the composite mapping
$$
\begin{CD}
\varphi:M @>g>> Z@>h>>R
\end{CD}
$$
is equal to~$r[\xi]$, where $r$ is the number of sheets of the
covering~$M\setminus\Xi\to Z\setminus\Sigma$.

The combinatorial spheres~$Y_i$ dual to $n$-dimensional cells of
the pseudo-manifold~$Z$ are isomorphic to the boundary of a
simplex. Hence for $j<n$ the link of each $j$-dimensional simplex
of~$M'$ possesses an anti-automorphism. By a result of N.~Levitt
and C.~Rourke~\cite{LeRo78} this implies that the rational
Pontryagin classes of~$M$ vanish. Consider the Pontryagin numbers
of the mapping~$\varphi$ (see~~\cite{CoFl64}). Suppose,
$\omega=(i_1,i_2,\ldots,i_m)$, $|\omega|=\sum i_j$, $b\in
H^{n-|\omega|}(R)$; then
$$
\left\langle p_{\omega}(M)\smile \varphi^*b,[M]\right\rangle=
\left\{
\begin{aligned}
&0, &\text{ if }|\omega|>0;\\
r\langle b&,[\xi]\rangle, &\text{ if }|\omega|=0.
\end{aligned}
\right.
$$
In~\cite{CoFl64} it is proved that the class
$[\varphi]\in\Omega_n^{\SPL}(R)\otimes\Bbb
Q\cong\Omega_n^{\SO}(R)\otimes\Bbb Q$ represented by the
mapping~$\varphi$ is completely determined by its Pontryagin
numbers. Therefore the element
$\frac{[\varphi]}r\in\Omega^{\SPL}_n(R)\otimes\Bbb Q$ is
independent of the arbitrariness in the construction of the
manifold~$M$ and does not change if we replace the cycle~$\xi$
with a homologous one. Thus for each topological space~$R$ our
construction yields a mapping
$$\theta_{R}:H_*(R;\Bbb Z)\to\Omega^{\SPL}_*(R)\otimes\Bbb Q,$$
which is a right inverse to the augmentation
mapping~$\Omega^{\SPL}_*(R)\otimes\Bbb Q\to H_*(R;\Bbb Q)$.
Obviously $\theta_{R}$ is a natural transformation of homology
theories~$H_*(\cdot\,;\Bbb
Z)\to\Omega_*^{\SPL}(\cdot\,)\otimes\Bbb Q$ that for a one-point
space $\theta_{R}$ coincides with the standard embedding~$\Bbb
Z\subset\Omega_0^{\SPL}\otimes\Bbb Q$. Therefore the
mapping~$\theta_{R}$ coincides with the composition
$$
\begin{CD}
H_*(R;\Bbb Z)@>{\eta}>>H_*(R;\Omega^{\SPL}_*\otimes\Bbb
Q)@>{\left(\ch^{\SPL}\right)^{-1}}>>\Omega^{\SPL}_*(R)\otimes\Bbb
Q,
\end{CD}
$$
where $\eta$ is the homomorphism induced by the standard
embedding~$\Bbb Z\subset\Omega^{\SPL}_0\otimes\Bbb Q$ and
$$
\ch^{\SPL}:\Omega^{\SPL}_*(R)\to H_*(R;\Omega^{\SPL}_*\otimes\Bbb
Q)
$$
is the Chern-Dold character in oriented piecewise-linear bordism.

\subsection{Realization of cycles dual to simplicial cocycles}

Suppose $R$ is an oriented $m$-dimensional combinatorial manifold,
$C^*(R;\Bbb Z)$ is its simplicial cochain complex. Let $c\in
C^{m-n}(R;\Bbb Z)$ be a cocycle. The cycle~$\xi$ dual to~$c$ lies
in the cellular chain group of the decomposition~$R^*$. (This
group will be denoted by~$C_n(R^*;\Bbb Z)$.) Assume that we wish
to realize a multiple of the homology class~$[\xi]$ by an image of
the fundamental class of a manifold. Obviously, we can easily
replace the cycle~$\xi$ by a homologous simplicial cycle belonging
to the group~$C_n(R';\Bbb Z)$. Thus we shall reduce our problem to
the problem considered in the previous section. However, a more
interesting results can be obtained if we consider the cycle~$\xi$
as a cycle consisting of the simple cells of the
decomposition~$R^*$.

The cycle~$\xi$ can be easily realized as an image of the
fundamental class of a simple cell pseudo-manifold~$Z$ under a
mapping~$h$ that maps each cell of~$Z$ isomorphically onto a cell
of~$R^*$. We resolve the singularities of the pseudo-manifold~$Z$
using the construction described in section~4.3. The image of the
fundamental class of the manifold~$M$ under the composite mapping
$$
\begin{CD}
\varphi:M @>g>> Z@>h>>R
\end{CD}
$$
is equal to~$r[\xi]$, where $r$ is the number of sheets of the
covering~$M\setminus\Xi\to Z\setminus\Sigma$.

\begin{propos}
The mapping~$\varphi^*$ takes the rational Pontryagin classes of
the manifold~$R$ to the rational Pontryagin classes of the
manifold~$M$.
\end{propos}
The proof is postponed to the next section.

The Pontryagin numbers of the mapping~$\varphi$ are
$$
\left\langle p_{\omega}(M)\smile \varphi^*b,[M]\right\rangle=
r\left\langle p_{\omega}(R)\smile b,[\xi]\right\rangle.
$$
Therefore the element
$\frac{[\varphi]}{r}\in\Omega_n^{\SPL}(R)\otimes\Bbb Q$ is
independent of the arbitrariness in the construction of the
manifold~$M$ and does not change if we replace the cycle~$\xi$
with a homologous one. Thus for each combinatorial manifold~$R$
our construction yields a mapping
$$\theta_{R^*}:H_*(R;\Bbb Z)\to\Omega^{\SPL}_*(R)\otimes\Bbb Q,$$
which is a right inverse to the augmentation
mapping~$\Omega^{\SPL}_*(R)\otimes\Bbb Q\to H_*(R;\Bbb Q)$. Unlike
the mapping~$\theta_{R}$, the mapping~$\theta_{R^*}$ is not a
natural transformation of homology theories. Indeed, the mapping
dual to~$\theta_{R^*}$ is a natural transformation of cohomology
theories.

\begin{propos}
The mapping~$\theta_{R^*}$ coincides with the composition
$$
H_*(R;\Bbb Z)\xrightarrow{D} H^*(R;\Bbb Z) \xrightarrow{\eta}
H^*(R;\Omega_{\SPL}^*\otimes\Bbb Q) \xrightarrow{\ch_{\SPL}^{-1}}
\Omega_{\SPL}^*(R)\otimes\Bbb Q
\xrightarrow{D_{\SPL}^{-1}\otimes\,\Bbb Q}
\Omega^{\SPL}_*(R)\otimes\Bbb Q,
$$
where $\eta$ is the homomorphism induced by the standard embedding
$\Bbb Z\subset\Omega_{\SPL}^*\otimes\Bbb Q$, $D$ and $D_{\SPL}$
are the Poincar\'e duality operators in cohomology and oriented
piecewise linear cobordism respectively, and
$$
\ch_{\SPL}:\Omega_{\SPL}^*(R)\to H^*(R;\Omega_{\SPL}^*\otimes\Bbb
Q)
$$
is the Chern-Dold character in oriented piecewise linear
cobordism.
\end{propos}
\begin{proof}
Let $\pi^j(R)=\lim\limits_{\longrightarrow}[\Sigma^qR^+,S^{q+j}]$
be the stable cohomotopy group of the space~$R$. By a theorem of
J.-P.~Serre there is a natural isomorphism~$H^*(R;\Bbb Q)\cong
\pi^*(R)\otimes\Bbb Q$. By a construction of
V.~M.~Buchstaber~\cite{Buc78} the mapping~$\ch_{\SPL}^{-1}$
inverse to the Chern-Dold character in oriented piecewise-linear
cobordism can be represented as the composition
$$
H^*(R;\Omega_{\SPL}^*\otimes\Bbb
Q)\stackrel{\cong}{\longrightarrow}
\pi^*(R)\otimes\Omega_{\SPL}^*\otimes\Bbb Q \stackrel{H\otimes
\Bbb Q} {\longrightarrow} \Omega_{\SPL}^*(R)\otimes\Bbb Q,
$$
where $H:\pi^*(R)\otimes\Omega_{\SPL}^*\to\Omega_{\SPL}^*(R)$ is
the Hurewicz homomorphism taking an element~$(a,\alpha)$,
$a\in\pi^j(R)$, $\alpha\in\Omega_{\SPL}^{-l}$, to the image
of~$\alpha$ under the homomorphism
$$
\Omega_{\SPL}^{-l}\cong\widetilde{\Omega}_{\SPL}^{q+j-l}(S^{q+j})
\stackrel{a^*}{\longrightarrow}\widetilde{\Omega}_{\SPL}^{q+j-l}
(\Sigma^qR^+)\cong\Omega_{\SPL}^{j-l}(R).
$$

Suppose, $\alpha=1\in\Omega^0_{\SPL}$, $R$ is an oriented
$m$-dimensional combinatorial manifold. We put,
$x=(D_{\SPL}^{-1}\circ H)(a,1)\in \Omega_{m-j}^{\SPL}(R)$. Let
$$
\gamma\in
\widetilde{\Omega}_{\SPL}^{q+j}(\Sigma^qR^+)=\Omega_{\SPL}^{q+j}(R\times
D^q,R\times S^{q-1}).
$$
be the image of the fundamental class of the sphere~$S^{q+j}$
under the mapping~$a^*$. Let $y\in\Omega_{m-j}^{\SPL}(R\times
D^q)$ be the Poincar\'e-Lefschetz dual of the class~$\gamma$.
Obviously the class~$y$ goes to the class~$x$ under the canonical
isomorphism~$\Omega_{m-j}^{\SPL}(R\times
D^q)\to\Omega^{\SPL}_{m-j}(R)$. The class~$y$ can be represented
by a transversal preimage of a point under a mapping of the
homotopy class~$a$. This preimage is a submanifold with a trivial
normal bundle. Hence the class~$x$ can be represented by a
mapping~$\varkappa:N\to R$ such that $\varkappa^*(p_i(R))=p_i(N)$,
where $p_i$ are the rational Pontryagin classes.

Suppose,
$$z=\left((D_{\SPL}^{-1}\otimes\Bbb Q)\circ \ch_{\SPL}^{-1}\circ\,\eta\circ D
\right)([\xi])\in \Omega_n^{\SPL}(R)\otimes\Bbb Q.$$ Then $z\in
(D_{\SPL}^{-1}\circ H)(\pi^*(R))\otimes\Bbb Q$. Hence a certain
multiple of the class~$z$ can be represented by an image of a
manifold whose rational Pontryagin classes are the pullbacks of
the rational Pontryagin classes of~$R$. On the other hand, by
Proposition~4.3, a certain multiple of the
class~$\theta_{R^*}([\xi])$ can also be represented by an image of
a manifold whose rational Pontryagin classes are the pullbacks of
the rational Pontryagin classes of~$R$. Hence, the Pontryagin
numbers of~$\theta_{R^*}([\xi])$ coincide with the Pontryagin
numbers of~$z$. Therefore,~$\theta_{R^*}([\xi])=z$.
\end{proof}

\subsection{Proof of Proposition~4.3}

In~\cite{Gai04} the author proved the following proposition (see
also section~5.1 of the present paper).
\begin{propos}
For each positive integer~$l$ there is a $\Bbb Q$-valued
function~$f$ on the set of isomorphism classes of oriented
$(4l-1)$-dimensional combinatorial spheres such that~$f(-Y)=-f(Y)$
and for every combinatorial manifold~$K$ the chain
$$
f_{\sharp}(K)=\mathop{\sum\limits_{\sigma\,\,\text{\textnormal{a
simplex of}}\,K,}}\limits_{\codim\sigma=4l} f(\Lk\sigma)\sigma,
$$
is a cycle whose homology class is the Poincar\'e dual of the
$l$th rational Pontryagin class of~$K$.
\end{propos}
\begin{corr}
For each combinatorial manifold~$K$ and each positive integer~$j$
the chain
$$
f_{\sharp}^{(j)}(K)=\mathop{\sum\limits_{\sigma\,\,\text{\textnormal{a
simplex of}}\,K,}}\limits_{\codim\sigma=4l}
f((\Lk\sigma)^{(j)})\sigma,
$$
is a cycle whose homology class is the Poincar\'e dual
of~$p_l(K)$. \textnormal{(}By $Y^{(j)}$ we denote the $j$th
barycentric subdivision of the complex~$Y$.\textnormal{)}
\end{corr}
This corollary can be deduced from Proposition~2.13. Nevertheless,
it is easier to notice that the chain~$f_{\sharp}^{(j)}(K)$ is a
cycle homologous to the cycle~$f_{\sharp}(K^{(j)})$. In the same
way we obtain the following assertion.
\begin{corr}
For each cubic cell combinatorial manifold~$K$ and each
nonnegative integer~$j$ the cubic chain
$$
f_{\sharp}^{(j)}(K)=\mathop{\sum\limits_{\sigma\,\,\text{\textnormal{a
cube of}}\,K,}}\limits_{\codim\sigma=4l}
f((\Lk\sigma)^{(j)})\sigma,
$$
is a cycle whose homology class is the Poincar\'e dual
of~$p_l(K)$.
\end{corr}
Suppose~$Z$ is a simple cell pseudo-manifold. Let $c^{(j)}(Z)\in
C^{4l}(Z;\Bbb Q)$ be the cellular cochain such that
$c^{(j)}(Z)(P)=f(Y^{(j)})$ for every simple cell~$P$ dual to a
combinatorial sphere~$Y$.
\begin{corr}
If $K$ is a simplicial cell or cubic cell combinatorial manifold,
then~$\left[c^{(j)}(K^*)\right]=p_l(K)$.
\end{corr}
Let~$X$ be the cubic complex obtained from~$M$ by passing to large
cubes (see section~2.4). Consider a $4l$-dimensional cell~$Q$ of
the complex~$X^*$. There are two possibilities (see Remark~4.2).

1)  $\dim g(Q)<4l$. Then the cell~$Q$ possesses an
anti-automorphism. Hence,~$$c^{(0)}(X^*)(Q)=0.$$

2) The mapping~$g$ maps the cell~$Q$ onto a~$4l$-dimensional
cell~$P$ of the decomposition~$R^*$. Besides, the interior of~$Q$
is mapped homeomorphically onto the interior of~$P$. If~$P$ is a
simple cell dual to a combinatorial sphere~$Y$, then $Q$ is a
simple cell dual to the combinatorial sphere~$Y'$. Therefore,
$$c^{(0)}(X^*)(Q)=c^{(1)}(R^*)(P).$$

Thus, $g^*\left(c^{(1)}(R^*)\right)=c^{(0)}(X^*)$. Consequently,
$g^*(p_l(R))=p_l(M)$.

\section{Local formulae for rational Pontryagin classes}

In this section we work with combinatorial spheres and
combinatorial manifolds and we always omit the index~$\CS$ in the
notation for the groups~$\Cal T_*^{\CS}$ and~$\Cal T^*_{\CS}(A)$.
However, all results of sections~5.1--5.8 still hold if we replace
the class~$\CS$ by the class~$\HS$, that is, if we replace
simplicial spheres by simplicial homology spheres and
combinatorial manifolds by simplicial homology manifolds.

Recall that by Proposition~3.4 the homomorphism
$$
\delta^*:H^*(\Cal T^*(\Bbb Q))\to\Hom(\Omega_*^{\SPL},\Bbb
Q)\cong\Bbb Q[p_1,p_2,\ldots]
$$
is an isomorphism.

For simplicity we shall assume that all combinatorial manifolds
considered in this section are oriented. All results still hold
for unorientable combinatorial manifolds if we replace ordinary
simplicial chains by {\it cooriented simplicial chains}
(see~\cite{GGL75}, see also~\cite{Gai04}). To each
function~$f\in\Cal T^n(\Bbb Q)$ and each $m$-dimensional
combinatorial manifold~$K$ we assign a simplicial chain
$$
f_{\sharp}(K)=\mathop{\sum\limits_{\sigma\,\,\text{\textnormal{a
simplex of}}\,K,}}\limits_{\codim\sigma=n} f(\Lk\sigma)\sigma.
$$
Recall that the summand~$f(\Lk\sigma)\sigma$ is independent of the
orientation of the simplex~$\sigma$. In~\cite{Gai04} the author
proved the following proposition.
\begin{propos}
The chain~$f_{\sharp}(K)$ is a cycle for every combinatorial
manifold~$K$ if and only if the function~$f$ is a cocycle of the
complex~$\Cal T^*(\Bbb Q)$. The homology class~$f_{\sharp}(K)$
depends only on the cohomology class of~$f$. Suppose that
$\delta^*[f]=F(p_1,p_2,\ldots)$, where $F$ is a polynomial; then
for every combinatorial manifold~$K$ the homology class~$\left[
f_{\sharp}(K)\right]$ is the Poincar\'e dual of the cohomology
class~$F(p_1(K),p_2(K),\ldots)$, where $p_i(K)$ are the rational
Pontryagin classes of~$K$.
\end{propos}
Thus the function~$f$ such that~$\delta^*[f]=F(p_1,p_2,\ldots)$
yields a {\it universal local formula} for the polynomial~$F$ in
rational Pontryagin classes. ``Universality'' means that the
function~$f$ does not depend on a combinatorial manifold~$K$ and
the coefficient of a simplex in the cycle~$f_{\sharp}(K)$ is
determined solely by the combinatorial type of the link of this
simplex. We shall say that such function~$f$ is a {\it local
formula} for a polynomial~$F(p_1,p_2,\ldots)$. By Proposition~5.1,
for each polynomial in rational Pontryagin classes there is a
local formula unique up to a coboundary of the complex~$\Cal
T^*(\Bbb Q)$.

\subsection{Local formulae for the Hirzebruch $L$-classes}
In this section we describe explicitly all local combinatorial
formulae for the Hirzebruch $L$-polynomials in Pontryagin classes.
Recall that~$L_l(p_1,p_2,\ldots,p_l)$ is the homogeneous
polynomial of degree~$4l$ (the degree of a variable~$p_i$ is~$4i$)
given by the formula
$$
1+\sum_{l=1}^{\infty}L_l(p_1,p_2,\ldots,p_l)=\prod_{j=1}^{\infty}\frac{\sqrt{t_j}}{\tanh(\sqrt{t_j})},
$$
where $p_i$ is the $i$th elementary symmetric polynomial in
variables~$t_j$. For each $4l$-dimensional oriented closed
manifold~$M$ there is a classical Hirzebruch formula
$$
\Sign M=\left\langle
L_l(p_1(M),p_2(M),\ldots,p_l(M)),[M]\right\rangle.
$$
\begin{propos}
Suppose $f\in\Cal T^{4l}(\Bbb Q)$ is a local formula for the~$l$th
Hirzebruch polynomial. Then for any balanced
set~$Y_1,Y_2,\ldots,Y_k$ of oriented $(4l-1)$-dimensional
combinatorial spheres the function~$f$ satisfies the equation
$$
f(Y_1)+f(Y_2)+\ldots+f(Y_k)=\frac{\Sign X}{r},\eqno(*)
$$
where $X$ and $r$ are the oriented cubic cell combinatorial
manifold and the positive integer in Theorem~1.2. Vice versa, each
function~$f\in\Cal T^{4l}(\Bbb Q)$ satisfying the system of
equations~$(*)$ is a local formula for the polynomial~$L_l$.
\end{propos}
\begin{proof} Suppose $f$ is a local formula for the
polynomial~$L_l$; then
\begin{multline*}
\Sign X=\left\langle
L_{l}(p_1(X),p_2(X),\ldots,p_l(X),[X]\right\rangle=
\varepsilon(f_{\sharp}(X'))=\\
=r\left(f(Y_1')+f(Y_2')+\ldots+f(Y_k')\right)
=r\left(f(Y_1)+f(Y_2)+\ldots+f(Y_k)\right),
\end{multline*}
where $\varepsilon:C_0(X';\Bbb Z)\to\Bbb Z$ is the augmentation.
The last equality follows from Proposition~2.13, since the
function~$f$ is a cocycle and the sum~$Y_1+Y_2+\ldots+Y_k$ is a
cycle.

A local formula for the polynomial~$L_l(p_1,p_2,\ldots,p_l)$ is
unique up to a coboundary of the complex~$\Cal T^{*}(\Bbb Q)$. On
the other hand, cycles of the complex~$\Cal T_*$ are exactly
balanced sets of combinatorial spheres. Therefore a solution
$f\in\Cal T^{4l}(\Bbb Q)$ of the system of equations~$(*)$ is also
unique up to a coboundary of the complex~$\Cal T^{*}(\Bbb Q)$.
Thus solutions of the system~$(*)$ are exactly local formulae for
the polynomial~$L_l$.
\end{proof}

Since the right hand sides of the equations~$(*)$ admit a
straightforward combinatorial computation, the system of
equations~$(*)$ provides an explicit combinatorial description for
all local formulae for the $l$th Hirzebruch polynomial. Indeed,
the cubic cell combinatorial manifolds~$X$ is defined by the
explicit combinatorial construction described in section~2. The
signature of~$X$ can be computed combinatorially either by
definition or by an explicit (nonlocal) combinatorial formula
obtained by A.~Ranicki and D.~Sullivan in 1976~\cite{RaSu76}.
Since~$X$ is a cubic cell combinatorial manifold, we should use
the following modification of the Ranicki-Sullivan formula.
\begin{propos}
Suppose $X$ is an oriented $4l$-dimensional cubic cell
combinatorial manifold. By $C_i$ we denote the $i$th cellular
chain group of~$X$ with a fixed basis consisting of
$i$-dimensional cubes of~$X$. Then
$$
\Sign X=\Sign \left(\begin{array}{lc}
A&B\\
B^t &0
\end{array}\right),
$$
where $B$ is the matrix of the boundary
operator~$\partial:C_{2l+1}\to C_{2l}$ and $A$ is the matrix of
the symmetric bilinear form~$\alpha$ on~$C_{2l}$ such that
$$
\alpha(\sigma,\tau)=\sum\gamma(\sigma,\tau,\eta).
$$
Here the sum is taken over all $4l$-dimensional cubes~$\eta$
of~$X$, $\gamma(\sigma,\tau,\eta)=\pm 1$ if
$\dim(\sigma\cap\tau)=0$ and $\sigma,\tau\subset\eta$, and
$\gamma(\sigma,\tau,\eta)=0$ otherwise. In the first case
$\gamma(\sigma,\tau,\eta)=+1$ if and only if the product of the
orientation of~$\sigma$ with the orientation of~$\tau$ yields the
given orientation of~$X$.
\end{propos}

Unfortunately, the obtained explicit description of all local
formulae for Hirzebruch polynomials is very inefficient. The
construction of the manifold~$X$ is rather complicated. Hence we
need to compute the signatures of matrices of a very large order.
Therefore the described formulae can hardly be used for concrete
computations.

\subsection{Choice of a canonical formula}
Besides the description of all local formulae for a polynomial in
Pontryagin classes one always wants to construct explicitly a
single canonical local formula for this polynomial. To choose a
canonical local formula for the
polynomial~$L_l(p_1,p_2,\ldots,p_l)$ we need to choose a canonical
solution~$f_0$ of the system~$(*)$. This can be done using the
following standard trick.

In 1989 U.~Pachner~\cite{Pac87},~\cite{Pac91} proved that every
combinatorial sphere can be obtained from the boundary of a
simplex by a finite sequence of {\it bistellar moves} (see
also~\cite{BuPa04}). By~$T_{4l}^{(q)}$ we denote the set of
oriented $(4l-1)$-dimensional combinatorial spheres that can be
obtained from the boundary of a $4l$-dimensional simplex by a
sequence of not more than $q$ bistellar moves. Notice that the
set~$T_{4l}^{(q)}$ admits an algorithmic description. Now let us
consequently determine the restrictions of~$f_0$ to the
sets~$T_{4l}^{(q)}$. Assume that we have already determined the
restriction of~$f_0$ to the set~$T_{4l}^{(q-1)}$. Consider all
functions $f:T_{4l}^{(q)}\to \Bbb Q$ coinciding with $f_0$
on~$T_{4l}^{(q-1)}$ and satisfying the equations~$f(-Y)=-f(Y)$ and
equations~$(*)$ for all balanced sets of combinatorial spheres
$Y_1,Y_2,\ldots,Y_k\in T_{4l}^{(q)}$. Among them for the
restriction of $f_0$ we choose the function with the smallest
value
$$
\sum_{Y\in T_{4l}^{(q)}}\left(f(Y)\right)^2
$$
Such function exists, is unique, and can be found by solving a
system of linear equations. Thus this function can be computed by
an algorithm.
\begin{remark}
One cannot use the sets of combinatorial spheres with not more
than $q$ vertices instead of the sets~$T_{4l}^{(q)}$, since the
set of combinatorial spheres with not more than $q$ vertices
cannot be described by a finite algorithm for~$l$ and~$q$ large
enough. However, if we work with simplicial homology manifolds one
can replace the sets~$T_{4l}^{(q)}$ with the sets of simplicial
homology spheres with not more than $q$ vertices.
\end{remark}
\begin{remark}
We do not need the described procedure for the choice of a
canonical solution of the system~$(*)$ to obtain a simplicial
cycle whose homology class is the Poincar\'e dual of the
Hirzebruch $L$-class of a concrete combinatorial manifold~$K$.
Indeed, to obtain such cycle we need to know the values~$f(Y)$
only for those $(4l-1)$-dimensional combinatorial spheres~$Y$ that
appear as the links of simplices of~$K$. Hence we need to consider
only those equations~$(*)$ that correspond to balanced sets
$Y_1,Y_2,\ldots,Y_k$ such that each combinatorial sphere~$Y_i$ is
isomorphic to the link of some simplex of~$K$. Among these
equations there are only finitely many linearly independent
equations. Thus we obtain a finite system of linear equations and
we may choose an arbitrary solution of it.
\end{remark}
\subsection{Multiplication of local formulae}
By Corollary~3.4, we have an additive isomorphism
$$
\delta^*:H^*(\Cal T^*(\Bbb Q))\to\Hom(\Omega_*^{\SPL},\Bbb
Q)\cong\Bbb Q[p_1,p_2,\ldots].
$$
We naturally face a question whether it is possible to describe
combinatorially the multiplication obtained in cohomology of the
complex~$\Cal T^*(\Bbb Q)$ (see~\cite{Gai04}). In this section we
define combinatorially a multiplication of cocycles of~$\Cal
T^*(\Bbb Q)$ that induces the required multiplication in
cohomology. Unfortunately, this multiplication is neither
bilinear, nor associative, nor commutative. It does not satisfy
the Leibniz formula and most probably has no natural extension to
the whole complex~$\Cal T^*(\Bbb Q)$. Thus the following question
is still open.
\begin{question}
Does the complex~$\Cal T^*(\Bbb Q)$ admit a bilinear associative
multiplication satisfying the Leibniz formula and inducing those
multiplication in cohomology with respect to which the
isomorphism~$\delta^*$ is multiplicative?
\end{question}
Even if a multiplication of cocycles of~$\Cal T^*(\Bbb Q)$ is
neither bilinear nor associative its combinatorial definition
immediately allows us to construct explicitly a local formula for
the product of two polynomials in rational Pontryagin classes if
we are given local formulae for those two polynomials. Since we
already know explicit local formulae for the Hirzebruch
$L$-polynomials (see sections~5.1 and~5.2), we can now construct
explicitly local formulae for all polynomials in rational
Pontryagin classes. (Recall that the Hirzebruch polynomials
generate the ring~$\Bbb Q[p_1,p_2,\ldots]$.)

Cocycles of the complex~$\Cal T^*(\Bbb Q)$ provide local formulae
for cycles whose homology classes are dual to polynomials in
Pontryagin classes. However, if we wish to construct a local
formula for the product of polynomials from given local formulae
for the multipliers we are more convenient to work with cocycles
representing polynomials in Pontryagin classes. The first who
considered local formulae for cocycles representing characteristic
classes were N.~Levitt and C.~Rourke~\cite{LeRo78}. They
considered simplicial cocycles in the first barycentric
subdivision of a given combinatorial manifold such that the value
of a cocycle on every simplex is solely determined by the
combinatorial structure of the star of the minimal vertex of this
simplex. (Here the {\it minimal vertex} is the vertex that is the
barycenter of a simplex of a minimal dimension.) We are more
convenient to work with the canonical cubic subdivision of a
combinatorial manifold (see section~5.5) rather than with the
barycentric subdivision. In section~5.4 we define a cochain
complex~$\Cal W^*(A)$, which is an analogue of the  cochain
complex~$\Cal T^*(A)$ in our context. For each ring~$\Lambda$ we
endow the complex~$\Cal W^*(\Lambda)$ with an associative
multiplication satisfying the Leibniz formula. Our main result is
Theorem~5.1 providing an isomorphism~$H^*(\Cal T^*(\Bbb Q))\cong
H^*(\Cal W^*(\Bbb Q))$. This theorem allows us to use the
multiplication in~$\Cal W^*(\Bbb Q)$ to construct the required
multiplication of cocycles of the complex~$\Cal T^*(\Bbb Q)$ (see
section~5.7).

\subsection{Cochain complex $\mathcal{W}^*(A)$}
In section~2.1 we have noticed that there are two equivalent
definitions for the links of simplices of a simplicial complex. In
the rest of this papers, that is, in sections~5.4--5.8 we are
convenient to use the definition distinct from one used in the
previous sections of this paper. Thus in the sequel the link of a
simplex~$\sigma$ of a simplicial complex~$K$ is defined to be the
subcomplex of~$K$ consisting of all simplices~$\tau$ such that
$\sigma\cap\tau=\emptyset$ and $K$ contains a simplex including
both~$\sigma$ and~$\tau$.

For each Abelian group~$A$ we define a cochain complex~$\Cal
W^*(A)$ in the following way. Elements of the group~$\Cal W^n(A)$
are functions~$h$ that assign to every combinatorial sphere $Y$
(of an arbitrary dimension and without distinguished orientation)
a simplicial cochain~$h(Y)\in C^{n-1}(Y;A)$ such that for any
isomorphism~$i:Y_1\to Y_2$ the pullback of~$h(Y_2)$ coincides
with~$h(Y_1)$. In particular, the cochain~$h(Y)$ must be invariant
both under orientation-preserving and under orientation-reversing
automorphisms of~$Y$. The value of the cochain~$h(Y)$ on a
chain~$\xi\in C_{n-1}(Y;\Bbb Z)$ will be denoted by~$h(Y,\xi)$. As
before we often do not distinguish between a combinatorial sphere
and its isomorphism class. We define the mapping
$$
\delta:\Cal W^n(A)\to\Cal W^{n+1}(A)
$$
by
$$
(\delta h)(Y,\xi)=(-1)^nh(Y,\partial\xi)+(-1)^{n-1}\sum_{y\in
\V(Y) }h(\Lk y,\xi_y),
$$
where $\xi_y$ is the chain that contains each simplex~$\sigma$
of~$\Lk y$ with a coefficient equal to the coefficient of the
simplex~$y*\sigma$ in the chain~$\xi$.
\begin{propos}
$\delta^2=0$.
\end{propos}
\begin{proof}
We put,
$$
(\delta_1 h)(Y,\xi)=(-1)^nh(Y,\partial\xi),\qquad (\delta_2
h)(Y,\xi)=(-1)^{n-1}\sum_{y\in \V(Y)}h(\Lk y,\xi_y).
$$
Evidently, $\delta_1^2=0$. To show that~$\delta_2^2=0$ we notice
that $\left(\xi_x\right)_y=-\left(\xi_y\right)_x$ for every
chain~$\xi\in C_{n-1}(Y;\Bbb Z)$ and every two vertices~$x,y\in Y$
connected by an edge. To show that
$\delta_1\delta_2+\delta_2\delta_1=0$ we notice
that~$\partial(\xi_y)=-\left(\partial\xi\right)_y$ for every
chain~$\xi\in C_{n-1}(Y;\Bbb Z)$ and every vertex~$y\in Y$.
\end{proof}

Thus $\Cal W^*(A)$ is a cochain complex with
differential~$\delta$.

We define a homomorphism~$\alpha:\Cal W^n(A)\to \Cal T^n(A)$ by
$$
\alpha(h)(Y)=h(Y,[Y]),
$$
where $[Y]$ is the fundamental cycle of the combinatorial
sphere~$Y$. Since the cochain~$h(Y)$ is independent of the
orientation of~$Y$, we see that the number~$\alpha(h)(Y)$ reverses
its sign whenever we reverse the orientation of~$Y$. It is easy to
check that $\alpha$ is a chain homomorphism. Hence $\alpha$
induces a homomorphism
$$
\alpha^*:H^*(\Cal W^*(A))\to H^*(\Cal T^*(A)).
$$
\begin{theorem}
For $A=\Bbb Q$ the homomorphism $\alpha^*$ is an isomorphism.
\end{theorem}
The proof is postponed to section~5.8.

Suppose $Y$ is a combinatorial sphere, $\xi\in C_l(Y;\Bbb Z)$ is a
simplicial chain, $\tau$ is an oriented $(n-1)$-dimensional
simplex of~$Y$. By $\xi_{\tau}\in C_{l-n}(\Lk\tau;\Bbb Z)$ we
denote the chain containing every simplex~$\rho$ of the
complex~$\Lk\tau$ with a coefficient equal to the coefficient of
the simplex~$\tau*\rho$ in the chain~$\xi$.

Let $\Lambda$ be an associative ring. We define a multiplication
$$
\Cal W^n(\Lambda)\otimes\Cal W^k(\Lambda)\to\Cal W^{n+k}(\Lambda)
$$
by
$$
(h_1h_2)(Y,\xi)=(-1)^{nk}
\mathop{\sum\limits_{\tau\,\,\text{\textnormal{a simplex
of}}\,Y,}}\limits_{\dim\tau=n-1}
h_1(Y,\tau)h_2(\Lk\tau,\xi_{\tau}),
$$
where $\xi\in C_{n+k-1}(Y;\Bbb Z)$. (The
summand~$h_1(Y,\tau)h_2(\Lk\tau,\xi_{\tau})$ is independent of the
orientation of~$\tau$.) It can be immediately checked that this
multiplication is associative and
$$
\delta(h_1h_2)=(\delta h_1)h_2+(-1)^nh_1\delta  h_2
$$
for every $h_1\in\Cal W^n(\Lambda)$, $h_2\in\Cal W^k(\Lambda)$.
Hence the constructed multiplication induces an associative
multiplication in cohomology of the complex~$\Cal W^*(\Lambda)$.

\subsection{The canonical cubic subdivision of a simplicial
complex}

For each simplicial complex~$K$ one can define its canonical cubic
subdivision~$\cub (K)$ (see, for example,~\cite{BuPa04}). Let us
describe this construction. Assume that the complex~$K$ has $q$
vertices and identify the vertex set of~$K$ with the set~$\{
1,2,\ldots,q\}$. Let $I^q=[0,1]^q$ be the standard cube. For any
nonempty simplices~$\tau\subset\sigma$ of the complex~$K$ we put
$$
C_{\tau\subset\sigma}=\left\{ (y_1,y_2,\ldots,y_q)\in
I^q:y_j=0\text{ for }j\in\tau,y_j=1\text{ for
}j\notin\sigma\right\}.
$$
Then $C_{\tau\subset\sigma}$ is a closed
$(\dim\sigma-\dim\tau)$-dimensional face of the cube~$I^q$. Let
$i:K\to I^q$ be the mapping such that $i$ takes the barycenter of
each simplex~$\sigma$ of~$K$ to the
vertex~$C_{\sigma\subset\sigma}$ and the restriction of~$i$ to
every simplex of the first barycentric subdivision of~$K$ is
linear. Then $i$ is an embedding whose image coincides with the
union of all faces~$C_{\tau\subset\sigma}$, where
$\tau\subset\sigma$ are simplices of~$K$. The preimages of
faces~$C_{\tau\subset\sigma}$ under the mapping~$i$ form the cubic
subdivision~$\cub(K)$ of~$K$. (This preimages will also be denoted
by~$C_{\tau\subset\sigma}$.)

Whenever the simplices~$\tau$ and~$\sigma$ are oriented, the cell~
$C_{\tau\subset\sigma}$ is endowed with the orientation such that
the product of the orientation of~$\tau$ and the orientation
of~$C_{\tau\subset\sigma}$ yields the orientation of~$\sigma$. We
define a multiplication in the cellular cochain
complex~$C^*(\cub(K);\Lambda)$ of the decomposition~$\cub(K)$ by
$$
(ab)\left(C_{\tau\subset\sigma}\right)=(-1)^{nk}
\mathop{\sum\limits_{\rho\,\,\text{\textnormal{a simplex
of}}\,K,}}\limits_{\dim\rho=\dim\tau+n}
a\left(C_{\tau\subset\rho}\right)
b\left(C_{\rho\subset\sigma}\right),
$$
where $a\in C^n(\cub(K);\Lambda)$, $b\in C^k(\cub(K);\Lambda)$,
and $\sigma$ and $\tau$ are oriented simplices of~$K$ such that
$\tau\subset\sigma$ and $\dim\sigma-\dim\tau=n+k$.

\begin{propos}
This multiplication is associative, satisfies the Leibniz formula,
and induces the standard multiplication in the cohomology of~$K$.
\end{propos}

\subsection{Local formulae for cocycles}
Suppose $K$ is a combinatorial manifold, $h$ is an arbitrary
element of~$\Cal W^n(A)$. We define a cochain
$$
h^{\sharp}(K)\in C^n(\cub(K);A)
$$
by
$$
h^{\sharp}(K)(C_{\tau\subset\sigma})=h(\Lk\tau,\sigma_{\tau}).
$$
The following propositions can be checked immediately.

\begin{propos}
$\delta\left(h^{\sharp}(K)\right)=\left(\delta
h\right)^{\sharp}(K)$.
\end{propos}
\begin{corr}
The cochain~$h^{\sharp}(K)$ is a cocycle for every combinatorial
manifold~$K$ if and only if~$h$ is a cocycle of the complex~$\Cal
W^*(A)$. If $h$ is a coboundary of~$\Cal W^*(A)$, then the
cochain~$h^{\sharp}(K)$ is a coboundary for every combinatorial
manifold~$K$.
\end{corr}
\begin{propos}
Suppose $\Lambda$ is a ring, $h_1,h_2\in\Cal W^*(\Lambda)$. Then
$$
(h_1h_2)^{\sharp}(K)=h_1^{\sharp}(K)h_2^{\sharp}(K)
$$
for every combinatorial manifold~$K$.
\end{propos}
\begin{propos}
Suppose $h\in\Cal W^n(A)$ is a cocycle, $f=\alpha(h)$. Then for
any oriented combinatorial manifold~$K$ the homology class
represented by the cycle~$f_{\sharp}(K)$ is the Poincar\'e dual of
the cohomology class represented by the cocycle~$h^{\sharp}(K)$.
\end{propos}
\begin{proof}
By~$m$ we denote the dimension of~$K$. Let $\tau$ be a simplex
of~$K$ of dimension~$m-n$. The decomposition~$\cub(K)$ is a
subdivision of the decomposition~$K^*$ dual to the
triangulation~$K$. The $n$-dimensional cell $D\tau$ dual to the
simplex~$\tau$ is the union of the cubes~$C_{\tau\subset\sigma}$,
where $\sigma$ runs over all $m$-dimensional simplices
containing~$\tau$. Let $c$ be the image of the
cochain~$h^{\sharp}(K)$ under the natural
homomorphism~$C^n(\cub(K);A)\to C^n(K^*;A)$. Then
$$
c(D\tau)=\mathop{\sum_{\sigma\supset\tau,}}_{\dim\sigma=m}
h^{\sharp}(K)(C_{\tau\subset\sigma})=
h(\Lk\tau,[\Lk\tau])=f(\Lk\tau).
$$
Hence the cocycle~$c$ is the Poincar\'e dual of the
cycle~$f_{\sharp}(K)$.
\end{proof}

Let us consider the isomorphisms
$$
\begin{CD}
H^*(\Cal W^*(\Bbb Q)) @>{\alpha^*}>> H^*(\Cal T^*(\Bbb Q))
@>{\delta^*}>> \Hom(\Omega^{\SPL}_*,\Bbb Q)=\Bbb
Q[p_1,p_2,\ldots].
\end{CD}
$$
Suppose $h\in\Cal W^*(\Bbb Q)$ is a cocycle,
$F(p_1,p_2,\ldots)=\delta^*\alpha^*[h]$ is the corresponding
polynomial in rational Pontryagin classes. Proposition~5.8
immediately implies that the cocycle~$h^{\sharp}(K)$ represents
the cohomology class~$F(p_1(K),p_2(K),\ldots)$ for every
combinatorial manifold~$K$. In particular, the cohomology
class~$\left[h^{\sharp}(K)\right]$ is independent of the choice of
a cocycle~$h$ representing a given cohomology class and of the
choice of a piecewise-linear triangulation~$K$ of a given
manifold. We shall say that such cocycle~$h$ is a {\it local
formula} for the polynomial~$F$ in Pontryagin classes. It follows
from Theorem~5.1 that for every polynomial in rational Pontryagin
classes there is a local formula~$h\in\Cal W^*(\Bbb Q)$ unique up
to a coboundary of~$\Cal W^*(\Bbb Q)$.

\subsection{Multiplication of cocycles of~$\Cal T^*(\Bbb Q)$}
By~$Z^n$ we denote the subgroup of the group~$\Cal T^n(\Bbb Q)$
consisting of all cocycles, that is, of all local formulae for
polynomials in Pontryagin classes (see Proposition~5.1). We shall
construct explicitly a mapping~$\gamma:Z^n\to\Cal W^n(\Bbb Q)$
such that~$\gamma(f)$ is a cocycle for every local formula~$f$ and
$\alpha\circ\gamma$ is the identity homomorphism.

Suppose $f\in Z^n$ is a cocycle. The cochain~$h=\gamma(f)$ must
satisfy the conditions~$\delta h=0$ and~$\alpha(h)=f$, that is,
$h$ must be a solution of the system of equations of the following
two types.

1) $h(Y,\partial\xi)=\sum\limits_{y\in \V(Y)}h(\Lk y,\xi_y)$ for
every combinatorial sphere~$Y$ and every chain $\xi\in
C_{n}(Y;\Bbb Z)$;

2) $h(Y,[Y])=f(Y)$ for every oriented $(n-1)$-dimensional
combinatorial sphere~$Y$.

We shall construct the set of cochains~$h(Y)$ satisfying this
system of equations by the induction on the dimension of~$Y$.

For any oriented $(n-1)$-dimensional combinatorial sphere~$Y$ with
$q$ maximal-dimensional simplices we put
$$
h(Y,\sigma)=\frac{f(Y)}{q}
$$
for every positively oriented $(n-1)$-dimensional simplex~$\sigma$
of~$Y$. Obviously, the cochain~$h(Y)$ is independent of the
orientation of~$Y$.

Suppose $m\geqslant n$. Assume that the cochains~$h(Y)$ are
already defined for all combinatorial spheres~$Y$ of dimension
less than~$m$. Let us determine the cochain~$h(Y)$ for
an~$m$-dimensional combinatorial sphere~$Y$.
Let~$\sigma_1,\sigma_2,\ldots,\sigma_q$ be all~$(n-1)$-dimensional
simplices of~$Y$. We fix an arbitrary orientation of every
simplex~$\sigma_i$. We consider all equations
$$
h(Y,\partial\xi)=\sum_{y\in \V(Y)}h(\Lk y,\xi_y),\eqno(**)
$$
where $\xi\in C_{n}(Y;\Bbb Z)$. Each of these equations can be
regarded as a linear equation in variables
$h(Y,\sigma_1),h(Y,\sigma_2),\ldots,h(Y,\sigma_q)$.
\begin{propos}
The system of equations~$(**)$ is compatible.
\end{propos}
\begin{proof}
We need to prove only that
$$\sum_{y\in \V(Y)}h(\Lk y,\xi_y)=0$$
if $\xi$ is a cycle. We consider two cases.

If $m=n$, then $\xi=l[Y]$ for a certain~$l\in\Bbb Z$. Then
$$
\sum_{y\in \V(Y)}h(\Lk y,\xi_y)=l\sum_{y\in \V(Y)}f(\Lk
y)=(-1)^nl(\delta f)(Y)=0.
$$

If $m>n$, then $\xi=\partial\eta$ for a certain chain~$\eta\in
C_{n+1}(Y;\Bbb Z)$. Therefore,
\begin{multline*}
\sum_{y\in \V(Y)}h(\Lk y,\xi_y)=\sum_{y\in \V(Y)}h(\Lk
y,(\partial\eta)_y)=
-\sum_{y\in \V(Y)}h(\Lk y,\partial(\eta_y))=\\
=\sum_{(x,y)}h\left(\Lk( x*y),(\eta_y)_x\right)=0,
\end{multline*}
where the last sum is taken over all pairs of vertices~$(x,y)$
connected by an edge in the complex~$Y$. The last equality holds,
since~$(\eta_x)_y=-(\eta_y)_x$.
\end{proof}
For the cochain~$h(Y)$ we choose that solution of the system of
equations~$(**)$ for which the sum
$$
h(Y,\sigma_1)^2+h(Y,\sigma_2)^2+\cdots+h(Y,\sigma_q)^2
$$
takes its minimal value. Such solution exists, is unique, is
rational, is invariant under all automorphisms of~$Y$, and its
determination reduces to the solution of a system of linear
equation.

Notice that to compute the cochain~$h(Y)$ for an~$m$-dimensional
combinatorial sphere~$Y$ we need to know only the values of the
function~$f$ on the links of all~$(m-n)$-dimensional simplices
of~$Y$.

We define a multiplication
$$
Z^n\times Z^k\to Z^{n+k}
$$
by
$$
f_1\diamond f_2=\alpha(\gamma(f_1)\gamma(f_2)).
$$
This multiplication is neither associative, nor commutative, nor
linear in the first argument. It can be checked that this
multiplication is linear in the second argument. Propositions~5.7
and~5.8 immediately imply that if $f_1$ and $f_2$ are local
formulae for polynomials~$F_1(p_1,p_2,\ldots)$
and~$F_2(p_1,p_2,\ldots)$ respectively, then $f_1\diamond f_2$ is
a local formula for the
polynomial~$F_1(p_1,p_2,\ldots)F_2(p_1,p_2,\ldots)$. Thus the
multiplication~$\diamond$ induces the multiplication in cohomology
of~$\Cal T^*(\Bbb Q)$ coinciding with the multiplication defined
by the isomorphism~$\delta^*$.

\begin{remark}
To compute the value of the function~$f_1\diamond f_2$ on an
$(n+k-1)$-dimensional combinatorial sphere~$Y$ we suffice to know
only the values of~$f_1$ on the links of all $(k-1)$-dimensional
simplices of~$Y$ and the values of~$f_2$ on the links of all
$(n-1)$-dimensional simplices of~$Y$. Besides, the procedure for
computing the values~$(f_1\diamond f_2)(Y)$ reduces to the
solution of systems of linear equations. Thus, by Remark~5.2, to
compute a cycle whose homology class is dual to a given polynomial
in the rational Pontryagin classes of a given combinatorial
manifold~$K$ one does not need to use the described in section~5.2
procedure for the choice of canonical local formulae for the
Hirzebruch $L$-classes. Actually, in the process of computation
one needs to operate only with those combinatorial spheres that
appear as the links of simplices of~$K$.
\end{remark}

\subsection{Proof of Theorem 5.1}
The homomorphism~$\alpha^*$ is an epimorphism since for any
cocycle~$f\in\Cal T^n(\Bbb Q)$ there is the cocycle
$h=\gamma(f)\in\Cal W^n(\Bbb Q)$ such that~$\alpha(h)=f$.

Suppose $h\in\Cal W^n(\Bbb Q)$ is a cocycle, $\alpha(h)=\delta f$.
Then $\alpha(h-\delta\gamma(f))=0$. Hence to prove that $\alpha^*$
is a monomorphism we suffice to prove that for each~$h$ such
that~$\alpha(h)=0$ there is an element~$g\in\Cal W^{n-1}(\Bbb Q)$
such that $h=\delta g$. The condition~$h=\delta g$ can be written
as the system of equations
$$
g(Y,\partial\xi)=\sum_{y\in \V(Y)}g(\Lk
y,\xi_y)+(-1)^{n-1}h(Y,\xi),\eqno(*\!*\!*)
$$
where $Y$ is an arbitrary combinatorial sphere and $\xi\in
C_{n-1}(Y;\Bbb Z)$ is an arbitrary chain.

Let us consequently define the cochains~$g(Y)$. For every
$(n-2)$-dimensional combinatorial sphere~$Y$ we put~$g(Y)=0$.

Suppose that $m\geqslant n-1$. Assume that the cochains~$g(Y)$ are
already defined for all combinatorial spheres of dimension less
than~$m$. Let us define a cochain~$g(Y)$ for an $m$-dimensional
combinatorial sphere~$Y$. Let~$\sigma_1,\sigma_2,\ldots,\sigma_q$
be all~$(n-2)$-dimensional simplices of~$Y$. We fix an arbitrary
orientation of every simplex~$\sigma_i$. We consider all equations
$(*\!*\!*)$, where $\xi\in C_{n-1}(Y;\Bbb Z)$. Each of these
equations can be regarded as a linear equation in variables
$g(Y,\sigma_1),g(Y,\sigma_2),\ldots,g(Y,\sigma_q)$.
\begin{propos}
The considered system of equations is compatible.
\end{propos}
The proof is completely similar to the proof of Proposition~5.9.
For the cochain~$g(Y)$ we choose an arbitrary
(automorphism-invariant) solution of the considered system of
linear equations.

\end{document}